\documentclass[11 pt, reqno]{amsart}

\usepackage{amsmath,amsfonts,amssymb,graphicx}
\usepackage[colorlinks=true,hyperindex=true]{hyperref}
\usepackage{cancel,bbm}

\usepackage{comment}
\usepackage{enumitem}
\usepackage{cite}

\usepackage{tikz}

\counterwithin*{equation}{section}

\usepackage[left=1in, right=1in, top=1in, bottom= 1in]{geometry}
\usepackage{color}
 \usepackage{fancyhdr}
\usepackage{latexsym}

\usepackage{bm}

\newtheorem{ccounter}{ccounter}[section]
\newtheorem{thm}[ccounter]{Theorem}

\newtheorem{lem}[ccounter]{Lemma}
\newtheorem{cor}[ccounter]{Corollary}
\newtheorem{defn}[ccounter]{Definition}
\newtheorem{prop}[ccounter]{Proposition}
\newtheorem{ass}[ccounter]{Assumption}
\newtheorem{ex}[ccounter]{Example}

\def\bet{\begin{thm}}
\def\eet{\end{thm}}
\def\bel{\begin{lem}}
\def\eel{\end{lem}}
\def\bas{\begin{ass}}
\def\eas{\end{ass}}
\def\bec{\begin{cor}}
\def\eec{\end{cor}}
\def\bed{\begin{defn}}
\def\eed{\end{defn}}
\def\bep{\begin{prop}}
\def\eep{\end{prop}}
\def\beq{\begin{equation}}
\def\eeq{\end{equation}}
\def\bea{\begin{equation*}}
\def\eea{\end{equation*}}

\def\bex{\begin{ex}}
\def\eex{\end{ex}}

\def\1{\boldsymbol{1}}

\def\eps{\varepsilon}
\renewcommand\leq\varleq
\renewcommand\geq\vargeq

\def\F{\mathcal{F}}

\def\Cov{\mathrm{Cov}}

\def\Z{\mathbb{Z}}
\def\T{\mathbb{T}}

\renewcommand{\P}{\mathbb{P}}
\newcommand{\eqd}{\,{\buildrel d \over =}\,}

\newcommand{\V}{\bm{V}}
\newcommand{\slope}{\mathrm{slope}}
\newcommand{\SFE}{\mathrm{SFE}}
\newcommand{\Ent}{\mathrm{Ent}}

\usepackage{enumitem} 

\title[Rigidity of sloped height functions in $d\ge 3$]{On the rigidity of sloped height functions in $d\ge 3$ \\ and non-crossing surfaces}

\author{Arka Adhikari}
\address{Department of Mathematics, University of Maryland. College Park, MD. \\ \url{arkaa@umd.edu}}

\author{Reza Gheissari}
\address{Department of Mathematics, Northwestern University. Evanston, IL. \\ \url{gheissari@northwestern.edu}}

\author{Ron Peled}
\address{Department of Mathematics, University of Maryland. College Park, MD. \\ \url{peledron@umd.edu}}

\begin{document}

\begin{abstract}
    We consider integer-valued $\nabla\phi$ height functions, with general  convex interactions, placed on a slope. Sheffield (2003) conjectured that for all slopes, such height functions are localized in dimensions $d\ge3$, in the sense of having tight fluctuations in finite volume, and admitting infinite-volume limits. We establish this conjecture when there are two coordinates on which the slope vector is zero and on which the interactions are even and low temperature. In this setting, we also prove the uniqueness, in the appropriate sense, of the infinite-volume limit and classify its extremal components. 
    
    We further study the structural properties of the infinite-volume limit. We establish the entropic repulsion between macroscopic domain walls (boundaries of level sets), showing that they are maximally separated in a precise sense described by a rotation-of-the-circle dynamical system. Lastly, we show exponential decay of correlations in the extremal components.

    Our setup includes, as a special case, infinitely many zero-slope integer-valued surfaces (of dimension two or higher) conditioned not to cross. We deduce the existence and structural properties of the bulk Gibbs measure over such non-crossing surfaces with any given average spacing. 
\end{abstract}

\maketitle

\vspace{-.8cm}
\section{Introduction}
Our goal in this paper is to study \emph{sloped} states for integer-valued height functions ($\nabla\phi$ models) with translation-invariant convex interactions. These height functions are the prototypical models in the statistical physics of random surfaces~\cite{bricmont1986random,Sheffield-Random-Surfaces,Velenik2006}, modeling a competition between a smoothing interaction and fluctuations induced by temperature and boundary conditions. They model a variety of physical phenomena including equilibrium crystal shapes and spin system interfaces.

We consider the following setup. Fix a dimension $d\ge 1$. Let $\V := (V_i)_{1\le i\le d}$ be a collection of functions $V_i:\Z\to\mathbb R \cup \{\infty\}$, termed the (gradient) \emph{interactions}. We assume that
\begin{equation}\label{eq:potential assumption}\tag{Convexity}
    \text{$V_i$ is convex and $\lim_{k\to-\infty}V_i(k)=\lim_{k\to\infty} V_i(k)=\infty$ for each $1\le i\le d$.}
\end{equation}
Here, $V$ is \emph{convex} if $\{k\colon V(k)<\infty\}$ is a non-empty (discrete, finite or infinite) interval and $V(k+1)+V(k-1)-2V(k)\ge 0$ for all $k$ in this interval. We sometimes further require an interaction $V$ to be \emph{even}, i.e., to satisfy $V(k) = V(-k)$ for all $k\in\Z$. In addition, we say that $\V$ is \emph{isotropic} if the $V_i$ are identical and even.

The \emph{integer-valued height function model with interactions $\V$} is the model on $\Omega:=\{\phi:\Z^d\to\Z\}$
with formal Hamiltonian
\begin{equation}\label{eq:height function model}
    H(\phi):= \sum_{(u,v)\in\vec{E}(\Z^d)}V_{uv}(\phi_v - \phi_u),
\end{equation}
where $\vec{E}(\Z^d):=\{(u,v)\colon u\in\Z^d, v = u+\mathfrak{e}_i\text{ for some $1\le i\le d$}\}$, with $\mathfrak{e}_i$ the $i$'th standard basis vector, are the edges of $\Z^d$ oriented in the positive coordinate directions and where we set $V_{uv} := V_i$ when $v = u+ \mathfrak{e}_i$. Our convention is to absorb the temperature parameter into the interactions $\V$.

The following examples serve to illustrate the usefulness of the generality of our setup:
\begin{enumerate}[label=(\alph*)]
    \item \label{item:nabla-phi} $|\nabla \phi|^p$ models: The isotropic model with $V_i(k)=V(k) = \beta|k|^p$ for some $p\ge 1$ and inverse temperature $\beta>0$. The $p=1$ case is termed the \emph{Solid-On-Solid (SOS) model} while the $p=2$ case is known as the \emph{integer-valued Gaussian free field}.
    \item \label{item:Lipschitz} Lipschitz height functions: The model is said to be \emph{Lipschitz} if, for each $i$, $V_i(k)=\infty$ outside of a finite interval. For instance, the isotropic model with $V(0)=0$, $V(\pm 1)=\beta$ and $V(\pm k)=\infty$ for $|k|>1$ is sometimes called the \emph{$1$-Lipschitz model}.
    \item \label{item:non-isotropic} Non-isotropic coupling constants: One may fix $V:\Z\to(-\infty,\infty]$ satisfying~\eqref{eq:potential assumption} and let $V_i := \beta_i V$ for some positive $(\beta_i)_{1\le i\le d}$ ($\beta_i=0$ is allowed in the Lipschitz case).
    \item \label{item:non-crossing-surfaces} Non-crossing surfaces \cite[Section 4.6]{Sheffield-Random-Surfaces}: Suppose $V_d(k) = \begin{cases}0&k\ge 0\\ \infty&k<0\end{cases}$.\footnote{To satisfy the requirement $V_d(k)\xrightarrow[k\to\infty]{}\infty$ of \eqref{eq:potential assumption}, we note that this choice of $V_d$ is equivalent to $V_d(k)= \alpha k$ for $k\ge 0$, for any $\alpha>0$, and $V_d(k) = \infty$ for $k<0$. Indeed, for any fixed boundary conditions on a finite domain, all these interaction choices give exactly the same finite-volume Gibbs measure.} The single height function $\phi:\Z^d\to\Z$ may then be regarded as a family of height functions $\psi^{(m)}:\Z^{d-1}\to\Z$ indexed by $m\in\Z$, defined by $\psi^{(v_d)}(v_1,\ldots, v_{d-1}):=\phi(v_1,\ldots, v_d)$. The interaction $V_d$ ensures that the surfaces are non-crossing in the sense that $\psi^{(m+1)}\ge \psi^{(m)}$ for all $m$, almost surely. The model may thus be interpreted as a model for infinitely many paths (for $d=2$) or surfaces (for $d\ge 3$) conditioned to stay ordered. See also Figure~\ref{fig:non-crossing-surfaces}. 
\end{enumerate}
The above definitions also make sense for real-valued height functions, but our focus in this paper is on the integer-valued case. Figure~\ref{fig:3d-sloped-surfaces} depicts a $d=3$ SOS model at $(0,0,1)$ slope.

\begin{figure}
\centering
\begin{tikzpicture}
        \node at (-4,-.25){\includegraphics[width = 0.35\textwidth]{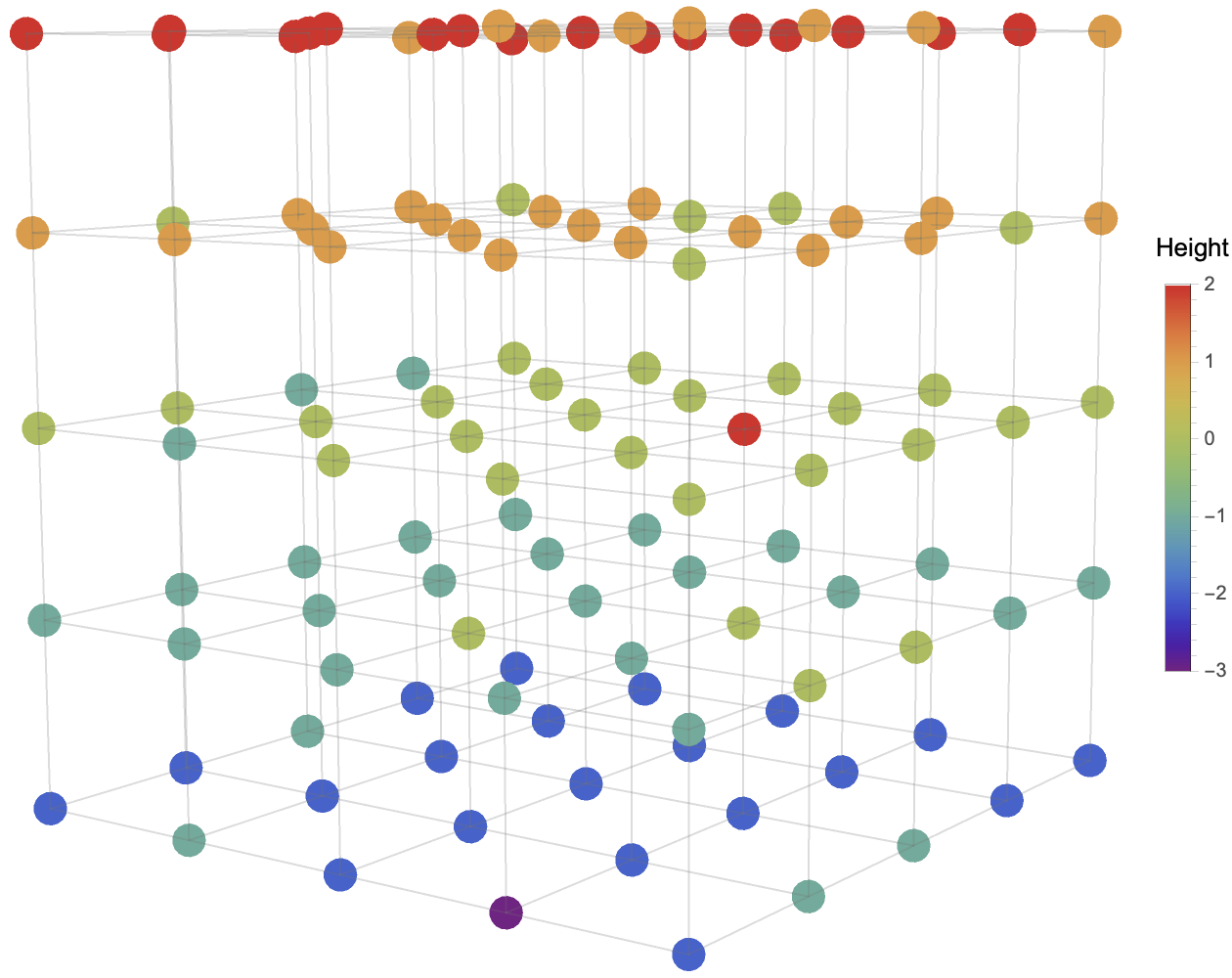}};
        \node at (4,0){
  \includegraphics[width = 0.42\textwidth]{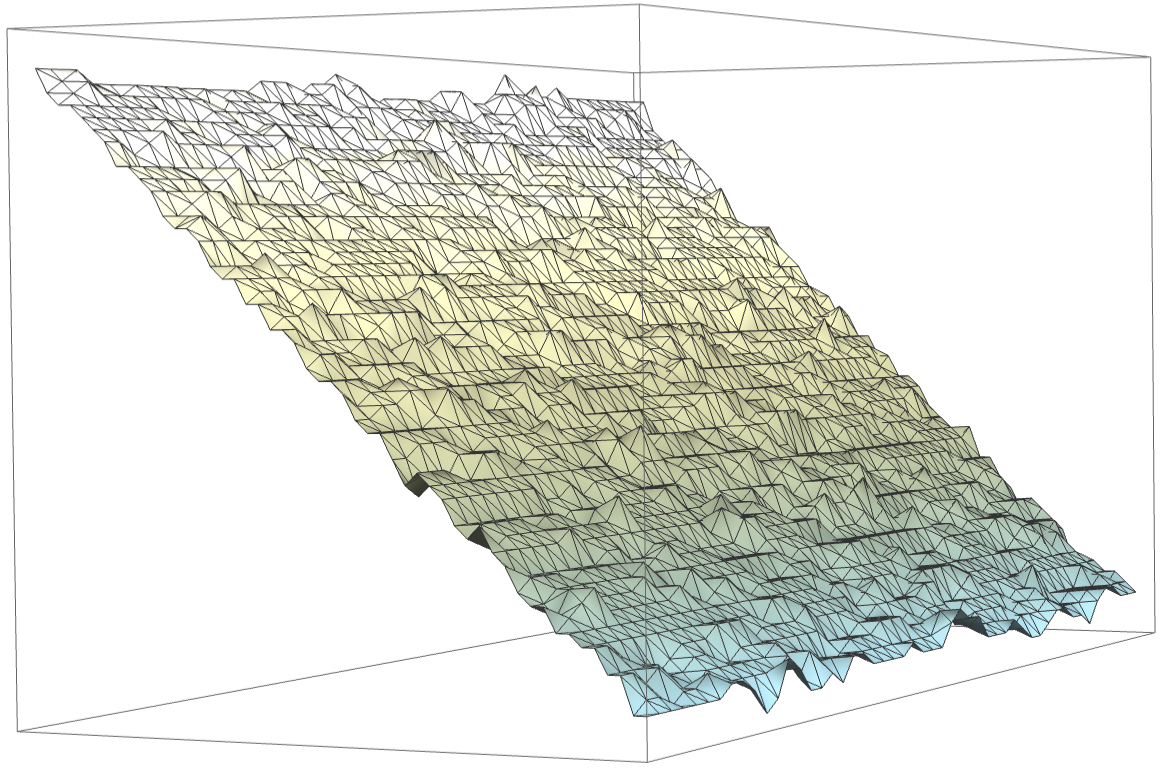}};
  \end{tikzpicture}
    \caption{Left: A snapshot from the bulk of a slope-$(0,0,1)$ 3D SOS (i.e., $
    |\nabla \phi|$) model with the colors indicating the height. Right: The $x_1 = 0$ cross-section from a sample of the same SOS model, now depicted as a surface with heights being given by the location on the $z$-axis (and also still depicted by the color scheme).}\label{fig:3d-sloped-surfaces}
\end{figure}

In this paper we study the \emph{existence and uniqueness problems} for Gibbs measures on height functions of a given slope. We elaborate on several perspectives on the problem in Section~\ref{subsec:additional-motivations}, but briefly note here that such sloped states are the building blocks for the Wulff construction and associated variational problems, and also that their study necessitates the development of new tools for dealing with the ``gas'' of macroscopic level surfaces that the slope induces.

At \emph{zero slope}, the existence problem received significant attention in the following form: Consider the height function model on a sequence of domains $(\Lambda_L)_L$ with zero boundary values. Is there a limiting distribution to the heights as $\Lambda_L$ increases to $\Z^d$ (maybe along a subsequence)? Is there a limiting distribution to the gradients? While one expects the gradient distribution to generically have subsequential limits, the question for the heights is more subtle: the heights are said to be \emph{localized} if a subsequential limit exists, and otherwise \emph{delocalized}. It is a standard fact that the heights are delocalized in dimension $d=1$ (in non-degenerate cases). In dimension $d=2$, for many choices of interactions, the model undergoes a \emph{roughening transition}~\cite{FrohlichSpencer}---the heights are localized at low temperatures and delocalized at high temperatures. Finally, in dimensions $d\ge 3$, it is generally believed that the heights at zero slope are localized, with rigorous proofs available in several cases (see Section~\ref{subsec:related-work}).

For non-zero slopes, the existence problem is obtained by replacing the zero boundary values with values that approximate a hyperplane of that slope. In this context, Sheffield~\cite[Section 10.2.2]{Sheffield-Random-Surfaces} conjectured that the heights are localized for all slopes in dimensions $d\ge 3$. Standard techniques for showing localization, such as cluster expansions, Peierls-type arguments, or correlation inequalities break down in the presence of non-zero slope, necessitating the development of new methods (see also Section~\ref{subsec:proof-ideas}).

If a slope is shown to be localized, the next fundamental question is whether and in what sense the Gibbs measure of that slope is unique. 
As global integer shifts of a height function clearly preserve the Gibbs condition, the natural uniqueness property is for the family of ergodic gradient Gibbs measures of the given slope.\footnote{For technical reasons, one needs to further restrict the uniqueness question to be among ergodic gradient Gibbs measures of \emph{minimal free energy} (see Section~\ref{sec:background gradient Gibbs measures}).}
(Note that there may generically be many extremal components to each such measure, which we will also be interested in.)
When the slope has an irrational coordinate or when $d=2$, uniqueness was established in~\cite{Sheffield-Random-Surfaces}, but it was left open 
 in the general case.\footnote{\cite[Section 8.8]{Sheffield-Random-Surfaces} also gave a counterexample to uniqueness for a $2\mathbb Z^3$-translation-invariant interaction in $d= 3$.}

\begin{figure}
\begin{tikzpicture}[xscale=0.8, yscale=1.2]
\node[] at (0,0) {
    \includegraphics[width= 0.75\textwidth]{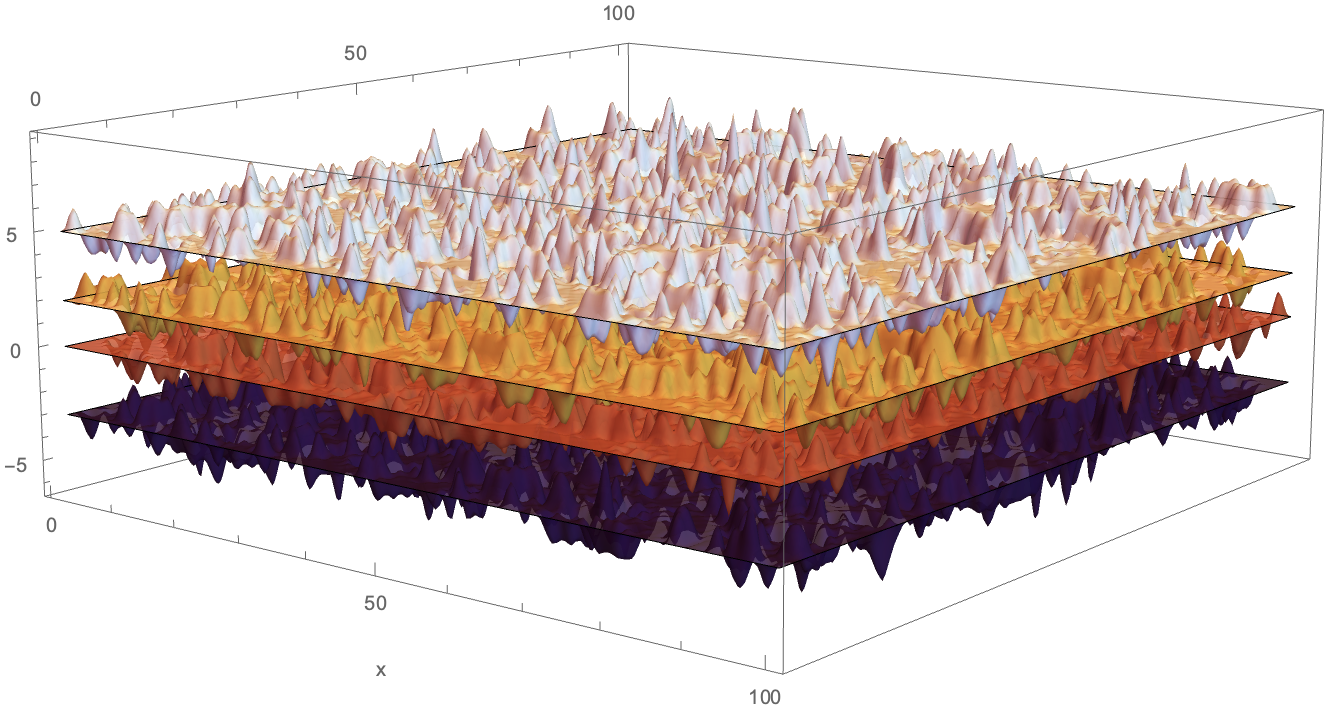}};
    \node[font = \Large] at (0,2) {$\boldsymbol{\vdots}$};   
    \node[font = \Large] at (0,-1.6) {$\boldsymbol{\vdots}$};

\end{tikzpicture}
    \caption{A sample from the Gibbs measure on non-crossing low-temperature zero-slope SOS surfaces with average spacing $5/2$. As explained in Corollary~\ref{cor:non-crossing-surfaces}, the surfaces pick out one of two ``dominant height" sequences, with spacings between consecutive dominant heights alternating as $(...,2,3,2,3,...)$ or as $(...,3,2,3,2,...)$, and each surface localizes about its own dominant height.}\label{fig:non-crossing-surfaces}
\end{figure}

Our first result (Theorem~\ref{thm:main-localization}) is to show that both existence and uniqueness hold under the condition that there are (at least) two coordinates on which the slope is zero and the interactions in those two coordinates are even and ``low temperature''.
More precisely, denoting the slope by $\mathfrak{s}$, we assume for a $\beta_0(d)$ sufficiently large, 
\begin{equation}\label{eq:main assumption}\tag{2-coord}
\begin{gathered}
\text{there exist two coordinates $1\le i_0<i_1\le d$ such that, for each $j\in\{i_0,i_1\}$,}\\
\text{$\mathfrak{s}_j=0$, $V_j$ is even and $V_j(1) - V_j(0)\ge \beta_0$}.
\end{gathered}
\end{equation}
(Note that no assumption is made on the interactions in the other coordinates.) For visualization, as in Figure~\ref{fig:3d-sloped-surfaces}, it may help the reader to think about $i_0=1, i_1 =2$, and partition $\mathbb Z^d$ into its \emph{layer} planes $\mathbb Z^2\times (v_3,...,v_d)$ within each of which the slope is zero. 
To our knowledge, this is the first general setting in which Sheffield's localization conjecture for non-zero slopes has been established (exempting a few special choices of pairs of slopes and interactions---see Section~\ref{subsec:related-work}). In particular, for $\mathfrak{s}\notin \mathbb Z^d$, the localization is already new for both the SOS and integer-valued Gaussian free field. 

Once existence and uniqueness are settled, we wish to obtain further structural information on the sampled height function. For slopes $\mathfrak{s}\notin\Z^d$, there are many height functions of minimal energy, distinguished by the locations of their macroscopic domain walls. We show that in samples from the ergodic gradient Gibbs measure of slope $\mathfrak{s}$, the domain walls become maximally separated (due to entropic repulsion) in a precise manner determined by a rotation-of-the-circle dynamical system. Moreover, the different ways of obtaining such maximal separation are in one-to-one correspondence with the extremal components of the measure.\footnote{The fact that the ergodic gradient Gibbs measure is not extremal when $\mathfrak{s}\notin\Z^d$ was already shown in~\cite[Section 8.7]{Sheffield-Random-Surfaces}, but we give a new description of the extremal components in terms of the macroscopic domain walls.} Lastly, we prove that there is exponential decay of correlations in these extremal components.

Our proofs rely on a notion of restricted reflection positivity, which we develop following Shlosman--Vignaud~\cite{ShlosmanVignaud} who introduced it to prove rigidity of slope zero interfaces in certain spin systems,
 together with a variety of probabilistic and ergodic-theoretic arguments. After establishing rigidity on finite domains with ``sloped-periodic" boundary conditions, the proofs identifying the domain wall locations in extremal components also make essential use of Sheffield's theory of gradient Gibbs measures for random surfaces~\cite{Sheffield-Random-Surfaces}. We have included in Section~\ref{sec:background gradient Gibbs measures} a brief summary of the results from~\cite{Sheffield-Random-Surfaces} that we use, that readers seeking a quick introduction to that theory may find useful. We outline our proof ideas in Section~\ref{subsec:proof-ideas}. 

\subsection{Further perspectives}\label{subsec:additional-motivations}
Before detailing our results, we mention a few further reasons for the study of sloped states:

\smallskip
\noindent (1) Analogously to the famous Wulff construction for the Ising model~\cite{DKS,CerfPisztora}, Sheffield~\cite[Theorem 7.3.1]{Sheffield-Random-Surfaces} proved that the macroscopic profile of the $\nabla\phi$ model with general continuous boundary conditions is, informally, attained by patching together sloped states in a way that matches the boundary conditions and minimizes the overall specific free energy. 
Sheffield's result extends landmark results on the limit shapes and large deviation principles for the dimer model in~\cite{CohnKenyonPropp,kenyon2006dimers}, and analogous principles for real-valued surfaces~\cite{DeuschelGiacominIoffe}.
The (affine) sloped states can thus be viewed as the ``basic building blocks'' for the macroscopic behavior of these models. Furthermore, their localization properties are closely related to cusps in the corresponding surface tension function~\cite{MiracleSoleStepFreeEnergy}.

\smallskip \noindent 
(2) A height function with non-zero slope necessarily has a density of (sometimes overlapping) macroscopic domain walls, on which the heights jump by $1$. These domain walls can be thought of as a gas of macroscopic objects with complicated interactions. Their emergence, as the slope changes from zero, is an instance of the commensurate-incommensurate transition (see, e.g.,~\cite[Chapter 10]{chaikin1995principles}) as the domain walls, on the one hand, are bound to lattice positions, and, on the other hand, obey a possibly irrational average spacing.

Following Fisher and Fisher~\cite{fisher1982wall}, this point of view was discussed in the seminal work of Bricmont, El Mellouki, and Fr\"ohlich~\cite{bricmont1986random}. The latter authors write ``\emph{We pose, but do not solve, the problem of constructing 
statistical mechanical methods (e.g., some sort of Mayer expansion) for a 
gas of extended objects, such as the steps appearing in a tilted interface}".
Our work provides new methods for studying tilted interfaces.

\smallskip \noindent 
(3) While the theory of random non-crossing \emph{paths} is well developed dating back to~\cite{KarlinMcGregor}, the literature on non-crossing \emph{surfaces} of two or more dimensions seems much sparser; they were discussed, but not analyzed, in~\cite[Section 3]{bricmont1986random} and~\cite[Section 4.6]{Sheffield-Random-Surfaces}. 

Our work includes and has an especially clean interpretation for non-crossing surfaces. Consider an infinite family of slope-zero low-temperature two-dimensional integer-valued height functions, conditioned on non-crossing; our main theorem constructs an infinite-volume Gibbs measure over such non-crossing surfaces, and gives the exact sequence of heights they each localize about. See Section~\ref{sec:non-crossing-surfaces}, specifically Corollary~\ref{cor:non-crossing-surfaces} for more on this important example.

 \subsection{Main results}\label{subsec:main-results}
 In this section, we present our main results.

\subsubsection{Existence and uniqueness of sloped states}
Our results on sloped states are best phrased in the language of (infinite-volume) Gibbs measures. Recall that a Gibbs measure $\mu$ for the integer-valued height function model with interactions $\mathbf{V}$ is a measure over $\Omega$ satisfying the Dobrushin--Lanford--Ruelle (DLR) consistency relations~\cite{georgii2011gibbs}. The gradient measure of a Gibbs measure $\mu$ is the law of the discrete gradients $(\phi_v - \phi_u)_{uv\in \vec{E}(\mathbb Z^d)}$ that it induces (equivalently, it is the law of the height function $\phi$ modulo an additive integer constant). When the gradient measure is translation invariant, we define the \emph{slope} of the Gibbs measure $\mu$, or equivalently of its gradient, as the vector
\begin{equation}\label{eq:slope-def}
    \slope(\mu)_i:=\mu[\phi_{\mathfrak{e}_i} - \phi_0] \qquad \text{for }1\le i\le d\,,
\end{equation}
whenever the expectation exists, with $\slope(\mu)$ undefined otherwise. It is clear that the slope vector lies in the set of \emph{admissible slopes},
\begin{equation}\label{eq:admissible-slopes}
\mathfrak{S} := \{\mathfrak{s}\in \mathbb R^d: V_i(\lfloor\mathfrak{s}_i\rfloor)<\infty\text{ and }V_i(\lceil\mathfrak{s}_i\rceil)<\infty \text{ for all }1\le i\le d\}.
\end{equation}
By way of example, $\mathfrak{S}=\mathbb{R}^d$ for the $|\nabla \phi|^p$ models from Example~\ref{item:nabla-phi}, $\mathfrak{S}=[-1,1]^{d}$ for the $1$-Lipschitz model from~\ref{item:Lipschitz}, and $\mathfrak{S} = \mathbb{R}^{d-1} \times [0,\infty)$ for non-crossing surfaces from~\ref{item:non-crossing-surfaces} with first $d-1$ interactions of $|\nabla \phi|^p$ type.
We write $\mathfrak{S}^\circ$ for the interior of $\mathfrak{S}$.

The first question we consider is the existence of a Gibbs measure with a given slope. Precisely, a slope $\mathfrak{s}$ is called \emph{localized} if there exists a Gibbs measure $\mu$, with \emph{ergodic} gradient measure, having slope $\mathfrak{s}$; otherwise, the slope is called \emph{delocalized}. Ergodicity rules out the possibility that $\mu$ is a mixture of measures of different slopes. Existence vs.\ non-existence of Gibbs measures is very closely related to other notions of localization vs.\ delocalization (e.g., in terms of the fluctuations of the height in finite volume with sloped boundary conditions) as the DLR condition requires that under resampling in an arbitrarily large finite domain, the heights remain tight.

The next question concerns uniqueness. Given a Gibbs measure $\mu$ and an integer $k$, one can construct the measure $\mu+k$ as the law of $(\phi_v+k)_v$ when $\phi\sim \mu$. This will also be a Gibbs measure having the same gradient measure as $\mu$. Therefore, the natural uniqueness question is whether there is a unique ergodic \emph{gradient Gibbs measure} of a given slope. 
We will address this question in the family of \emph{minimal} gradient measures, which means that the specific free energy of the gradient measure is minimal among gradient measures with slope $\mathfrak{s}$. The notions of gradient Gibbs measures and minimality are defined precisely in Section~\ref{sec:background gradient Gibbs measures}. Conjecture 10.3.1 of~\cite{Sheffield-Random-Surfaces} predicts that any ergodic gradient measure with finite free energy is necessarily minimal, in which case the latter constraint is not needed.

\begin{thm}\label{thm:main-localization}(Existence and uniqueness)
    Let $d\ge 2$. There exists $\beta_0(d)>0$ such that the following holds. Assume that the interactions $\V$ satisfy~\eqref{eq:potential assumption}. Let $\mathfrak{s}\in\mathfrak{S}^\circ$ and suppose that~\eqref{eq:main assumption} holds with $\beta_0(d)$.

    There is a \emph{unique} ergodic, minimal gradient Gibbs measure of slope $\mathfrak{s}$. Moreover, it is the gradient of a (height function) Gibbs measure. In particular, the slope $\mathfrak{s}$ is localized.
\end{thm}

A concrete construction of a Gibbs measure with ergodic, minimal gradient measure of slope $\mathfrak{s}$ as an infinite-volume limit of $L\times \cdots \times L$ boxes with certain ``sloped-periodic'' boundary conditions is described in Section~\ref{subsec:proof-ideas}. 

\subsubsection{Structure of the sloped state}
Once the localization of a slope has been established, we seek to obtain further information on the Gibbs measures realizing it. Of special interest is the structure of the macroscopic domain walls, or steps, in the sloped height function. There is a tension between the fact that the height function takes integer values and the fact that it needs to conform to a given slope vector $\mathfrak{s}$. 
When $\mathfrak{s}\in\Z^d$, the height function can respect both the integrality constraint and the slope by increasing, on average, by $\mathfrak{s}_i$ for every step in the $\mathfrak{e}_i$ coordinate. However, when $\mathfrak{s}\notin\Z^d$ (including the possibility that some coordinate in $\mathfrak{s}$ is irrational), the structure of the macroscopic domain walls is non-trivial. Sheffield~\cite[Section 8.7]{Sheffield-Random-Surfaces} showed that in this case the ergodic, minimal gradient Gibbs measure of slope $\mathfrak{s}$ will fail to be extremal, and described some of the properties of its extremal components: see Section~\ref{sec:background gradient Gibbs measures}. We sharpen this description, under our condition~\eqref{eq:main assumption}, by showing that the extremal components are characterized by
the locations of the macroscopic domain walls, with each extremal component corresponding to a different way in which these domain walls are maximally separated. We formulate this precisely in this section, using the following definition.

Since macroscopic domain walls only arise in directions orthogonal to non-zero coordinates of $\mathfrak{s}$, under~\eqref{eq:main assumption} it is helpful to think of the layers of $\mathbb Z^d$ given by fixing all coordinates except $i_0, i_1$.  

\begin{defn}\label{defn:dominant-height}
    Let $i_0,i_1$ be the coordinates from~\eqref{eq:main assumption}. 
    \begin{itemize}
        \item (Layer) The \emph{layer} of a vertex $v\in\Z^d$ is the plane $\mathcal{L}_v:=\{v + k\mathfrak{e}_{i_0}+\ell\mathfrak{e}_{i_1}\colon k,\ell\in\Z\}$.
        \item (Dominant height) For a Gibbs measure $\mu$, we say it has a \emph{dominant height} at layer $\mathcal{L}$ if, almost surely, when $\phi$ is sampled from $\mu$, there exists a (possibly random) $h\in\Z$ such that
        $\{w\in \mathcal{L}: \phi(w)=h\}$ has an infinite connected component, whose complement in $\mathcal{L}$ has only finite connected components (in the connectivity on $\mathcal{L}$ induced from $\Z^d$).
        
        If $\mu$ has dominant heights at every layer, we write $h({\mathcal{L}})$ to denote the dominant height at layer $\mathcal{L}$. For convenience, we also set $h_0 := h({\mathcal{L}_0})$.
    \end{itemize}
\end{defn}

\begin{thm}\label{thm:structural results} (Dominant heights and extremal decomposition)
    Suppose the same assumptions as in Theorem~\ref{thm:main-localization}. Let $\mu$ be a Gibbs measure with ergodic, minimal gradient measure of slope $\mathfrak{s}$. Then
    \begin{enumerate}
        \item (Existence of dominant heights) $\mu$ has dominant heights at every layer.
        \item (Structure of dominant heights) There exists a random variable $\alpha$ such that almost surely,
        \begin{align*}
            h({\mathcal L}) = h_0 + \lfloor \alpha + \langle \mathfrak{s}, \mathcal L\rangle\rfloor \qquad \text{for all layers }\mathcal{L}
        \end{align*}
        where $\langle \mathfrak{s}, \mathcal L\rangle$ is the Euclidean inner product of $\mathfrak{s}$ with any vertex in $\mathcal L$ (all of them having the same inner product since $\mathfrak{s}_{i_0}=\mathfrak{s}_{i_1} =0$). 
        \begin{itemize}
            \item If $\mathfrak{s}$ has an irrational coordinate then the distribution of $\alpha$ under $\mu$ is uniform on~$[0,1]$. 
            \item If $\mathfrak{s}$ has all rational coordinates, with lowest common denominator $n$, then the distribution of $\alpha$ under $\mu$ is uniform on $\{0,\frac{1}{n},\ldots, \frac{n-1}{n}\}$.
        \end{itemize}
        \item (Extremal decomposition) 
        $\mu$ conditioned on $h_0$ and $\alpha$ is an extremal Gibbs measure.
    \end{enumerate}
\end{thm}
We make two remarks on the theorem.
First, as $\mu$ has dominant heights at every layer, the macroscopic domain walls of $\phi$ occur (up to microscopic perturbations) between adjacent layers with different dominant heights, that is, between layers $\mathcal L$ and $\mathcal{L}+\mathfrak{e}_i$ when $h({\mathcal L})\neq h({\mathcal L+\mathfrak{e}_i})$, or equivalently when $\lfloor \alpha+\langle \mathfrak{s}, \mathcal{L}\rangle \rfloor\neq \lfloor \alpha+ \langle \mathfrak{s}, \mathcal{L}\rangle + \mathfrak{s}_i\rfloor$. The location of the domain walls then depends on the random variable $\alpha$, which serves as an initial offset to the ``rotation on the circle'' dynamical system $\langle\mathfrak{s}, \mathcal L\rangle$, with different values of $\alpha$ (in $[0,1)$ in the irrational case, and in $\{0,\frac{1}{n},\ldots, \frac{n-1}{n}\}$ in the rational case) yielding distinct location fields. This dynamical system is thus the way in which the height function model resolves the tension between the domain walls being at integer locations, the fact that their overall density is governed by $\mathfrak{s}$, and the entropic repulsion of the walls which pushes them to maximal separation. 
In particular, the favored dominant height sequences are those with the most entropy around them (maximally separated domain walls allow for the most excitations in each layer), and those are precisely the ones described by translates of the near-affine function $\lfloor\langle\mathfrak{s} , \mathcal L\rangle \rfloor$. 
Figure~\ref{fig:dominant height sequence} depicts the dominant height sequences characterizing the distinct extremal measures for the fixed choice of $h_0=0$. 

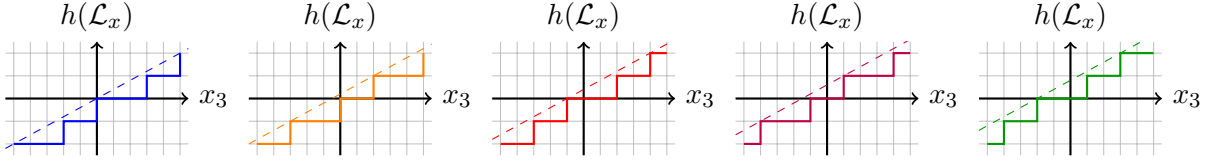
\begin{figure}
\centering
\begin{tikzpicture}[xscale=0.22, yscale=.3]
    \draw[lightgray, very thin, step=1]
        (-5.5,-2.5) grid (5.5,2.5);
    \draw[->, thick] (-5.5,0) -- (5.5,0) node[right] {$x_3$};
    \draw[->, thick] (0,-2.5) -- (0,2.5)
        node[above] {$h({\mathcal{L}_x)}$};

    \begin{scope}
        \clip (-5.5,-2.5) rectangle (5.5,2.5);
        \draw[blue, thin, dashed, domain=-5.5:5.5, samples=2]
            plot (\x,{2*\x/5});
    \end{scope}

    \draw[blue, thick, const plot]
        plot coordinates {
            (-5,-2) (-4,-2) (-3,-2) (-2,-1) (-1,-1)
            (0,0) (1,0) (2,0) (3,1) (4,1) (5,2)
        };
\end{tikzpicture}
%
\begin{tikzpicture}[xscale=0.22, yscale=.3]
    \draw[lightgray, very thin, step=1]
        (-5.5,-2.5) grid (5.5,2.5);
    \draw[->, thick] (-5.5,0) -- (5.5,0) node[right] {$x_3$};
    \draw[->, thick] (0,-2.5) -- (0,2.5)
        node[above] {$h({\mathcal{L}_x)}$};

    \begin{scope}
        \clip (-5.5,-2.5) rectangle (5.5,2.5);
        \draw[orange, thin, dashed, domain=-5.5:5.5, samples=2]
            plot (\x,{1/5 + 2*\x/5});
    \end{scope}

    \draw[orange, thick, const plot]
        plot coordinates {
            (-5,-2) (-4,-2) (-3,-1) (-2,-1) (-1,-1)
            (0,0) (1,0) (2,1) (3,1) (4,1) (5,2)
        };
\end{tikzpicture}
%
\begin{tikzpicture}[xscale=0.22, yscale=.3]
    \draw[lightgray, very thin, step=1]
        (-5.5,-2.5) grid (5.5,2.5);
    \draw[->, thick] (-5.5,0) -- (5.5,0) node[right] {$x_3$};
    \draw[->, thick] (0,-2.5) -- (0,2.5)
        node[above] {$h({\mathcal{L}_x)}$};

    \begin{scope}
        \clip (-5.5,-2.5) rectangle (5.5,2.5);
        \draw[red, thin, dashed, domain=-5.5:5.5, samples=2]
            plot (\x,{2/5 + 2*\x/5});
    \end{scope}

    \draw[red, thick, const plot]
        plot coordinates {
            (-5,-2) (-4,-2) (-3,-1) (-2,-1) (-1,0)
            (0,0) (1,0) (2,1) (3,1) (4,2) (5,2)
        };
\end{tikzpicture}
%
\begin{tikzpicture}[xscale=0.22, yscale=.3]
    \draw[lightgray, very thin, step=1]
        (-5.5,-2.5) grid (5.5,2.5);
    \draw[->, thick] (-5.5,0) -- (5.5,0) node[right] {$x_3$};
    \draw[->, thick] (0,-2.5) -- (0,2.5)
        node[above] {$h({\mathcal{L}_x)}$};

    \begin{scope}
        \clip (-5.5,-2.5) rectangle (5.5,2.5);
        \draw[purple, thin, dashed, domain=-5.5:5.5, samples=2]
            plot (\x,{3/5 + 2*\x/5});
    \end{scope}

    \draw[purple, thick, const plot]
        plot coordinates {
            (-5,-2) (-4,-1) (-3,-1) (-2,-1) (-1,0)
            (0,0) (1,1) (2,1) (3,1) (4,2) (5,2)
        };
\end{tikzpicture}
%
\begin{tikzpicture}[xscale=0.22, yscale=.3]
    \draw[lightgray, very thin, step=1]
        (-5.5,-2.5) grid (5.5,2.5);
    \draw[->, thick] (-5.5,0) -- (5.5,0) node[right] {$x_3$};
    \draw[->, thick] (0,-2.5) -- (0,2.5)
        node[above] {$h({\mathcal{L}_x)}$};

    \begin{scope}
        \clip (-5.5,-2.5) rectangle (5.5,2.5);
        \draw[green!60!black, thin, dashed,
              domain=-5.5:5.5, samples=2]
            plot (\x,{4/5 + 2*\x/5});
    \end{scope}

    \draw[green!60!black, thick, const plot]
        plot coordinates {
            (-5,-2) (-4,-1) (-3,-1) (-2,0) (-1,0)
            (0,0) (1,1) (2,1) (3,2) (4,2) (5,2)
        };
\end{tikzpicture}
\caption{(Solid lines): At slope $\mathfrak{s} = (0,0,2/5)$, the five extremal measures having $h_0 =0$ are characterized by the five distinct dominant height sequences above. (Dashed lines): the functions $\alpha + \langle \mathfrak{s} , \mathcal L\rangle$ whose floor gives the dominant heights.}\label{fig:dominant height sequence} 
\end{figure}

Second, the extremal components of $\mu$ are indexed by the values of $h_0$ (of positive probability) and $\alpha$, while the extremal components of the gradient measure of $\mu$ are indexed solely by $\alpha$. In particular, when $\mathfrak{s}\notin\Z^d$, the gradient measure is ergodic but not extremal.

\smallskip
The Gibbs measure with ergodic gradient, $\mu$, necessarily has long-range correlations (heights at far away vertices give information on $h_0$ and $\alpha$ which determine the dominant heights everywhere). Still, the following result shows that within each extremal component, deviations of the height at a vertex $v$ away from the dominant height of its layer have exponential tails, and that the extremal components exhibit exponential decay of correlations. 

\begin{thm}\label{thm:structural results-exponential-decay} (Exponential decay)
    Suppose the same assumptions as in Theorem~\ref{thm:main-localization}. Let $\mu$ be a Gibbs measure with ergodic, minimal gradient measure of slope $\mathfrak{s}$. Then there exist $C,c>0$, depending only on $\beta_0(d)$ of~\eqref{eq:main assumption} and the dimension $d$, such that the following holds:
        \begin{enumerate}
            \item For each $v\in\Z^d$ and $t>0$,
            \begin{equation}
                \mu(|\phi_v - h(\mathcal L_v)|\ge t)\le Ce^{-ct}.
            \end{equation}
            \item Let $f_1,f_2:\Z^{\Z^d}\to[-1,1]$ be local in the sense that $f_i(\phi)$ depends only on the restriction of $\phi$ to a finite $\Lambda_i\subset\Z^d$, for each $i\in\{1,2\}$. Then
            \begin{equation}\label{eq:exponential decay local functions}
                \mathbb{E}\left[|\Cov_{\mu | h_0, \alpha}(f_1(\phi), f_2(\phi))|\right]\le C\sum_{u\in\Lambda_1} e^{-cd(u,\Lambda_2)}
            \end{equation}
            where $\Cov_{\mu | h_0,\alpha}$ denotes covariance when $\phi$ is sampled from (the extremal measure) $\mu$ conditioned on $h_0$ and $\alpha$, where the expectation $\mathbb{E}$ is over $h_0$ and $\alpha$ and where $d(u,A)=\inf_{v\in A}\|u-v\|_1$ is the graph distance in $\Z^d$ from $u$ to $A$.
        \end{enumerate}
\end{thm}

This exponential decay of correlations is in contrast to the real-valued Gaussian free field which exhibits power-law decay of correlations in $d\ge 3$, and importantly holds even in the sloped directions where we do not assume strong coupling strengths.

\subsection{The application to non-crossing surfaces}\label{sec:non-crossing-surfaces}

As described earlier in Example~\ref{item:non-crossing-surfaces}, the generality of the assumptions on $\mathbf{V}$ allows us to deduce a notable implication for the study of non-crossing surfaces. Bricmont--El Mellouki--Fr\"ohlich~\cite{bricmont1986random} posed understanding the behavior of many zero-slope, 2D $|\nabla \phi|^p$ surfaces conditioned not to cross, as motivation for their simpler questions of one SOS surface conditioned either to take values in an interval (between a pair of idealized flat surfaces called hard walls), or to only take positive values (above a hard wall). 
When a low-temperature surface is only constrained to be above a wall, they identified that there is no infinite-volume Gibbs measure due to an \emph{entropic repulsion} away from the floor, whereby the surface delocalizes in order to make room for downward excitations.
The work led to a rich line of research on more refined aspects of this entropic repulsion phenomenon: a sampling of these works is~\cite{BolthausenDeuschelZeitouni,BolthausenDeuschelGiacomin,Sakagawa} for the real-valued Gaussian free field,~\cite{CLMST16,LMS16,chen2026limitshapeemergencediscrete} for the SOS and more generally $|\nabla \phi|^p$ surfaces, and recently~\cite{chen2025logarithmic} for the 3D Ising and Potts Dobrushin interfaces. Our Theorems~\ref{thm:main-localization}--\ref{thm:structural results-exponential-decay} imply that the mutual entropic repulsions between consecutive non-crossing surfaces \emph{localize} the surfaces to heights given by the same ``maximally separated" dominant height sequences.  

To map to our setting, let $d=D+1$ for $D\ge 2$. By way of example, take the $|\nabla \phi|^p$, $p\ge 1$, interaction on the surfaces as in Example~\ref{item:nabla-phi}, so that $$V_1 (k) = \cdots = V_D(k) = \beta |k|^p \qquad \text{and} \qquad  V_{D+1} (k) = \begin{cases}
    0 & k\ge 0 \\ \infty & k<0
\end{cases}\,.$$ The $(D+1)$'th coordinate is the index of the surface, and the first $D$ coordinates give the heights taken by that $|\nabla \phi|^p$ surface. Finally, consider a slope vector $\mathfrak{s} = (0,...,0,\theta)$ where $\theta>0$ will have the interpretation of  the expected spacing between consecutive surfaces. 

For purposes of the DLR condition, gradient measures, and definitions of exponential decay, the family of non-crossing surfaces is encoded by a single height function, as in Example~\ref{item:non-crossing-surfaces}. 

\begin{cor}\label{cor:non-crossing-surfaces}
    Fix $D\ge 2$. There exists $\beta_0(D)>0$ such that for all $p\ge 1$,  $\beta>\beta_0$, and $\theta \in (0,\infty)$, the following holds. There exists an infinite-volume Gibbs measure over an infinite family $(\psi^{(i)}_v)_{v\in \mathbb{Z}^D, i\in \mathbb Z}$ of zero-slope $|\nabla\phi|^p$ height functions conditioned on non-crossing ($\psi^{(i+1)}_v \ge \psi_{v}^{(i)}$) and expected spacing $\theta$ in the sense of $\mathbb{E}[\psi_v^{(i+1)}- \psi_v^{(i)}]=\theta$ for all $v,i$, such that: 
    \begin{enumerate}
        \item It is unique (modulo global addition of an integer) among the family of such measures with ergodic minimal gradient.
        \item There exists a pair of random variables $h_0, \alpha$, where $\alpha \sim \text{Unif}[0,1]$ if $\theta \notin \mathbb Q$ and $\alpha \sim \text{Unif}\{0,...,\frac{n-1}{n}\}$ if $\theta\in \mathbb Q$ with denominator $n$, such that  in almost every extremal component, for all $k\in \Z$, the $k$'th surface has dominant height $h_0 + \lfloor \alpha + k \theta \rfloor$. 
        \item The family of surfaces has exponential decay of correlations in the $\mathbb Z^{D+1}$-distance within its extremal components in the sense of~\eqref{eq:exponential decay local functions}.
    \end{enumerate}
\end{cor}

We refer the reader to Figure~\ref{fig:non-crossing-surfaces} for a depiction of Corollary~\ref{cor:non-crossing-surfaces} for the $p=1$ case, with $\theta = 5/2$.

Our results imply similar facts for infinitely many non-crossing sloped surfaces in $D\ge 3$ dimensions satisfying~\eqref{eq:main assumption}.

\subsection{Related work}\label{subsec:related-work}
In this section we discuss prior work on the localization of $\nabla \phi$ height functions on the lattice, placing our results in context.

\subsubsection{Real-valued height functions} We first discuss real-valued $\nabla \phi$ height functions.

\subsubsection*{Gaussian free field:}
The most basic example of a $\nabla \phi$ height function is the \emph{real-valued} Gaussian free field: The case where $\phi$ takes real values, with an interaction $\V$ that is isotropic, equaling $V(x) = x^2$ in all directions. Its Gaussian distribution permits exact calculations, which imply, in particular, that at all slopes the model delocalizes in dimension $d=2$ and localizes in $d\ge 3$. 
Moreover, for each slope there is a unique ergodic minimal gradient measure of that slope. 
\smallskip
\subsubsection*{Zero slope:} The seminal work of Brascamp--Lieb--Lebowitz~\cite{brascamp1975statistical} considered the real-valued case with an even interaction $V:\mathbb{R}\to\mathbb{R}$, identical in all directions, assumed only to satisfy that $\int_0^\infty e^{-\alpha V(x)}dx<\infty$ for all $\alpha>0$. They focused on the zero slope case and conjectured that delocalization always holds in dimension $d=2$ while localization always holds in dimensions $d\ge 3$. Their work establishes the conjecture under various conditions on the interaction, with delocalization proofs relying on Mermin--Wagner-type arguments, and localization proofs using the Brascamp--Lieb inequality~\cite{brascamp1975some, brascamp1976extensions} (which compares the model to the Gaussian free field) which relies on the convexity, and often strict convexity, of $V$. Later works extended significantly the allowed class of interactions, including ~\cite{dobrushin1980nonexistence,frohlich1981absence,ioffe20022d, milos2015delocalization} for delocalization results (in two dimensions) and~\cite{brydges2012fluctuation, ye2019models, magazinov2022concentration, Dario2024, sellke2024localization} for localization results (in dimensions $d\ge 3$). There is also significant literature on Gaussian free field scaling limits; see the recent~\cite{buchholz2026gradient} and references therein.

\subsubsection*{Non-zero slope:} It is a simple observation that any non-zero slope may be mapped to the zero slope, by subtracting the corresponding hyperplane from the heights, at the cost of considering (possibly) non-even interactions, which may differ between different coordinate directions. However, we are not aware of works focusing on localization for non-even interactions, or at non-zero slope (in dimensions $d\ge 3$). Arguments using the Brascamp--Lieb inequality apply to show localization around the expected height for arbitrary boundary conditions, when the interaction $\V = (V_i)_{1\le i\le d}$ satisfies that each $V_i$ is 
\emph{strictly} convex, though it need not be even (see also~\cite[Remark at end of Section 4.1]{magazinov2022concentration}). 
\smallskip
\subsubsection*{Uniqueness:} 
Uniqueness of ergodic minimal gradient Gibbs measures of a given slope (in the real-valued case) was shown by Funaki--Spohn~\cite{funaki1997motion} in the isotropic case under certain assumptions and extended by Sheffield~\cite[Theorem 8.6.3]{Sheffield-Random-Surfaces} in great generality. This uniqueness may break for \emph{non-convex} interactions at zero slope~\cite{biskup2007phase, buchholz2021phase}.

\smallskip
\subsubsection{Integer-valued height functions}
Our focus in this paper is on \emph{integer-valued} $\nabla \phi$ height functions, where the theory is less developed.
\subsubsection*{Two dimensions:} In dimension $d=2$ our results do not add to the existing literature: First, uniqueness of sloped states in two dimensions was established by Sheffield~\cite[Theorem 9.1.1]{Sheffield-Random-Surfaces}. Second, when $d=2$, Assumption~\eqref{eq:main assumption} implies that the slope $\mathfrak{s}=0$ and the low temperature condition holds in both coordinates, whence our existence and structural results are derivable via standard Peierls-type methods (c.f., Brandenberger--Wayne~\cite{brandenberger1982decay}). We thus refrain from further detailed discussion of the two-dimensional case, remarking briefly that there is significant past and contemporary work around the roughening transition which takes place at zero slope, including~\cite{FrohlichSpencer,lammers2022height,aizenman2021depinning, van2023elementary,bauerschmidt2024discrete,bauerschmidt2024discrete2,park2025central, duminil2021macroscopic,chandgotia2021delocalization,lis2021delocalization,Glazman2021,glazman2023transition,karrila2023logarithmic,DuminilCopin2024,glazman2025delocalisation}. Regarding non-zero slope,~\cite[Corollary 9.1.2]{Sheffield-Random-Surfaces} proved delocalization (at all temperatures) when $\mathfrak{s}\notin\Z^2$. More recent works established delocalization at general slopes, sometimes with quantitative delocalization rates, in certain scenarios~\cite{LammersOtt,garban2023statistical,ott2025quantitative,laslier2024tilted}. Shlosman~\cite{Shlosman} proved localization at low temperatures when $\mathfrak{s}\in\Z^2$ and the interaction is isotropic and strictly convex.\footnote
{The result of~\cite{Shlosman} is written for $\mathfrak{s}=(0,n)$ or $\mathfrak{s}=(n,0)$ with $n\in\Z$, but seems to extend to $\mathfrak{s}\in\Z^2$.} 

\subsubsection*{Zero slope:} We turn to dimensions $d\ge 3$. Again, existing work has concentrated on the isotropic case at zero slope. Localization and exponential decay at low temperatures with zero boundary conditions follow from Peierls-type methods as in~\cite{brandenberger1982decay}. However, as far as we know, the literature does not provide an accompanying general uniqueness result for ergodic minimal gradient Gibbs measures of slope zero (Theorem~\ref{thm:main-localization} fills this gap).

At high (or intermediate) temperatures less is known: for certain interactions, such as $|\nabla \phi|^p$ with $0< p\le 2$, the height function with zero boundary conditions is known to be localized. This is proved in~\cite{van2025duality}, and also follows by noting that the height function is
dominated (in a suitable sense) by the real-valued height function with the same interaction~\cite[Section 9.4 and Lemma D.1]{aizenman2021depinning}, so that localization follows from the known real-valued results. For the SOS model, an alternative proof of localization is given in~\cite[Section 4]{bricmont1982surface}. For the integer-valued Gaussian free field, exponential decay has also been established~\cite{gopfert1982proof}.

Localization, and uniqueness of ergodic gradient Gibbs measures, at zero slope have also been shown for certain Lipschitz functions in high dimensions~\cite{Peled2017,galvin2015phase,Feldheim2018, peled2020long}.
\subsubsection*{Non-zero slope:} Unlike the real-valued case, one may map between two slopes, at the cost of considering non-isotropic interactions, only when their difference is in $\Z^d$.
The only localization results in $d\ge 3$ that we are aware of at non-zero slope are: (1) the simple fact that the integer-valued Gaussian free field at a slope $\mathfrak{s}\in\Z^d$ is localized at all temperatures, as it has the same distribution as the zero slope field to which a hyperplane of slope $\mathfrak{s}$ is added. (2) It seems that the method of Shlosman~\cite{Shlosman} continues to apply to show localization for $\mathfrak{s}\in\Z^d$ at low temperature for strictly convex interactions. (3) Localization at a few \emph{special integer slopes} for Ising model interfaces (which are related to the SOS model) in $3+1$ dimensions at low temperatures is established in~\cite[Section 6]{messager1977correlation} by extending van Beijeren's technique~\cite{van1975interface}.

\subsubsection*{Uniqueness:} As mentioned, uniqueness of ergodic, minimal gradient Gibbs measures of slope $\mathfrak{s}\in\mathfrak{S}^\circ$ is shown in~\cite[Corollary 8.7.6]{Sheffield-Random-Surfaces} when $\mathfrak{s}$ has an irrational coordinate. We are not aware of other results in this direction.

\subsection{Proof ideas}\label{subsec:proof-ideas}
In this subsection, we will summarize some of the central challenges to obtaining results on sloped height functions, then describe some key ideas from our proof. 

\medskip
\noindent \emph{Challenges.}
At zero slope, there are several general-purpose tools available for proving localization of height functions. For certain choices of interactions (e.g., $|\nabla \phi|^p$ for $1\le p \le 2$) there are correlation inequalities like absolute-value FKG, or the Ginibre-type inequalities which dominate the integer-valued model's  fluctuations by those of the real-valued free field. However, the applicability of these inequalities seems to be limited to zero slope.
A more robust tool one could hope to appeal to, especially given our assumption~\eqref{eq:main assumption} of the slope having two zero coordinates and strong couplings in those two coordinates, is an argument of Peierls-type. However, the fact that the slope may be non-zero in the other coordinates precludes standard Peierls-type arguments, as any Peierls map will have to contend with interference from the constant density of macroscopic domain walls which necessarily exist.  

The problem is even more acute when $\mathfrak{s}\notin \mathbb Z^d$, as there is not a unique reference ground state with respect to which one defines excitations. 
In fact, in this case there is an exponential multiplicity of ground states of global slope $\mathfrak{s}$, and 
large excitations may align in such a way to move the interface between one ground state and another with minimal energy cost. 
Moreover, as Theorem~\ref{thm:structural results} shows, the multiplicity of ground states, and the entropic considerations from the interactions in the sloped directions, play a crucial role in the behavior of the model and cannot be ignored.

\medskip

\noindent \emph{Restricted reflection positivity.} 
Our starting point for circumventing the above challenges will be a restricted notion of reflection positivity, which helps isolate the rigidity in the slope-zero directions from the considerations in the non-zero slope directions. Reflection positivity is a celebrated property of many statistical physics models, which gives a powerful method (alternative to cluster expansion/Peierls arguments) for establishing long-range order. The idea is that on an $L\times L$ torus, the probability of a certain local event $\mathcal E_e$ (e.g., a non-zero gradient along an edge $e$) is bounded by the $L^2$'th root of the probability of the global event $\bigcap_{\tau \in T} \tau\mathcal E_{e}$, that suitable reflections of the local event occur on all $L^2$ edges simultaneously. See Section~\ref{subsec:reflection-positivity-primer} for a primer on reflection positivity.

Traditionally, reflection positivity requires strong symmetries of the underlying system that are broken by a non-zero slope. A key insight of our argument is to view the sloped height function distribution as having a restricted form of reflection positivity that holds only in the two zero-slope directions (which we refer to as horizontal) and leverage this restricted reflection positivity. Our idea of using restricted reflection positivity originated from~\cite{ShlosmanVignaud,VignaudThesis}, who introduced it to give a new proof of the rigidity of the zero-slope Ising interface~\cite{Dobrushin72a} (as well as other interfaces). 

\medskip
\noindent \emph{Sloped-periodic Gibbs measures}.
In order to leverage the reflection positivity in the two zero-slope directions, we need a distribution to be translation and reflection invariant in those two directions. We therefore wish to define a finite-volume Gibbs measure on the torus $\mathbb T_L^d$ ``with slope $\mathfrak{s}$". We follow the construction of  \emph{sloped-periodic} boundary conditions in Section 4 of~\cite{Sheffield-Random-Surfaces}, where edges interior to the fundamental domain $\{1,...,L\}^d$ of the torus have the standard interaction of~\eqref{eq:height function model}, but gradients along edges that ``wrap around" the torus, i.e., connect $v,w$ on opposite hyperplanes of the fundamental domain, are computed with a shift of $\langle \mathfrak{s}, v-w\rangle$: roughly, 
\begin{align}\label{eq:sloped-periodic-V-shift}
    V_{vw}( \phi_w - \phi_v - \langle \mathfrak{s},v-w\rangle)\,.   
\end{align}
We are more precise about this construction in Section~\ref{subsec:ergodic-minimal-grad-meas}. 
By the convexity of $\mathbf{V}$, see e.g. ~\cite[Lemma 4.2.6]{Sheffield-Random-Surfaces}, subsequential limits of the induced gradient measures exist and have slope $\mathfrak{s}$. 

\medskip
\noindent \emph{Non-zero gradients along horizontal edges are exponentially unlikely}. 
In order to lift limiting gradient measures to infinite-volume measures over height functions, the main work is to show that the height differences $\phi_x - \phi_y$ remain uniformly tight as $d(x,y)\to \infty$ if $x,y$ are in the same layer.  

We will do so by showing that in the gradient measure with sloped-periodic boundary conditions, non-zero gradients within a layer are exponentially suppressed (for large $\beta_0$). Such an estimate is the content of Theorem~\ref{thm:1-2-direction-rigidity}. 
Towards that, by applying restricted reflection positivity to the sloped-periodic gradient measure, for each horizontal edge $e$, the aim is to define the appropriate family of events $(\mathcal E_{e,i})_{i}$ that cover the event of a non-zero gradient at $e$, and whose horizontal dissemination across $\mathbb T_L^2$ can be shown to have probability $\exp( - \Omega( \beta_0 L^2))$. Because the reflections are only in the horizontal directions, bounding the probability of the reflected event is still not the probability of a fixed configuration, but of an exponentially large family (it marginalizes over values in the vertical directions). 

In particular, the entropy of the number of covering events has to be controlled only by the energy gain from the horizontal edges involved, while ensuring that the change in interactions along vertical edges does not incur a cost in energy. 
This involves the non-trivial Peierls-type combinatorial construction of a map that ``flattens" the gradients along the connected components of defect edges in the column $\{e\} \times \mathbb T_L^{d-2}$ with the same sign on their gradients. There is  ambiguity on how to ``flatten" the gradients, and this poses challenges given the non-zero vertical slope. Definition~\ref{def:single-edge-map} constructs a map that performs this flattening and, for each edge $f=uv$ in a cluster of horizontal defect edges, does the following: 
\begin{itemize}
    \item If $\frac{\phi_u + \phi_v}{2}\in \Z$, it sets the new height of both $u$ and $v$ to that value;  
    \item If $\frac{\phi_u + \phi_v}{2}\notin \mathbb Z$, it makes a judicious choice of the two possible roundings ($\lfloor \frac{\phi_u + \phi_v}{2}\rfloor$ or $\lceil\frac{\phi_u + \phi_v}{2}\rceil$) and sets the new height at both $u$ and $v$ to that value. 
\end{itemize}
In the second case, we leverage convexity of the vertical interactions to show that there is always a good choice of rounding such that the vertical interaction energies will not increase when this is performed throughout a horizontal defect cluster. 

We refer the reader to Section~\ref{sec:reflection-positivity} which is self-contained and describes this mapping on reflected configurations and our novel use of restricted reflection positivity.

\medskip

\noindent \emph{Dominant heights.}
Having shown that horizontal gradients are exponentially unlikely, it is straightforward that in each horizontal layer, there is one (possibly random) dominant height whose level set percolates.  The construction of infinite-volume Gibbs measures from the limiting gradient measure obtained from the sloped-periodic measure then goes by assigning a fixed height (say $0$) to be the dominant height at layer zero (this then determines all the other heights). 

The fact that this construction yields a valid Gibbs measure is because the dominant height level set percolates, so the height prescriptions we made are tail measurable, and thus can be reconstructed when doing resampling of finite regions. This argument is found in Section~\ref{sec:limits}. 

\medskip

\noindent \emph{Uniqueness and classifying the extremal components.} 
One is then interested in the law on the sequence of dominant heights, among the infinite family of possible distributions over dominant height sequences which achieve slope $\mathfrak{s}$. 
Sheffield~\cite{Sheffield-Random-Surfaces} showed that if a slope $\mathfrak{s}$ is localized, which the above construction established, then the number of extremal components and mixture over them in the ergodic gradient measures are determined. 

We study the law of the resulting dominant height sequences by identifying a one-to-one correspondence between extremal components of the ergodic measure and dominant height sequences with maximally separated jumps. 
To identify the dominant height sequence of a given extremal component, we use translation covariance properties of the infinite-volume measures, and show that if an extremal component had any dominant height sequence besides $h_0 + \lfloor \alpha + \langle \mathfrak{s}, \mathcal L\rangle \rfloor$, then it would eventually violate stochastic orderings of~\cite{Sheffield-Random-Surfaces} between different extremal measures. 

Finally, the uniqueness of the ergodic measure is proved using a disagreement percolation argument in the style of~\cite{VanDenBerg}. An important step here is Lemma~\ref{lem:ergodic-minimal-has-dominant-height} which shows that the same control over non-zero horizontal gradients that is obtained from the sloped-periodic limit is applicable to \emph{every} ergodic, minimal gradient Gibbs measure of the given slope. The transfer of chessboard estimates from torus-limit Gibbs measures to general translation-invariant Gibbs measures was previously known only for the hard-core model~\cite{hadas2022columnar}, though that proof does not apply to our setting due to the unbounded interactions (lack of a ``safe symbol'') and the fact that we only have restricted reflection positivity.

\subsection*{Acknowledgments}

The research of R.G.\ is supported in part by NSF CAREER grant 2440509 and NSF DMS grant 2246780. The research of R.P. is partially supported by the National Science Foundation grant DMS-2451133 and the Brin professorship at the University of Maryland. 

We thank Fabio Toninelli and Giambattista Giacomin for a helpful discussion of the literature on sloped states for $\nabla\phi$ height functions. We thank Daniel Hadas for fruitful discussions regarding the chessboard estimate in infinite volume. We thank Senya Shlosman for helpful correspondence.

\subsection*{Statement on the use of generative AI}
We have come up with all proofs in the paper ourselves without AI assistance. The writing of the paper was done without AI assistance besides proofreading of the final draft. 
We consulted ChatGPT (versions from 5.4 to 5.6) in our literature search, including clarification of details of existing papers. ChatGPT~5.6 helped with coding the simulations in Figures~\ref{fig:3d-sloped-surfaces} and~\ref{fig:non-crossing-surfaces}.

\section{Background on gradient Gibbs measures}\label{sec:background gradient Gibbs measures}
In this section, we recap the basic theory of gradient Gibbs measures, their free energy, and central results of~\cite{Sheffield-Random-Surfaces} that we will appeal to in the later sections of the paper. The reader primarily interested only in localization of the slope $\mathfrak{s}$ should feel free to skip this section, as Section~\ref{sec:reflection-positivity} is a self-contained proof of rigidity in finite volume with the sloped-periodic boundary condition.

\subsection{Gibbs and gradient Gibbs measures}

We will work extensively with Gibbs and gradient Gibbs measures for the formal Hamiltonian~\eqref{eq:height function model}. The meaning of the formal Hamiltonian~\eqref{eq:height function model} is the following. For each finite non-empty $\Lambda\subset\Z^d$ and each boundary condition $\psi\in\Omega$, the \emph{finite-volume Gibbs measure} $\mu^{\Lambda, \psi}$ on $\phi\in\Omega$ is defined by
\begin{equation}
    \mu^{\Lambda, \psi}(\phi):=\frac{1}{Z^{\Lambda,\psi}}\exp\Bigg(-\sum_{\substack{uv \in \vec{E}(\mathbb Z^d): \\ \{u,v\}\cap \Lambda \ne \emptyset}}V_{uv}(\phi_v - \phi_u)\Bigg)\prod_{v\notin\Lambda}\1_{\phi_v = \psi_v}
\end{equation}
where
\begin{equation}
    Z^{\Lambda,\psi}:=\sum_{\substack{\phi\in\Omega\\\phi\equiv\psi\text{ on $\Lambda^c$}}} \exp\Bigg(-\sum_{\substack{uv \in \vec{E}(\mathbb Z^d): \\ \{u,v\}\cap \Lambda \ne \emptyset}}V_{uv}(\phi_v - \phi_u)\Bigg)
\end{equation}
is the normalization constant (partition function) turning $\mu^{\Lambda, \psi}$ to a probability measure (these depend only on the restriction $\psi|_{\Lambda^c}$). One checks that $Z^{\Lambda,\psi}<\infty$ under our assumption~\eqref{eq:potential assumption}. The probability measure $\mu^{\Lambda, \psi}$ is well defined when also $Z^{\Lambda,\psi}>0$. 
We say that $\psi$ is $\V$-admissible if $Z^{\Lambda,\psi}>0$ for all finite non-empty $\Lambda\subset\Z^d$, and we note that all $\psi$ are $\V$-admissible when none of the $V_i$ takes the value $+\infty$.

An infinite-volume (height function) \emph{Gibbs measure} $\mu$ is a probability measure on $\Omega$, supported on $\V$-admissible height functions and satisfying the DLR criterion: Suppose $\phi$ is sampled from $\mu$. For each finite non-empty $\Lambda\subset\Z^d$, conditioning on $\phi|_{\Lambda^c}$ the distribution of $\phi$ is the finite-volume Gibbs measure $\mu^{\Lambda,\psi}$ where $\psi\in\Omega$ is any function equaling $\phi|_{\Lambda^c}$ on $\Lambda^c$.

A central role in our work is played by the notion of \emph{gradient Gibbs measures}, as we now define. Define an equivalence relation on $\Omega$ by setting $\phi\sim\psi$ if $\phi-\psi$ is a constant. We write $\phi^\nabla$ for the equivalence class of $\phi$ and set
\begin{equation}
    \Omega^\nabla:=\{\phi^\nabla\colon\phi\in\Omega\}
\end{equation}
for the set of all equivalence classes, equipped with the quotient sigma algebra $\F^\nabla:=\F/\sim$ (i.e., $A\in\F^\nabla$ if and only if the union of the equivalence classes in $A$ is in $\F$). We point out that a gradient function may be equivalently represented as an antisymmetric function on oriented edges of $\Z^d$ (representing the gradient on that edge) with zero sum on every cycle in $\Z^d$.

Observe that the following holds for all finite non-empty $\Lambda\subset\Z^d$ and $\psi^1\sim\psi^2$:
\begin{itemize}
    \item $Z^{\Lambda,\psi^1} = Z^{\Lambda,\psi^2}$. In particular, $\psi^1$ is $\V$-admissible if and only if $\psi^2$ is $\V$-admissible so that the notion of $\V$-admissibility extends to equivalence classes $\psi^\nabla$.
    \item Suppose $\psi^1$ is $\V$-admissible. If $\phi^1$ is sampled from $\mu^{\Lambda, \psi^1}$ and $\phi^2$ is sampled from $\mu^{\Lambda, \psi^2}$ then $\phi^2+c\eqd \phi^1$, where $c\equiv\psi^1-\psi^2$ (and $\eqd$ denotes equality in distribution).
    
    Thus, for a $\V$-admissible equivalence class $\psi^\nabla$, the probability measure $\mu^{\nabla, \Lambda, \psi^\nabla}$ is well defined on $\Omega^\nabla$ as the push-forward of $\mu^{\Lambda, \psi}$ through $\nabla$.
\end{itemize}
An infinite-volume \emph{gradient Gibbs measure} $\mu^\nabla$ is a probability measure on $\Omega^\nabla$, supported on $\V$-admissible equivalence classes, satisfying the DLR criterion: Suppose $\phi^\nabla$ is sampled from $\mu^\nabla$ (the choice of $\phi$ in the equivalence class is arbitrary). For each finite non-empty $\Lambda\subset\Z^d$, conditioning on $\phi^\nabla|_{\Lambda^c}$ the distribution of $\phi^\nabla$ is the finite-volume Gibbs measure $\mu^{\nabla,\Lambda,\psi^\nabla}$ where $\psi\in\Omega$ is any function equaling $\phi|_{\Lambda^c}$ on $\Lambda^c$.

\subsection{Translations and ergodicity}

Given $v\in\Z^d$, the translation operator $\theta_v$ is defined as follows: For a function $\phi\in\Omega$, $(\theta_v\phi)_\cdot := \phi_{\cdot + v}$. For an event $A\in\F$, $\theta_v A:=\{\theta_v\phi\colon \phi\in A\}$. For a measure $\P$ on $\Omega$, $\theta_v\P$ is the pushforward of $\P$ by $\theta_v$. 

An event $A\in\F$ is called \emph{translation invariant} if $\theta_v A =A$ for all $v\in\Z^d$. A measure $\P$ on $\Omega$ is called \emph{translation invariant} if $\theta_v\P=\P$ for all $v\in\Z^d$. A translation invariant measure $\P$ on $\Omega$ is called \emph{ergodic} if $\P(A)\in\{0,1\}$ for all translation-invariant events $A\in\F$. All these definitions extend naturally to gradient functions $\nabla\phi\in\Omega^\nabla$, gradient events $A\in\F^\nabla$ and gradient measures~$\P^\nabla$.

\subsection{Free energies of gradient Gibbs measures}
The \emph{(specific) free energy} (SFE) of a translation-invariant gradient Gibbs measure $\mu^\nabla$ is defined as 
\begin{equation}\label{eq:SFE def}
    \SFE(\mu^\nabla):=\lim_{L\to\infty}\frac{1}{|\Lambda_L|}\left(\mu^\nabla(H^{\Lambda_L,\mathrm{free}}(\phi^\nabla))-\Ent(\mu^\nabla|_{\Lambda_L})\right)
\end{equation}
with the following definitions: $\Lambda_L:=\{-\frac{L}{2},...,\frac{L}{2}\}^d$, and 
\begin{equation}
    H^{\Lambda_L,\mathrm{free}}(\phi):= \sum_{\substack{uv\in \vec{E}(\mathbb Z^d)\\u,v\in\Lambda_L}}V_{uv}(\phi_v - \phi_u)
\end{equation}
is the restriction of the formal Hamiltonian~\eqref{eq:height function model} to $\Lambda_L$ with free boundary conditions, $\Ent$ is the Shannon entropy of a discrete probability measure and $\mu^\nabla|_{\Lambda_L}$ is the restriction of $\mu^\nabla$ to $\Lambda_L$ (i.e., it is the distribution of $(\nabla_{(u,v)}\phi^\nabla \colon u,v\in\Lambda_L, u\sim v)$ when $\phi^\nabla$ is sampled from $\mu^\nabla$).

Sheffield~\cite[Lemma 2.3.4 and following paragraph]{Sheffield-Random-Surfaces} shows that the limit in~\eqref{eq:SFE def} exists in $(-\infty,\infty]$. 
It is also shown~\cite[Lemma 2.5.1 and Theorem 2.5.2]{Sheffield-Random-Surfaces} that $\SFE(\mu^\nabla)$ is minimized among translation-invariant measures $\mu^\nabla$ on gradient functions, with the minimum value in $\mathbb{R}$ and any measure achieving the minimum being a gradient Gibbs measure.

Recall the definition of the set of admissible slopes $\mathfrak{S}$ for interaction $\mathbf{V}$ from~\eqref{eq:admissible-slopes}, and the slope of a translation-invariant gradient Gibbs measure from~\eqref{eq:slope-def}. 

The \emph{surface tension} $\tau:\mathbb{R}^d\to\mathbb{R}\cup\{\infty\}$ of a slope is defined by
\begin{equation} \label{eq:surfacetension}
    \tau(\mathfrak{s}):=\inf\{\SFE(\nu^\nabla) \colon \nu^\nabla\text{ is a translation-invariant gradient measure with }\slope(\nu^\nabla)=\mathfrak{s}\}.
\end{equation}
It is clear that $\tau(\mathfrak{s})<\infty$ if and only if $\mathfrak{s}\in\mathfrak{S}$ (see also~\cite[Lemma 4.2.4 and Corollary 4.2.5]{Sheffield-Random-Surfaces}).
Given $\mathfrak{s}\in\mathfrak{S}$, a translation-invariant gradient measure $\mu^\nabla$ of slope $\mathfrak{s}$ is called \emph{minimal} if $\SFE(\mu^\nabla)=\tau(\mathfrak{s})$ (whenever we say minimal, we mean among gradient measures of slope $\mathfrak{s}$). A minimal $\mu^\nabla$ is necessarily a gradient Gibbs measure by~\cite[Theorem 2.5.2]{Sheffield-Random-Surfaces}.

Finally, recall that we say a slope $\mathfrak{s}$ is localized if there exists a Gibbs measure on height functions whose gradient Gibbs measure is translation invariant and has slope $\mathfrak{s}$, and it is delocalized otherwise.  Theorem 8.6.3 of~\cite{Sheffield-Random-Surfaces} shows that if $\mathfrak{s}$ is localized, then any translation invariant gradient Gibbs measure of slope $\mathfrak{s}$ with minimal SFE will be the gradient of a height-function Gibbs measure.

\subsection{Infinite free energy and non-minimality}\label{sec:infinite free energy}

We briefly discuss the necessity in the above discussion, of the conditions on finite slope, minimal free energy, and the possibility of gradient measures with infinite free energy. This discussion is only for context, and not used in our paper. 

Let $\mu^\nabla$ be a gradient Gibbs measure. Sheffield~\cite[Lemma 2.3.8]{Sheffield-Random-Surfaces} shows that if $\SFE(\mu^\nabla)<\infty$ then $\slope(\mu^\nabla)$ is defined (i.e., the expectation in~\eqref{eq:slope-def} exists). Conversely,~\cite[Conjecture 10.3.2]{Sheffield-Random-Surfaces} conjectures that if $\mu^\nabla$ is ergodic and $\slope(\mu^\nabla)\in \mathfrak{S}^\circ$ then $\SFE(\mu^\nabla)<\infty$ (the restriction to the interior of $\mathfrak{S}$ may not be relevant in our integer-valued setup). Additionally,~\cite[Conjecture 10.3.1]{Sheffield-Random-Surfaces} conjectures that if $\mu^\nabla$ is ergodic, $\SFE(\mu^\nabla)<\infty$ 
and $\mathfrak{s}=\slope(\mu^\nabla)\in\mathfrak{S}^\circ$ then $\mu^\nabla$ is minimal. For Lipschitz height functions (in our integer-valued setup), or when the interactions $\V$ are isotropic, this is proved as part of the variational principle~\cite[Theorem 6.3.1]{Sheffield-Random-Surfaces}.

We are not aware of results on which interactions $\V$ (if any) admit an ergodic gradient Gibbs measure $\mu^\nabla$ with $\SFE(\mu^\nabla)=\infty$. 
To highlight the subtlety of the issue, we briefly leave the scope of integer-valued height functions and consider the real-valued Gaussian free field 
and ask whether the model possesses an ergodic Gibbs measure with infinite specific free energy. In dimensions $d=1,2$ the model is rough (has no Gibbs measures at all), but in $d\ge 3$ the question is equivalent to the existence of non-trivial ergodic harmonic functions on $\Z^d$, and is still open.\footnote{To see the equivalence, let~$\mu^0$ be the infinite-volume Gibbs measure obtained in the limit with zero boundary conditions (known to exist when $d\ge 3$ and to have finite specific free energy). Observe that for every (deterministic) \emph{harmonic} function $m:\Z^d\to\mathbb{R}$ (i.e., $m_v = \frac{1}{2d}\sum_{u\colon u\sim v}m_u$ at every $v\in\Z^d$) there exists a Gibbs measure $\mu^m$, obtained as the distribution of $\phi + m$ when $\phi$ is sampled from $\mu^0$. In fact, these $\mu^m$ are exactly the extremal (i.e., tail trivial) Gibbs measures of the model~\cite[Theorem 13.24 and (13.29)]{georgii2011gibbs}. One may then verify that every ergodic Gibbs measure has the form $\mu^m$ where $m$ is now a random harmonic function, with ergodic distribution under $\Z^d$ shifts. For an ergodic harmonic function $m$, if $\mathbb{E}[(m_u-m_v)^2]<\infty$ for all $u\sim v$ then $m$ is almost surely (a deterministic) constant (one way to see this is to note that for each $1\le i\le d$, $(m_{X_n}-m_{X_n+e_i})_n$ is an $L^2$-bounded martingale, and thus almost surely convergent, when $(X_n)$ is a simple random walk. Then couple two such walks with different starting points). One may then conclude that the model admits an ergodic Gibbs measure $\mu$ with $\SFE(\mu)=\infty$ if and only if there exists an ergodic harmonic function which is not almost surely constant.} It has been proved that such harmonic functions do not exist in two dimensions~\cite[Appendix B]{buhovsky2022discrete} (see also~\cite{bou2025unique}), but it remains open whether they exist in dimensions $d\ge 3$ (some related continuum objects have been shown to exist, even in two dimensions~\cite{MR1422707,tsirelson2007divergence, buhovsky2019translation, glucksam2019measurably, glucksam2025measurable}).

\subsection{Ergodic minimal gradient Gibbs measures}\label{subsec:ergodic-minimal-grad-meas}

We recap the results of~\cite{Sheffield-Random-Surfaces} on ergodic minimal gradient Gibbs measures of the $\nabla\phi$ model 
 for convex translation-invariant $\mathbf{V}$. 

 \medskip 
 \noindent \emph{Sloped-periodic gradient measures}: For a slope $\mathfrak{s}$, in Section~4.2 of~\cite{Sheffield-Random-Surfaces}, the following construction was made of the aforementioned sloped-periodic gradient measure of slope $\mathfrak{s}$ on the finite-volume $L\times\cdots\times L$ torus $\mathbb T_L^d$.  Fix a fundamental domain $\Lambda_{L} = \{1,...,L\}^d \subset \mathbb Z^d$ and let $$\mathcal E_L^+ = \{(v,v+\mathfrak{e}_i): v\in \Lambda_L, 1\le i\le d\}$$ be the set of edges in $\mathbb Z^d$ consisting of those internal to $ \{1,...,L\}^d$ together with those on its outer boundary and in the upper orthant (to not double count). Extend a gradient function $\phi^\nabla$ from the fundamental domain to $\mathbb Z^d$ as follows: for $w\in \mathbb Z^d$ with $w = v + (n_1 L,...,n_d L)$ for integer $n_i$ and $v$ in the fundamental domain, define the approximation $\mathfrak{s}^{(L)} = \frac{1}{L} \lfloor \mathfrak{s} L\rfloor$
\begin{align}\label{eq:extend-torus-configuration}
    \phi (w) - \phi(v)  = \langle \mathfrak{s}^{(L)},(n_1,...,n_d)\rangle  L
\end{align}
The sloped-periodic gradient measure on $\T^d_L$ is then defined as 
\begin{align}\label{eq:torus-measure-with-slope}
    \mu_{L,\mathfrak{s}}^{\nabla}(\phi) \propto \prod_{e \in \mathcal E_L^+} e^{ -V_e(\nabla_e \phi)}\,.
\end{align}

By the convexity of $\mathbf{V}$, see e.g. ~\cite[Lemma 4.2.6]{Sheffield-Random-Surfaces}, subsequential limits as $L\to\infty$ of  $\mu_{L,\mathfrak{s}}^\nabla$ exist, have finite specific free energy, and are $\mathbb Z^d$ translation-invariant.  Furthermore, Corollary 6.1.5 of~\cite{Sheffield-Random-Surfaces} demonstrates that any such subsequential limit $\mu^\nabla$ will be minimal. In fact, \cite[Lemmas~8.2.7--8.2.8]{Sheffield-Random-Surfaces} use convexity of the surface tension to show that almost every ergodic component of such a subsequential limit will have slope $\mathfrak{s}$ and have minimal free energy. 

We remark that, as a consequence of Theorem~\ref{thm:main-localization}, under~\eqref{eq:main assumption}, the infinite-volume limit of the sloped-periodic construction exists (without taking a subsequence) and is ergodic.

\medskip
\noindent \emph{Structure of infinite-volume gradient measures}.
 Given the above, one could aim to classify the ergodic gradient Gibbs measures with slope $\mathfrak{s}$ having minimal SFE by characterizing their extremal components. This was done in Theorems 8.7.9 and 8.7.8 of~\cite{Sheffield-Random-Surfaces} under the assumption that the slope $\mathfrak{s}$ is localized. We reproduce those results here for the benefit of the reader. The extremal measures in the gradient of a height function measure were characterized in~\cite{Sheffield-Random-Surfaces} by the value of their \emph{height offset spectrum}, which can be expressed as
 $$
 h(\phi) = \lim_{L \to \infty} |\Lambda_L|^{-1} \sum_{x \in \Lambda_L} \phi_x. \footnote{Note, that~\cite{Sheffield-Random-Surfaces} uses $\liminf$ in the definition of the height offset spectrum in Theorem 8.7.3, but the proof uses the ergodic theorem to argue that the $\liminf$ can actually be reduced to $\lim$.}
 $$

 \begin{thm}[Theorem 8.7.9  and Theorem 8.7.8 of \cite{Sheffield-Random-Surfaces}]\label{thm:Sheffield}
Suppose $\mathfrak{s}\in \mathfrak{S}^\circ$ is localized, and that $\mu_{\mathfrak{s}}$ is a Gibbs measure with minimal, ergodic gradient of slope $\mathfrak{s}$. There are now two cases to consider. 

\smallskip
\noindent \textbf{Case 1: All coordinates of $\mathfrak{s}$ are rational.}  
Writing $\mathfrak{s} = p_i/q_i$ for relatively prime integers $p_i,q_i$, let $n$ be the least common multiple of $q_i$. 
There exists a $\mathfrak{c}$ and a unique family $(\mu_{\mathfrak{s},a})_a$ of extremal Gibbs measures such that: 
\begin{enumerate}
\item \textbf{Height Offset Property:} For each $a \in \mathfrak{c} + \frac{1}{n} \mathbb{Z}$, we have that for $\mu_{\mathfrak{s},a}$ almost all $\phi$, the height offset variable $h(\phi) =a$.
\item \textbf{Stochastic Domination Property:} $\mu_{\mathfrak{s},a_1}$ stochastically dominates $\mu_{\mathfrak{s},a_2}$ when $a_1 \ge a_2$. 

\item \textbf{Height Translation Property:} For  each $b \in \mathbb{Z}$, we have $b + \mu_{\mathfrak{s},a} = \mu_{\mathfrak{s},a+b}$. Here, adding an integer value $b$ to a measure $\mu$ means the pushforward of the measure $\mu$ under adding the integer $b$ pointwise to configurations.

\item \textbf{Decomposition Property:} The gradient measure associated to $\frac{1}{n}\sum_{j =0}^{n-1} \mu_{\mathfrak{s}, a+ \frac{j}{n}}$ is the same as the gradient measure $\mu_{\mathfrak{s}}^\nabla$ for any $a \in \mathfrak{c}+\frac{1}{n} \mathbb{Z}$. 

\item \textbf{Extremality:} For all $a \in \mathfrak{c}+\frac{1}{n} \mathbb{Z}$, $\mu_{\mathfrak{s},a}$ is extremal.

\item \textbf{Spatial Translation Property:} For each $x \in \mathbb{Z}^d$ and each $a$, $\theta_x \mu_{\mathfrak{s},a} = \mu_{\mathfrak{s}, a+ (\mathfrak{s},x)}$.

\end{enumerate}
\textbf{Case 2: $\mathfrak{s}$ has at least one irrational coordinate.}
In this case, we have to make the following adjustments beyond the obvious replacement of each appearance of $a \in \mathfrak{c}+ \frac{1}{n} \mathbb{Z}$ with $a \in \mathbb{R}$. We remark further that in this case there is a unique ergodic component up to translation.
\begin{itemize}

\item \textbf{Decomposition Property:} The gradient measure associated to $\int_0^1 \mu_{\mathfrak{s},\mathfrak{c} +a} \text{d}a$ is the same as the gradient measure $\mu_{\mathfrak{s}}^\nabla$ for any $\mathfrak{c} \in \mathbb{R}$.

\item  \textbf{Right-continuity}: For each increasing event $A$, $\mu_{\mathfrak{s},a}(A)$ is increasing and right-continuous in $a$. (Note that this is unique to the case that $\mathfrak{s}$ has at least one irrational coordinate.)
\end{itemize}
\end{thm}

After we have established that the limits of sloped-periodic measures have dominant heights and all slopes $\mathfrak{s}$ satisfying~\eqref{eq:main assumption} are localized, the dominant heights will be related to the height offset spectrum to apply Theorem~\ref{thm:Sheffield} and identify the typical dominant height sequences.

\section{Reflection positivity and rarity of non-zero horizontal gradients}\label{sec:reflection-positivity}

For the rest of the paper, without loss of generality we take $i_0, i_1$ in~\eqref{eq:main assumption} to be $1$ and $2$ respectively. 
The main result of this section will be that in the $\mathfrak{e}_1,\mathfrak{e}_2$-directions (in which the slope is flat and the interaction strength is large), non-zero gradients are unlikely. 

We will work in this section with a slope vector $\mathfrak{s}= (0,0,s_3,...,s_d)$ and the sloped-periodic gradient measure $\mu_{L,\mathfrak{s}}^\nabla$ on $\phi^\nabla$ which are elements in the equivalence class of $\{\phi:\mathbb T_L^d\to \mathbb Z\}$ modulo global shifts, with energy function from~\eqref{eq:height function model} as defined in~\eqref{eq:extend-torus-configuration}--\eqref{eq:torus-measure-with-slope}. 
To avoid cumbersome notation, we will talk about height configurations $\phi$ and define maps between them but these will be understood as maps on the gradients in the obvious way (e.g., only thinking about the representatives of each equivalence class which take height zero at the origin). 

We say an edge $e\in \mathbb T_L^d$ is $\mathfrak{e}_i$-oriented if it is parallel to the basis vector $\mathfrak{e}_i$. When evaluating gradients along $\mathfrak{e}_i$-oriented edges, as stated in the introduction, they are always evaluated in the direction of $\mathfrak{e}_i$. (Note that since $V_1,V_2$ are assumed to be even, this does not matter for $\{\mathfrak{e}_1,\mathfrak{e}_2\}$-oriented edges, but does for $\mathfrak{e}_3,...,\mathfrak{e}_d$.) 
Throughout the remainder of the paper, we assume and may not restate, that $\mathbf{V}$ satisfies~\eqref{eq:potential assumption}, is even in its first two coordinates, and without loss of generality, is shifted so that $\min_{k} V_i(k)=0$ for all $1\le  i\le d$.
\begin{thm}\label{thm:1-2-direction-rigidity}
    Fix $d\ge 3$ and suppose $L$ is even.\footnote{The same proofs go through in $d=2$, but most of the arguments degenerate to simpler Peierls bounds, so the presentation from now on will assume $d\ge 3$ so definitions are not vacuous.} There exist constants $c,\beta_0>0$ such that if  $\min_{i=1,2} V_i(1)\ge \beta_0$ the following holds. 
    Fix $\mathfrak{s} = (0,0,s_3,...,s_d)$ and consider an $\mathfrak{e}_1$ or $\mathfrak{e}_2$ oriented edge $e$; for all $r\ge 1$, 
    \begin{align*}
        \mu_{L,\mathfrak{s}}^\nabla (|\nabla_e \phi| \ge r) \le e^{ - c \min_{i=1,2} V_i(1) r }\,.
    \end{align*}
    More generally, for a family of distinct $\{\mathfrak{e}_1,\mathfrak{e}_2\}$-oriented edges $e_1,...,e_m$, and integers $r_1,...,r_m\ge 1$, 
    \begin{align*}
        \mu_{L,\mathfrak{s}}^\nabla \Big( \bigcap_{i=1}^m \{|\nabla_{e_i} \phi | \ge r_i\} \Big) \le e^{ - c \min_{i=1,2} V_i(1) \sum_{i=1}^m r_i}\,.
    \end{align*}
\end{thm}

To preview, in Section~\ref{sec:limits}, Theorem~\ref{thm:1-2-direction-rigidity} will be combined with a Peierls argument in the two-dimensional $\text{Span}\{\mathfrak{e}_1,\mathfrak{e}_2\}$, to establish that in every layer $\mathbb T_L^2 \times \{v\}$ for $v\in \mathbb T_{L}^{d-2}$, there is a dominant height phase (most vertices in that slice are at the same height, and the set of vertices not at that height behave like a sub-critical percolation process).

Our proof of Theorem~\ref{thm:1-2-direction-rigidity} will use the  \emph{reflection positivity} of the sloped-periodic gradient measure in the $\mathfrak{e}_1,\mathfrak{e}_2$-directions where the slope is zero and the interaction is even. We refer to the $\mathfrak{e}_1,\mathfrak{e}_2$ directions as the ``horizontal" directions, and $\mathfrak{e}_3,...,\mathfrak{e}_d$ as vertical. We describe the reflection positivity framework in more detail in Subsection~\ref{subsec:reflection-positivity-primer}, but its upshot will be that the work in the proof of Theorem~\ref{thm:1-2-direction-rigidity} is bounding the probability of height configurations where the event $\{\nabla_{\mathfrak{e}_1}\phi =k\}$ has been reflected horizontally across the torus.

\subsection{Peierls-type maps within prisms}

Our aim in this section is to define an appropriate event, which contains the event in Theorem~\ref{thm:1-2-direction-rigidity} and that when reflected horizontally, the probability of the reflected event can be bounded. Let us begin by introducing some geometric notions. 

\begin{defn}[Prisms and slices]
We begin by defining a base domain, which is the $2\times 2$ cell $\{0,1\}^2$ cross product with the $d-2$ other dimensions, called the \emph{base prism}: $\mathcal P = \{0,1\}^2 \times \mathbb T^{d-2}_L$. This prism is naturally embedded in $\mathbb T^d_L$. 

We refer to $\mathcal S_0 = \{(0,0),(1,0)\}\times \mathbb T^{d-2}$ or $\mathcal S_1 = \{(0,1),(1,1)\}\times \mathbb T^{d-2}$ as the two $\mathfrak{e}_1$-oriented slices of $\mathcal P$. 
\end{defn}

For a configuration $\sigma$ on $\mathcal P$, let $E(\mathcal P)$ be the set of edges fully internal to $\mathcal P$, and denote by $H_{\mathcal P}$ the sub-sum of the Hamiltonian in~\eqref{eq:extend-torus-configuration}--\eqref{eq:torus-measure-with-slope} where only contributions from edges in $E(\mathcal P)$ are counted. 

For the set of $\mathfrak{e}_1$-oriented edges, $E_1(\mathcal P)$, define the adjacency structure where two edges $e,e'\in E_1(\mathcal P)$ are adjacent if their midpoints are at distance $1$ from one another---equivalently, we are identifying $E_1(\mathcal P)$ via its midpoints, located at $\{\frac{1}{2}\}\times \{0,1\}\times \mathbb T^{d-2}_L$, with the graph $\{0,1\}\times \mathbb T^{d-2}_L$ with its natural notion of adjacency. 
Let $o= (0,0,...,0)$ be the origin vertex. 

\begin{defn}[Signed gradient cluster]\label{def:signed-gradient-cluster}
    For a configuration $\phi$ with $\nabla_{\{o,\mathfrak{e}_1\}}\phi \ne 0$, define the signed gradient cluster $\mathcal C_{\{o,\mathfrak{e}_1\}}(\phi)$ as the maximal $E_1(\mathcal P)$-connected component of 
    \begin{align*}
        \{ f\in E_1(\mathcal P): \text{sgn}( \nabla_{f}\phi) = \text{sgn}(\nabla_{\{o,\mathfrak{e}_1\}}\phi)\}
    \end{align*}
    (The sign of zero is taken to be zero to distinguish it.)
\end{defn}

Our main step to establishing Theorem~\ref{thm:1-2-direction-rigidity} is the following lemma showing the existence of a Peierls-type map on configurations on $\mathcal P$ that flattens a non-zero gradient on an $\mathfrak{e}_1,\mathfrak{e}_2$-oriented edge $e$ (along with all the  gradients in its signed gradient cluster). By translation invariance and rotation symmetry, it suffices to consider the specific choice $e = \{o,\mathfrak{e}_1\}$.

\begin{lem}\label{lem:Peierls-criteria-within-prism}
    There
    exists a map $\Phi$ between configurations in $\mathcal P$, and  a constant $C(d)>0$ such that for any $\mathbf{V}$ satisfying~\eqref{eq:potential assumption}, 
    \begin{itemize}
        \item For any $\phi$ with $\nabla_{e}\phi\ne 0$, one has 
        \begin{align*}
            H_{\mathcal P}(\phi) - H_{\mathcal P}(\Phi(\phi)) \ge \sum_{f\in E_1(\mathcal C_e)} V_1(\nabla_f\phi) \ge V_1(1) \cdot |\mathcal C_e(\phi)|\,.
        \end{align*}
        \item For any $k\ge 1$, the number of pre-images of a $\phi'\in \text{Im}(\Phi)$, satisfies 
        \begin{align*}
            \{\phi: \Phi(\phi) = \phi'\,,\, H_{\mathcal P}(\phi) - H_{\mathcal P}(\phi') \le V_1(1) \cdot k\} \le C^k\,.
        \end{align*}
    \end{itemize}
\end{lem}

    The construction of the map will go in two phases. By symmetry, it suffices to define the map for the case where $\nabla_{e} \phi>0$. The first step will flatten the gradients in $\mathcal C_e(\phi)$ to lie in $\{0,1\}$, and the second step will send the remaining $1$ gradients in the cluster to $0$. Let us define each of these stages of the map before analyzing it. 

    \begin{defn}\label{def:single-edge-map}
        For a configuration $\phi$ with $\nabla_e \phi >0$, 
    \begin{itemize}  
        \item \emph{Stage 1}: for all $\mathfrak{e}_1$-oriented edges $f= \{w,w+\mathfrak{e}_1\}$ internal to $\mathcal C_e(\phi)$, in $\Phi_1(\phi)$, set 
        $$ \Phi_1(\phi)_w   = \lfloor \bar \phi_f\rfloor  \qquad\text{and} \qquad \Phi_1(\phi)_{w+\mathfrak{e}_1} = \lceil \bar \phi_f\rceil \quad \text{where} \quad \bar \phi_f = \frac{1}{2} (\phi_w + \phi_{w+\mathfrak{e}_1})$$
        (Note that this changes the gradient to zero if $\bar \phi_f \in \Z$ and to $1$ otherwise.)
        \item \emph{Stage 2}: In the second stage, there are two options for the operation to perform: 
        \begin{itemize}
            \item $\Phi^{2,-}$ takes all $\mathfrak{e}_1$-oriented edges internal to $\mathcal C_e(\phi)$ with gradient exactly equal to $1$ and lowers the height on all right-vertices by one;
            \item $\Phi^{2,+}$ takes all such edges and raises the height on all left-vertices by one.
        \end{itemize} 
        Let $\Phi_2$ be the choice among $\{\Phi^{2,-},\Phi^{2,+}\}$ which lowers the energy more (and if there is a tie, pick arbitrarily).  
    \end{itemize}
    Let $\Phi = \Phi_2 \circ\Phi_1$ be the composition of the two operations.  
    \end{defn}

    We will establish that the map $\Phi$ constructed in Definition~\ref{def:single-edge-map} satisfies Lemma~\ref{lem:Peierls-criteria-within-prism}. 

\smallskip

\begin{proof}[\textbf{\emph{Energy change}}]
Fix $e = \{0,\mathfrak{e}_1\}$, configuration $\phi$ having positive gradient on $e$, and let $\mathcal C_e = \mathcal C_e(\phi)$ be as in Definition~\ref{def:signed-gradient-cluster}. Let $\mathcal E_{1}^{\ge 2}$ be those $\mathfrak{e}_1$-oriented edges internal to $\mathcal C_e$ whose gradient is at least $2$. Let $\mathcal E_{\perp}$ be those edges of $\mathcal P$ that are in directions $\mathfrak{e}_2,...,\mathfrak{e}_d$. These split either into those in the $\mathfrak{e}_2$-direction, or they are in the $(\mathfrak{e}_i)_{i\ge 3}$ directions in which case they can be decomposed according to whether their first two coordinates are $(0,0)$ at both endpoints, or both $(0,1),(1,0)$ or $(1,1)$. We compute the change in energy across these different types of edges, in stage 1 and stage 2 sequentially. 

\medskip
\noindent \textbf{Interactions along $\mathfrak{e}_1$-oriented edges}

\medskip
\noindent    \emph{Analysis of Stage 1}: 
For all $\mathfrak{e}_1$ oriented edges $f \in E_1(\mathcal C_e)$, if the average $\bar \phi_f$  was an integer, its gradient became~$0$. Since $f$ was internal to the signed gradient cluster, prior to application of Stage 1, $\nabla_f \phi>0$. The contribution to $H_{\mathcal P}(\phi) - H_{\mathcal P}(\Phi_1(\phi))$ from edges in $E_1(\mathcal C_e)$ with even gradient is  
\begin{align*}
    \sum_{f\in E_1(\mathcal C_e): \nabla_f \phi \in 2\mathbb N} V_1(\nabla_f \phi)\,.
\end{align*}
For $\mathfrak{e}_1$-oriented edges where the gradient along them was odd, in stage $1$, $\Phi_1$ decreases them to have gradient exactly $1$, and therefore the energy change from them is exactly  
\begin{align*}
    \sum_{f\in E_1(\mathcal C_e): \nabla_f\phi \in 2\mathbb N +1} \big(V_1(\nabla_f\phi) - V_1(1)\big)\,.
\end{align*}
Finally there is no contribution from $\mathfrak{e}_1$-oriented edges not in $E_1(\mathcal C_e)$, nor those in $E_1(\mathcal C_e)$ that have exactly gradient $1$, as they do not get changed by $\Phi_1$.

\medskip
\noindent  \emph{Analysis of Stage 2}: In this stage, none of the edges $f\in E_1(\mathcal C_e)$ whose value was initially even are changed, as their gradient was already made zero by $\Phi_1$. Those edges in $E_1(\mathcal C_e)$ whose value was initially odd all have their gradients changed from $1$ to $0$ by $\Phi_2$. 

\medskip
Combining the above analyses, we get that the total contribution from $\mathfrak{e}_1$-oriented edges in $E(\mathcal P)$ from application of the map $\Phi$ satisfies 
\begin{align}
    \sum_{f\in E_1(\mathcal P)} V_1(\nabla_f \phi) - V_1(\nabla_f \Phi(\phi)) & = \sum_{f\in E_1(\mathcal C_e)} V_1(\nabla_f\phi) \nonumber \\
    & \ge V_1(1) \sum_{f\in E_1(\mathcal C_e)} |\nabla_f\phi| \ge V_1(1) \cdot|E_1(\mathcal C_e(\phi))|\label{eq:e_1-oriented-energy-contribution}
\end{align}
(Here, the equality used $V(0)= 0$, the first inequality follows by evenness and convexity, and the last inequality by the fact that $|\nabla_f\phi|\ge 1$ for all $f\in E_1(\mathcal C_e)$ by definition of the signed gradient cluster.)

\medskip
\noindent 
\textbf{Interactions along $\mathfrak{e}_2,...,\mathfrak{e}_d$ edges}

Recall $E_{\perp}(\mathcal C_e)$ denotes the set of non-$\mathfrak{e}_1$-oriented edges in the prism $\mathcal P$, i.e., those that are in the direction of the basis vectors $\mathfrak{e}_j$ for $j\in \{2,...,d\}$. We call such edges \emph{perpendicular}. 

We begin with an important inequality that we will use repeatedly, using the convexity of the interactions: if $V$ is a convex function (extended piecewise linearly from the integers to the reals in its support) and $|b_1|\le |b_2|$, for any $a$, 
\begin{align}\label{eq:convexity-inequality}
    V(a +b_1) + V(a-b_1) \le V(a+b_2) + V(a- b_2)\,.
\end{align}
To see~\eqref{eq:convexity-inequality}, note that since convexity of $V$ requires that $\{k:V(k)<\infty\}$ be an interval, the inclusion $[a-|b_1|,a+|b_1|]\subset [a-|b_2|,a+|b_2|]$ means the left-hand side is finite whenever the right-hand side is finite so the inequality holds when one of the values is $\infty$. Otherwise, if all of $a\pm b_1,a\pm b_2$ are in $\{k: V(k)<\infty\}$, then it follows by writing $a+|b_1| = \lambda (a-|b_2|) + (1-\lambda) (a+|b_2|)$ and similarly $a-|b_1| = (1-\lambda)(a-|b_2|) + \lambda (a+|b_2|)$ for a $\lambda\in [0,1]$, then using the standard definition of convexity. 

The key role played by~\eqref{eq:convexity-inequality} will be when analyzing the energy change from $\Phi$ shifting heights on a pair of parallel edges. To describe that, let $f,g$ be a pair of parallel $\mathfrak{e}_1$ oriented edges with midpoints at distance one, i.e., $g= f+ \mathfrak{e}_j$ for some $j\in \{2,...,d\}$. Let $\Delta_f$ be half the gradient along the edge $f$, so that the height at the left endpoint of $f$ is $\bar \phi_f - \Delta_f$ and the height at the right endpoint is $\bar \phi_f +  \Delta_f$.  Suppose that $\Phi$ sends these values to $\bar \phi_f - \Delta_f'$ and $\bar \phi_f + \Delta_f'$ respectively, and similarly for $g$. 

Then the new energy along the pair of $\mathfrak{e}_j$-oriented edges connecting endpoints of $f$ and $g$ is
\begin{align}\label{eq:key-convexity-inequality-applied}
    &V_j((\bar \phi_f - \Delta_f') - (\bar \phi_g - \Delta_g')) + V_j((\bar \phi_f + \Delta_f') - (\bar \phi_g + \Delta_g')) \\
    & \qquad \qquad \qquad \le V_j((\bar \phi_f -  \Delta_f) - (\bar \phi_g - \Delta_g)) + V_j((\bar \phi_f +\Delta_f) - (\bar \phi_g + \Delta_g)) \nonumber
\end{align}
where the inequality follows from~\eqref{eq:convexity-inequality} by setting \begin{align}\label{eq:choices-of-a-b1-b2}
a = \bar \phi_f - \bar \phi_g \qquad b_1 = \Delta_g' - \Delta_f' \qquad b_2 =  \Delta_g -\Delta_f\end{align} as long as $|b_1|\le |b_2|$

\medskip
\noindent \emph{Analysis of Stage 1}: 
We first consider those perpendicular edges that are fully interior to $\mathcal C_e$, i.e., for a $w$ having first coordinate $0$ they are of the form $\{w,w'\}$  or $\{w+ \mathfrak{e}_1,w'+ \mathfrak{e}_1\}$, for $w' = w+\mathfrak{e}_j$. 

In these cases, let $f\in E_1(\mathcal C_e)$ be the edge $\{w,w+\mathfrak{e}_1\}$ and $f+\mathfrak{e}_j \in E_1 (\mathcal C_e)$. Then in $\phi$, the energy contribution coming from the pair of perpendicular edges $\{w,w+\mathfrak{e}_j\}$ and $\{w+\mathfrak{e}_1,w+\mathfrak{e}_1 +\mathfrak{e}_j\}$ is given by  
\begin{align*}
    V_j((\bar\phi_{f+\mathfrak{e}_j} + \Delta_{f+\mathfrak{e}_j}) - (\bar\phi_f + \Delta_f)  ) + V_j((\bar\phi_{f+ \mathfrak{e}_j} - \Delta_{f+\mathfrak{e}_j}) - (\bar\phi_f - \Delta_f) )\,.
\end{align*}
(Note in the event that the perpendicular edge under consideration is the one ``wrapping around" the torus in its fundamental domain, then by~\eqref{eq:extend-torus-configuration} in the evaluation of these energies, $\bar \phi_{f+\mathfrak{e}_j} \mapsto \bar \phi_{f+\mathfrak{e}_j} + s_j L$, but the gradients are unchanged and the task below remains the same.) 
After an application of $\Phi_1$, the four energy values change into $\bar \phi_{f + \mathfrak{e_j}} - \Delta_{f+\mathfrak{e}_j}', \bar \phi_f  - \Delta_f'$ and $\bar \phi_{f + \mathfrak{e}_j} + \Delta_{f+\mathfrak{e}_j}', \bar \phi_f + \Delta_{f}'$. We claim that $|\Delta_f'- \Delta_{f+\mathfrak{e}_j}'|\le |\Delta_f - \Delta_{f+\mathfrak{e}_j}|$ to apply~\eqref{eq:key-convexity-inequality-applied}. 
Indeed, there are a few cases to consider: 
\begin{itemize}
    \item If both gradients were initially even, then after $\Phi_1$, $\Delta_{f}',\Delta_{f+\mathfrak{e}_j}' =0$ so $|\Delta_f'- \Delta_{f+\mathfrak{e}_j}'|=0$. 
    \item If both gradients were initially odd, then they were both reduced to $1$ by $\Phi_1$, so $\Delta_{f}' = \Delta_{f+\mathfrak{e}_j}' = \frac{1}{2}$ and $|\Delta_f'- \Delta_{f+\mathfrak{e}_j}'|=0$. 
    \item If the gradients had different parities, then one of them was reduced to $0$ while the other was reduced to $1$, so they have $|\Delta_f' - \Delta_{f+\mathfrak{e}_j}'| =\frac{1}{2}$ but because the gradients had different parities, their initial difference satisfies $|\Delta_f - \Delta_{f+\mathfrak{e}_j}| \ge \frac{1}{2}$. 
\end{itemize}
In all of these cases, $|\Delta_f'- \Delta_{f+\mathfrak{e}_j}'|\le |\Delta_f - \Delta_{f+\mathfrak{e}_j}|$ holds and we can apply the convexity consequence of~\eqref{eq:key-convexity-inequality-applied}. 
Therefore, the contribution to the energy change from $\Phi_1$ along $\mathfrak{e}_j$-oriented edges for $j\in \{2,...,d\}$ where the endpoints are both in $\mathcal C_e$ is that it does not increase the energy. 

By a similar reasoning, we can consider the contributions of perpendicular edges along the boundary of $\mathcal C_e$, meaning that one endpoint is in $\mathcal C_e$ while the other is not. For these, initially, $\Delta_f>0$ because $f\in E_1(\mathcal C_e)$ while $\Delta_{f+\mathfrak{e}_j}\le 0$ as $f+ \mathfrak{e}_j \notin E_1 (\mathcal C_e)$. Then since after application of $\Phi_1$, the gradient $\Delta_f'$ is only less than or equal to $\Delta_f$ and remains non-negative, while the external gradient $\Delta_{f+\mathfrak{e}_j}'$ remains unchanged by $\Phi_1$, evidently $|\Delta_f' - \Delta_{f+\mathfrak{e}_j}'|\le |\Delta_f -\Delta_{f+\mathfrak{e}_j}|$.  

(Finally, if both $f$ and $f+\mathfrak{e}_j$ are outside $\mathcal C_e$ then the gradients along both are unchanged by $\Phi_1$.)
In total from these considerations we see that 
\begin{align}\label{eq:stage-1-perpendicular-edges}
    & \sum_{f\in E_\perp(\mathcal C_e)} V_f(\nabla_f \phi) - V_f(\nabla_f \Phi_1( \phi)) \nonumber \\ & \qquad =  \sum_{j=2}^{d}\sum_{\{w,w+\mathfrak{e}_1\} \in E_1(\mathcal C_e): w_1 =0}  \Big[V_j(\nabla_{\{w,w+\mathfrak{e}_j\}} \phi) + V_j(\nabla_{\{w+ \mathfrak{e}_1, w+ \mathfrak{e}_1+\mathfrak{e}_j\}} \phi) \nonumber \\ & \qquad \qquad \qquad \qquad \qquad \qquad \qquad  -   V_j(\nabla_{\{w,w+\mathfrak{e}_j\}} \Phi_1(\phi)) - V_j(\nabla_{\{w+ \mathfrak{e}_1, w+ \mathfrak{e}_1+\mathfrak{e}_j\}} \Phi_1(\phi)) \Big] \nonumber \\ & \qquad\ge 0\,.
\end{align}

\medskip
\noindent \emph{Analysis of Stage 2}: The remaining step is to argue that one of the two stage-2 options $\Phi^{2,-}$ or $\Phi^{2,+}$ for flattening the odd gradients in $\mathcal C_e$ to $0$ is such that it only decreases the energy contribution from perpendicular edges. We only have to consider the perpendicular edges that are on the boundary of $\mathcal C_e$ because all $\mathfrak{e}_1$-oriented edge gradients in $E_1(\mathcal C_e)$ will become $0$ after application of $\Phi_2$, and so if $f,f+\mathfrak{e}_j$ are internal edges to $\mathcal C_e$, then in applying~\eqref{eq:key-convexity-inequality-applied}, the quantities $\Delta'_f,\Delta_{f+\mathfrak{e}_j}'$ will both be $0$ and trivially $|\Delta_f' - \Delta_{f+\mathfrak{e}_j}'| \le |\Delta_f - \Delta_{f + \mathfrak{e}_j}|$. 

Let $\tilde\phi = \Phi_1(\phi)$ be the output of stage 1 and consider a pair of edges $f= \{w,w+\mathfrak{e}_1\}$ and $f+\mathfrak{e}_j$ with $f\in E_1(\mathcal C_e)$ but $f+\mathfrak{e}_j \notin E_1(\mathcal C_e)$ (the reverse case is symmetrical). If $\nabla_{f}\tilde \phi=0$, then all four endpoints of $f$ and $f+\mathfrak{e}_j$ will be unchanged by $\Phi_2$ so there is no change in energy along the edges $\{w,w+\mathfrak{e}_j\}$ and $\{w+\mathfrak{e}_1,w+\mathfrak{e}_1 + \mathfrak{e}_j\}$. 
Now assume that $\nabla_f \tilde \phi =1$  while $\nabla_{f+\mathfrak{e}_j} \phi = \nabla_{f+\mathfrak{e}_j} \tilde \phi \le 0$. By translating, we can suppose that $\tilde \phi_w =0, \tilde \phi_{w+\mathfrak{e}_1} = 1$, and let $\tilde \phi_{w+\mathfrak{e}_j}=A$ and $\tilde \phi_{w+\mathfrak{e}_1 + \mathfrak{e}_j} = A+2\Delta:=A+ \nabla_{f+\mathfrak{e}_j} \tilde \phi$. Under $\tilde \phi$, the energy in the pair of edges $\{w,w+\mathfrak{e}_j\}$ and $\{w+\mathfrak{e}_1, w+\mathfrak{e}_1 + \mathfrak{e}_j\}$ is 
\begin{align*}
    V_j(A) + V_j(A + 2\Delta -1)
\end{align*}
Under application of $\Phi^{2,-}$ it becomes 
\begin{align*}
    V_j(A) + V_j(A+2\Delta)
\end{align*}
while under application of $\Phi^{2,+}$ it becomes 
\begin{align*}
    V_j(A- 1) + V_j(A+2\Delta -1)\,.
\end{align*}
Taking the average of the energies after $\Phi^{2,-}$ and $\Phi^{2,+}$, we compare 
\begin{align*}
    V_j(A) + V_j(A+2\Delta-1) \qquad \text{to} \qquad \frac{1}{2} \big( V_j(A) + V_j(A+2\Delta) + V_j(A -1) + V_j(A+2\Delta-1)\big)
\end{align*}
By~\eqref{eq:convexity-inequality} with $a = A + \Delta - \frac{1}{2}$ and $b_2 = \Delta - \frac{1}{2}$ and $b_1 = \Delta + \frac{1}{2}$, the left is bigger than or equal to the right as long as $|2\Delta +1|\le |2\Delta-1|$, which holds true because $\Delta \le 0$.

If we sum the resulting inequality over all perpendicular boundary edges of $\mathcal C_e$, which we denote by $\partial_{\perp}\mathcal C$,  
\begin{align*}
    & \sum_{f\in \partial_{\perp} \mathcal C} V_j(A_f) + V_j(A_f+2\Delta_f -1) \\
    &  \qquad \qquad \qquad \ge \frac{1}{2} \Big(\sum  [V_j(A_f) + V_j(A_f+2\Delta_f)] +\sum [V_j(A_f-1) + V_j(A_f+2\Delta_f-1)] \Big)\,.
\end{align*}
Since the first sum on the right-hand side corresponds to the total energy from edges in $\partial_\perp \mathcal C$ after application of $\Phi^{2,-}$ and the second sum on the right-hand side corresponds to the total energy from edges in $\partial_\perp \mathcal C$ after application of $\Phi^{2,+}$, one of these two must be less than or equal to the initial energy along these perpendicular edges. By construction, that is the choice we make for the stage-2 map $\Phi_2$. 
As a consequence, we deduce the following: 
\begin{align}\label{eq:stage-2-perpendicular-edges}
    \sum_{f\in E_\perp(\mathcal P)} V_f(\nabla_{f} \Phi_1(\phi)) - V_f(\nabla_f \Phi_2(\Phi_1(\phi))_{f}) \ge 0\,.
\end{align}
Combined with the stage-1 effect on perpendicular edges~\eqref{eq:stage-1-perpendicular-edges}, and the analysis of the $\mathfrak{e}_1$-oriented edges from~\eqref{eq:e_1-oriented-energy-contribution}, we conclude item (1) of Lemma~\ref{lem:Peierls-criteria-within-prism}.   
\end{proof}

\begin{proof}[\textbf{\emph{Proof of the map's multiplicity}}]
    For the bound on the number of pre-images of a $\phi' \in \text{Im}(\Phi)$, we will construct a ``witness" object, using which the preimage of $\phi'$ can be uniquely reconstructed. We then bound the number of such witness objects. 

   \begin{defn}\label{def:witness} Our witness $\mathcal W = (A,\epsilon, W(A), \iota)$ will consist of 
    \begin{itemize}
        \item a connected set $A \subset \{\frac{1}{2}\} \times \{0,1\}\times \mathbb T_L^{d-2}$ containing $\frac{1}{2}\times \{0\}^{d-1}$;
        \item a sign $\epsilon\in \{-,+\}$ and an assignment  $(W_f)_{f\in E_1(A)}$ of positive integers to the $\mathfrak{e}_1$-oriented edges of $\mathcal P$ with midpoints at $A$;
        \item a sign $\iota \in \{-,+\}$ indicating left or right. 
    \end{itemize} 
    For each $\phi$, the witness $\mathcal W = \mathcal W(\phi)$ is uniquely defined by taking the connected set $A$ to be the (midpoints of the $\mathfrak{e}_1$ oriented edges in) signed gradient cluster $\mathcal C_e$, the sign $\epsilon$ to be the sign of the gradient at $e$, assignments of positive integers to $A$ to be the absolute gradients along those edges, and the sign $\iota$ to be the choice of which stage 2 map was used. 
    \end{defn}

    We now describe how to uniquely reconstruct $\phi$ starting only with $\phi'$ and $\mathcal W$. Suppose $\epsilon = +$; the case where the sign of the gradient on $e$ is negative is symmetrical. 
    \begin{enumerate}
        \item For each edge $f\in E_1(A)$, if $W_f$ is odd, then if $\iota = -$, increase the height of the right endpoint of $f$ by one, and if $\iota = +$ then decrease the height of the left endpoint of $f$ by one. 
        \item Next, for each edge $f\in E_1(A)$, if $W_f$ is even, then increase the right vertex height by $W_f/2$ and decrease the left vertex height by $W_f/2$. If $W_f$ is odd, increase the right vertex height by $(W_f -1)/2$ and decrease the left vertex height by $(W_f -1)/2$.   
    \end{enumerate}
    It is easily checked that these two operations invert $\Phi_2\circ \Phi_1$. 

    Because of the lower bound of part one of Lemma~\ref{lem:Peierls-criteria-within-prism}, for every preimage $\phi \in \Phi^{-1}(\phi')$, 
    \begin{align*}
        H_{\mathcal P}(\phi) - H_{\mathcal P}(\phi') \ge \sum_{f\in E_1(A)} V_1(W_f) \ge V_1(1) \sum_{f\in E_1(A)}  W_f
    \end{align*}
    for $\mathcal W= (A, \epsilon,W(A),\iota)$  the witness associated to $\phi$. Thus, necessarily we have  $\sum_{f\in E_1(A)} W_f \le V_1(1)^{-1} (H_{\mathcal P}(\phi) - H_{\mathcal P}(\phi'))$. 

    It is therefore sufficient to bound the number of possible witnesses in the following set $$\Big\{\mathcal W= (A,\epsilon,W(A),\iota): \sum_{f\in E_1(A)} W_f \le  k \Big\}\,.$$
    There are clearly $2$ choices for $\epsilon$ and $2$ choices for $\iota$. The number of possible sets $A$ is bounded by the number of connected vertex subsets of $\{0,1\} \times \mathbb T_L^{d-2}$ containing the origin, and having size at most $k$ (as every site in $A$ must contribute at least one non-zero gradient). The number of such $A$ is bounded by $C_{d}^{k}$ by standard connective constant bounds. 

    Finally, the number of possible $W = (W_f)_{f\in A}$ is bounded by the number of integer partitions of some $M \le V_1(1)^{-1} (H_{\mathcal P}(\phi) - H_{\mathcal P}(\phi'))\le k$. There are at most $V_1(1)^{-1}(H_{\mathcal P}(\phi) - H_{\mathcal P}(\phi'))$ many choices for $M$, and then having picked $M$, there are at most $2^{2M}$ many ways to partition $M$ into integers $m_1,...,m_{|A|}$ with $\sum m_i = M$ and each $m_i \ge 1$. 
\end{proof}

\subsection{Extending to multiple non-zero gradients in the slice}
We generalize the above map into the case where the event we are trying to bound is having several distinguished edges $f_1,...,f_\ell$ that are all in the same prism $\{0,1\}^2\times \T_L^{d-2}$, all having non-zero gradient. 

Consider a set of $\ell$ distinct $\mathfrak{e}_1$-oriented edges $\vec{f} = \{f_1,...,f_\ell\}$ all belonging to $E_1(\mathcal P)$. Let $\mathcal E_{\vec{f},\vec{\Delta}}$ be the event that along each of these edges, there is a non-zero gradient $(2\Delta_f)_{f\in \vec{f}}$.

\begin{defn}\label{def:Phi-multi-edges}
For any configuration $\phi$ on the prism $\mathcal{P}$ belonging to $\mathcal E_{\vec{f},\vec{\Delta}}$, let $\Phi_{\vec{f}}$ be the map that iteratively applies $\Phi_f$ of Definition~\ref{def:single-edge-map} for $f \in \vec{f}$.
(Note that if some two of the edges $f_i \ne f_j$ happened to belong to the same signed gradient cluster, then after applying the first operation $\Phi_{f_i}$, the second operation $\Phi_{f_j}$ would behave trivially as the identity.)
\end{defn}

We also construct a witness for $\Phi_{\vec{f}}$ analogously to how we defined a witness for $\Phi_f$ in Definition~\ref{def:witness} by letting $\vec{\mathcal W}$ be given by $S \subset [\ell]$, together with $(\mathcal W_{f_i})_{i\in S}$, where $S$ is the subset of $[\ell]$ such that when $\Phi_{f_i}$ is applied, it still has a gradient at $f_i$ (i.e., the gradient at $f_j$ wasn't part of the signed gradient cluster of $f_i$ for $i<j$), and for such $i\in S$, the witness $\mathcal W_{f_i}$ is constructed as per Definition~\ref{def:witness}. 

The set $\mathcal E_{\vec{f},\vec{\Delta}}$ decomposes naturally into a union over constituent events indexed by their corresponding witness: 
\begin{align}\label{eq:EfDeltaW-event}
    \mathcal E_{\vec{f},\vec{\Delta}} = \bigcup_{\vec{\mathcal W}} \mathcal E_{\vec{f},\vec{\Delta},\vec{\mathcal W}}
\end{align}
where the union is over witnesses that can correspond to $\phi \in \mathcal E_{\vec{f},\vec{\Delta}}$. 
We utilize this decomposition to define the events that will be reflected from the prism $\mathcal P$ across to the entire domain $\mathbb T_L^{d}$ using reflection positivity. The next lemma then follows directly from repeated application of Lemma~\ref{lem:Peierls-criteria-within-prism}.

\begin{lem}\label{lem:many-edges-and-gradients-bound}
    Fix any $\vec{f},\vec{\Delta}$ and compatible $\vec{\mathcal W}$. Let $\Phi$ be the map on configurations on $\mathcal E_{\vec{f},\vec{\Delta},\vec{\mathcal W}}$ as per Definition~\ref{def:Phi-multi-edges}. For any $\phi \in \mathcal E_{\vec{f},\vec{\Delta},\vec{\mathcal W}}$
    \begin{align*}
        H_{\mathcal P}(\phi) - H_{\mathcal P}(\Phi(\phi)) \ge \sum_{f\in E_1({\vec{\mathcal W}})} V_1(W_f) \,.
    \end{align*}
    Further, the number of witnesses $\vec{\mathcal W}$ having $\sum_{f} V_1(W_f) \le V_1(1)  k$ is at most $C^k$ for a constant $C(d)>0$. 
\end{lem}

\subsection{Reflection positivity preliminaries}\label{subsec:reflection-positivity-primer}

In this section, we collect the necessary definitions and  facts about reflection positivity, and its use with chessboard estimates to establish properties of  low-temperature phases. We follow the presentation of~\cite[Section 3]{hadas2022columnar}  and refer to~\cite{BiskupReflectionPositivity-Notes} and ~\cite[Chapter 10]{friedli_velenik_2017} for more details on reflection positivity more generally.

For a rectangle $R = R_{K\times M,(x_0,y_0)}$ (i.e., $\text{Width}(R)=K$, $\text{Height}(R)=M$, and its bottom-left vertex is $(x_0,y_0)$) consider the set of coordinate-direction lines  through vertices of $(x_0  + K\mathbb Z,y_0+M\mathbb Z)$ and let $T^R$ be the group of reflections of $\mathbb R^2$ generated by reflections across these lines. 
For $v\in R\cap \Z^2$, let $\tau_{R,v}\in T^R$ be the unique isometry which sends $R$ to $R - (x_0,y_0)+v$. 
Finally, for a function $f$ of configurations on $R \cap \Z^2$ and a map $\tau \in T^R$, let $\tau f(\sigma) = f(\sigma\circ \tau)$ (i.e., the function evaluated on the configuration at $\tau$ applied to $R \cap \Z^2$). 

For a rectangle $\Lambda \subset \Z^2$, consider the periodic boundary conditions and a distribution $\nu_{\Lambda}^\text{per}$ over  configurations $\sigma \in \Omega_{\Lambda}^\text{per}$ in some discrete product space. For us, this $\Lambda$ will be of width and height $L$, so that it defines the torus $\mathbb T_L^2$ (with $L^2$ vertices) by identification of its left and right, and its top and bottom boundary vertices.

\begin{defn}[Reflection positivity]\label{def:reflection-positivity}
    We say a measure is vertex reflection positive if the following holds. For every $\tau$ that is a reflection across a coordinate line of the torus, and every function $f$ measurable with respect to the configuration on one side of the line,
    \begin{align*}
        \nu_{\Lambda}^{\text{per}}(f \cdot \tau f) \ge 0\,.
    \end{align*}
\end{defn}

An essential implication of reflection positivity is the so-called chessboard estimate of~\cite{FrohlichLieb}. 

\begin{defn}
    A rectangle $R$ is called \emph{a block} of a rectangle $\Lambda$ if twice the width of $R$ divides the width of $\Lambda$ and twice the height of $R$ divides the height of $\Lambda$. 
\end{defn}

For a block $R$, let $T_{\Lambda}^R = T^R/ (\text{Width}(\Lambda) \mathbb Z \times \text{Height}(\Lambda)\mathbb{Z})$ be the subgroup of $T^R$ that correspond to distinct reflections in the torus corresponding to $\Lambda$. Its size is given by $|T_{\Lambda}^R|= |\Lambda|/|R|$. 

\begin{defn}[Chessboard seminorm]\label{defn:chessboard-seminorm}
    For an $R$-local function $f$, define the chessboard seminorm 
    \begin{align*}
        \|f\|_{R \vert \Lambda} := \nu_{\Lambda}^{\text{per}}\Big(\prod_{\tau \in T_\Lambda^R} \tau f\Big)^{1/|T_{\Lambda}^R|}\,.
    \end{align*}
    Similarly, for an event $\mathcal E$, we use $\|\mathcal E\|_{R\vert \Lambda}$ to mean $\|\mathbf 1_{\mathcal E} \|_{R\vert \Lambda}$. 
\end{defn}

\begin{lem}[Chessboard estimate]\label{lem:chessboard-estimate}
Let $R$ be a block of $\Lambda$, let $A \subset T_{\Lambda}^R$ and let $(f_{\tau})_{\tau \in A}$ be  
    $R$-local functions. Then  
    \begin{align*}
        \nu_{\Lambda}^{\text{per}}\Big(\prod_{\tau \in A} \tau f_\tau\Big) \le \prod_{\tau \in A} \|f_{\tau}\|_{R\vert \Lambda}\,.
    \end{align*}
\end{lem}

\subsection{Reflection positivity for the sloped-periodic gradient measure}

The reflection positivity tool is better suited to Gibbs measures over height functions than their induced gradient measures. Furthermore the reflection positivity is restricted to reflections across $\{\mathfrak{e}_1,\mathfrak{e}_2\}$-oriented lines through vertices of $\mathbb Z^2$. To deal with these, in what follows, we naturally identify a distribution on $\varphi:\mathbb T^d_L \to \mathbb Z$ with a distribution over $\tilde \Omega: =\{\tilde\varphi: \mathbb T^2_L \to \mathbb Z^{\mathbb T^{d-2}_L}\}$. In this way, our height function measures are spin systems only on a periodic torus (whereas in the sloped directions the periodicity is broken by the slope-inducing interaction), allowing application of the tools of reflection positivity. 

We will show reflection positivity for the gradient measure by showing it for approximating reflection positive height function measures, and taking a limit of their induced gradient measures to get to the right distribution. Reflection positivity for the gradient measures has been shown before, e.g.,~\cite[Proposition 3.5]{milos2015delocalization}; we include a proof in our setting of restricted reflection positivity for completeness. 
For this purpose, define the measure $\tilde \mu_{\eps}$ over $\tilde \Omega$ as 
\begin{align*}
    \tilde \mu^{(\eps)}_{L,\mathfrak{s}} (\tilde \varphi) \propto \mu^\nabla_{L,\mathfrak{s}}(\varphi^\nabla) e^{ - \eps \sum_v \varphi_v^2}
\end{align*}
where $\varphi$ is the configuration in $\Omega$ corresponding to $\tilde\varphi$. Since $\eps>0$, this is a valid probability measure on $\tilde \Omega$. At the same time, for each fixed $L$, as $\eps \downarrow 0$, the gradient distribution it induces converges to $\mu_{L,\mathfrak{s}}^\nabla$, so we will be able to pull back its reflection positivity.

\begin{lem}\label{lem:even-convex-implies-reflection-positive}
    Suppose $\mathbf{V}$ satisfies~\eqref{eq:potential assumption} and $V_1,V_2$ are even. For every $\eps>0$, the distribution $\tilde \mu_{L,\mathfrak{s}}^{(\eps)}$ is vertex-reflection positive. 
\end{lem}
\begin{proof}
    
    This will be an immediate consequence of e.g., Corollary 5.4 of~\cite{BiskupReflectionPositivity-Notes}, by taking the spin space for each vertex to be $\Z^{\mathbb T_L^{d-2}}$, and the underlying graph to be $\mathbb T_L^2$. To see the assumption that the Hamiltonian satisfies the requisite equation (5.10) there with only the $A$ and $\vartheta A$ terms, for a configuration $\varphi$ on the vertices on $(\mathbb T_L^2)^+$ (the half of the torus on one side of the reflection plane $P$ that goes through vertices) is seen by defining 
    \begin{align*}
        - A(\tilde \phi) & : = \frac{1}{2} \sum_{u\sim v: u,v\in P\times \mathbb T_L^{d-2}} V_{uv}(\phi_v - \phi_u) + \sum_{\substack{u\sim v: \\ u\notin P \times \mathbb T_L^{d-2} \text{ or }v\notin P \times \mathbb T_L^{d-2}, \\ u,v\in (\mathbb T_L^2)^+ \times \mathbb T_L^{d-2}}}V_{uv}(\phi_u - \phi_v) \\ 
        & \qquad + \frac{\eps}{2} \sum_{u\in P\times \mathbb T_L^{d-2}}  \phi_u^2 + \frac{\eps}{2} \sum_{u\notin P\times \mathbb T_L^{d-2}\,,u\in (\mathbb T_L^2 )^+\times \mathbb T_L^{d-2}}  \phi_u^2\,.   
    \end{align*}
    where $\phi$ is the configuration in $\{\varphi: \mathbb T_L^d \mapsto \mathbb Z\}$ that $\tilde \varphi$ corresponds to. Note that the fact we split the interaction along $P$ in half between $A$ and $\vartheta A$, was only possible using evenness of the potential $V$ for the $\{\mathfrak{e}_1,\mathfrak{e}_2\}$-oriented edges as the reflection plane is $\{\mathfrak{e}_1,\mathfrak{e}_2\}$-oriented.
\end{proof}

As a corollary, we get that the chessboard estimate is applicable to the gradient measure $\mu_{L,\mathfrak{s}}^\nabla$.  
\begin{cor}\label{cor:gradient-chessboard-estimate}
    Let $R$ be a block of $\Lambda$, and for each $R$-local gradient function $f$ (i.e., function only of the gradients of $\phi$ on edges of $R\times \mathbb T_L^{d-2}$), define its chessboard seminorm as 
    \begin{align*}
        \|f\|_{R\vert \Lambda} = \mu_{L,\mathfrak{s}}^\nabla \Big( \prod_{\tau\in T_{\Lambda}^R} \tau f\Big)^{1/|T_\Lambda^R|}
    \end{align*}
    Then for $A \subset T_{\Lambda}^R$ and bounded $(f_\tau)_{\tau\in A}$ that are $R$ local, 
    \begin{align*}
                \mu_{L,\mathfrak{s}}^\nabla\Big(\prod_{\tau \in A} \tau f_\tau\Big) \le \prod_{\tau \in A} \|f_{\tau}\|_{R\vert \Lambda}\,.
    \end{align*}
\end{cor}
\begin{proof}
    Let $\mu^{\nabla,(\eps)}_{L,\mathfrak{s}}$ be the gradient measure induced by $\tilde \mu_{L,\mathfrak{s}}^{(\eps)}$. By Lemma~\ref{lem:even-convex-implies-reflection-positive} and Lemma~\ref{lem:chessboard-estimate}, for $R$-local gradient functions $f_\tau$ for all $\eps>0$ fixed, we have 
    \begin{align*}
        \mu_{L,\mathfrak{s}}^{\nabla,(\eps)}\Big(\prod_{\tau \in A} \tau f_\tau\Big) \le \mu_{L,\mathfrak{s}}^{\nabla,(\eps)} \Big( \prod_{\tau\in T_{\Lambda}^R} \tau f_\tau\Big)^{1/|T_\Lambda^R|}\,.
    \end{align*}
    By continuity of probability measures, for bounded $f_\tau$, sending $\eps\downarrow 0$ on both sides concludes. 
\end{proof}

\subsection{Concluding the proof}

For our purposes, the two-by-two block $\{0,1\}^2 = R_{1\times 1,(0,0)}$ or its shifts across $\T_L^2$ are taken as $R$. In order for this to be a block in $\T_L^2$ we therefore assume $L$ is even. 
For a fixed set of edges $\vec{f}$ in the prism $\mathcal P$, a fixed set of gradients along them $\vec{\Delta}$ and a fixed set of witnesses $\vec{\mathcal W}$, the indicator functions for events of the form $\mathcal E_{\vec{f},\vec{\Delta},\vec{\mathcal W}}$ from~\eqref{eq:EfDeltaW-event} are then evidently bounded $R$-local gradient functions.

We will therefore focus on bounding the chessboard seminorm of events $\mathcal E_{\vec{f},\vec{\Delta},\vec{\mathcal W}}$ by disseminating the event across $\T^{d}$ via the reflections $\tau \in T_{\Lambda}^R$ across the lines of $\T_L^2$.

\begin{lem}\label{lem:chessboard-seminorm-bound}
    Suppose $L$ is even. Fix any $\vec{f} = (f_1,...,f_k) \in E_1(\mathcal P)$, any $\vec{\Delta}$, and any $\vec{\mathcal W}$ compatible with those. If $\mathbf{V}$ satisfies~\eqref{eq:potential assumption} and $V_1(1)\ge \beta_0$, then 
    \begin{align*}
        \|\mathcal E_{\vec{f},\vec{\Delta},\vec{\mathcal W}}\|_{R \vert \Lambda_L} \le \exp\Big( - \frac{1}{4} V_1(1) \sum_{i=1}^k \sum_{f\in \mathcal W_i} |W_f| \Big).
    \end{align*}
\end{lem}

\begin{proof}
We consider the event $\bar{\mathcal E} := \bigcap_{\tau \in T_\Lambda^R} \tau \mathcal E_{\vec{f}, \vec{\Delta}, \vec{\mathcal W}}$, which is the prism event reflected across $\mathbb T_L^2$. We will apply a Peierls argument (which is simply the reflection of the basic Peierls operation done on the prism). Namely, for a configuration in  $\bar {\mathcal E}$, let $\bar\Phi$ be the operation that on each prism $\tau \mathcal P$ for $\tau \in T_{\Lambda}^R$, performs the reflection of the basic operation $\Phi$; more formally, for a configuration $\bar \varphi \in \bar{\mathcal E}$ for each reflection $\tau \in T_\Lambda^R$, the local input on the prism $\tau P$ is $\varphi^\tau=\bar\varphi\circ\tau$. Then we can define for every $v\in \tau P$, the new configuration  
\begin{align}\label{eq:global-map-construction}
\bar\Phi(\bar\varphi)(v) - \bar\varphi(v)
=\Phi(\varphi^\tau(\tau^{-1}(v)))- \varphi^\tau(\tau^{-1}(v)),
\qquad v\in\tau P\,,
\end{align}
where $\Phi$ is as defined in Definition~\ref{def:Phi-multi-edges}. There are multiple $\tau$ such that $v\in \tau\{0,1\}^2$, but it is straightforward to verify that the above operation does not depend on the choice of $\tau$, as the pre-image vertex and gradients in the base prism will not depend on the choice of $\tau$.

This operation thus gives a configuration in $\{\tilde \phi: \mathbb T_L^2\to \mathbb Z^{\mathbb T_L^{d-2}}\}$ which is naturally associated to a configuration on $\{ \phi:\mathbb T_L^d \to \mathbb Z\}$. 
On the one hand, by~\eqref{eq:global-map-construction}, for each edge $e$ internal to prism $\tau \mathcal P$, the energy change induced along $e$ from applying the map $\bar \Phi$ is the same as that of the map $\Phi$ to $\tau^{-1}e$, which induces an energy change, per Lemma~\ref{lem:Peierls-criteria-within-prism}, of 
\begin{align*}
    V_1(1) \sum_{i} \sum_{f\in \mathcal W_i} |W_f|\,\,
\end{align*}
per prism. Summing the above display over all edges in $\mathbb T_L^d$, all $\{\mathfrak{e}_1,\mathfrak{e}_2\}$-oriented edges are counted exactly twice, and all  vertical edges are counted exactly four times. Thus, in total, 
\begin{align*}
     H(\bar \varphi) - H(\bar \Phi (\bar \varphi)) \ge \frac{|T_\Lambda^R|}{4}  V_1(1)\sum_i \sum_{f\in \mathcal W_i} |W_f|.
\end{align*}

Next, we observe that by the definition of a witness, for fixed $\vec{f},\vec{\Delta},\vec{\mathcal W}$, the map $\Phi$ is an injection when applied to $\mathcal E_{\vec{f},\vec{\Delta},\vec{\mathcal W}}$. Since the reflections generating the different prisms from the base one can be inverted, this means the map $\bar \Phi$ is an injection when applied on configurations in $\bar{\mathcal E}$. Therefore, we have the bound 
\begin{align*}
    \|\bar{\mathcal E}\|_{R\vert \Lambda}^{|T_{\Lambda}^R|} =  \sum_{\varphi \in \bar{\mathcal E}} \mu_{L,\mathfrak{s}}^\nabla(\bar\varphi) = \sum _{ \bar\varphi \in \bar{\mathcal E}} \frac{\mu_{L,\mathfrak{s}}^\nabla(\bar\varphi)}{\mu_{L,\mathfrak{s}}^\nabla(\bar \Phi(\bar \varphi))}\mu_{L,\mathfrak{s}}^\nabla (\bar \Phi(\bar\varphi))\,.
\end{align*}
Using injectivity of the map $\bar \Phi$ on the set $\bigcap_{\tau \in T_{\Lambda}^R} \tau \mathcal E_{\vec{f},\vec{\Delta},\vec{\mathcal W}}$ and using the energy change from the map from the first part of Lemma~\ref{lem:many-edges-and-gradients-bound},  we get 
\begin{align*}
    \|\mathcal E_{\vec{f},\vec{\Delta},\vec{\mathcal W}}\|_{R \vert \Lambda} \le \exp \Big( - \frac{1}{4} V_1(1) \sum_{i}\sum_{f\in \mathcal W_i} |W_f|\Big)
\end{align*}
where the power of $1/|T_{\Lambda}^R|$ on the probability cancels out with the fact that each of the $|T_\Lambda^R|$ many prisms contribute  an energy decrease of at least $\frac{1}{4} V_1(1) \sum_i \sum_{f\in \mathcal W_i} |W_f|$. 
\end{proof}

\medskip
\begin{proof}[\textbf{\emph{Proof of Theorem~\ref{thm:1-2-direction-rigidity}}}]
     Consider the probability that a fixed set of $m=k+\ell$ horizontal  edges $g_1,...,g_m$  in $\mathbb T_L^d$, with $g_1,...,g_{k}$ being the $\mathfrak{e}_1$-oriented ones, and $g'_1,...,g'_{\ell}$ being the $\mathfrak{e}_2$-oriented ones,  have gradients $|2\Delta_1| \ge r_1,...,|2\Delta_{k}|\ge r_k$ and $|2\Delta_1'|\ge r_{k+1},...,|2\Delta_{\ell}'|\ge r_{m}$ respectively. Suppose without loss of generality that $\sum_{i=1}^k |\Delta_i| \ge \sum_{i=1}^{\ell}|\Delta_i'|$ so we consider only $\mathfrak{e}_1$-oriented edges, incurring a factor of $2$ for a union bound in the final result. 
     
     Denote by $\mathcal A_{\vec{g},\vec{\Delta}}$ the event that $g_1,...,g_k$ have gradients $|2\Delta_1|,...,|2\Delta_k|$. For each such edge, let $\tau P$ be a prism to which it belongs (making consistent choices of prisms so that no pair of edges which share a prism are given different prisms); this naturally splits that set of horizontal edges into subsets belonging to disjoint $(\tau P)_{\tau \in S}$ for $S \subset T_{\Lambda}^R$. For each such $\tau \in S$, let $g_{i_1}^{(\tau)},...,g_{i_{k_\tau}}^{(\tau)}$ be the edges in $\tau P$ amongst $g_1,...,g_k$, and let $\tau \mathcal W^{(\tau)}$ be a potential witness (in the prism $\tau P$) for the collection $g_{i_1}^{(\tau)},...,g_{i_{k_\tau}}^{(\tau)}$. This generates a family of events $(\tau \mathcal E_{\vec{g}^{(\tau)},\vec{\Delta}^{(\tau)},\vec{\mathcal W}^{(\tau)}})_{\tau \in S}$ such that 
     \begin{align}\label{eq:A-in-tauE-inclusion}
         \mathcal A_{\vec{g},\vec{\Delta}} \subset \bigcup_{(\vec{\mathcal W}^{(\tau)})_{\tau \in S}} \bigcap_{\tau \in S} \tau \mathcal E_{\vec{g}^{(\tau)},\vec{\Delta}^{(\tau)},\vec{\mathcal W}^{(\tau)}}
     \end{align}
     where the union runs over collections of witnesses (shifted back to the origin prism) which are compatible with the edge set and gradient set of $\tau P$. 

     Therefore, we have by several union bounds, 
     \begin{align*}
         \mu_{L,\mathfrak{s}}^\nabla\Big(\bigcap_{i=1}^\ell \{|\nabla_{g_i} \phi| \ge r_i\}\Big) \le \sum_{\vec{\Delta}: \Delta_i \ge r_i} \mu_{L,\mathfrak{s}}^\nabla  (\mathcal A_{\vec{g},\vec{\Delta}}) & \le \sum_{\vec{\Delta}: \Delta_i \ge r_i} \sum_{(\vec{\mathcal W}^{(\tau)})_{\tau \in S}} \mu_{L,\mathfrak{s}}^\nabla  \Big(\bigcap_{\tau \in S} \tau \mathcal E_{\vec{g}^{(\tau)},\vec{\Delta}^{(\tau)},\vec{\mathcal W}^{(\tau)}}\Big)\,.
     \end{align*}
    By Corollary~\ref{cor:gradient-chessboard-estimate}, this implies 
    \begin{align*}
        \mu_{L,\mathfrak{s}}^\nabla\Big(\bigcap_{i=1}^\ell \{|\nabla_{g_i} \phi| \ge r_i\}\Big)\le \sum_{\vec{\Delta}: \Delta_i \ge r_i} \sum_{(\vec{\mathcal W}^{(\tau)})_{\tau \in S}} \prod_{\tau \in S} \|\mathcal E_{\vec{g}^{(\tau)},\vec{\Delta}^{(\tau)},\vec{\mathcal W}^{(\tau)}}\|_{R \vert \Lambda}\,.
    \end{align*}
      Now the chessboard seminorm on the right-hand side is bounded by Lemma~\ref{lem:chessboard-seminorm-bound} which gives  
    \begin{align*}
        \mu_{L,\mathfrak{s}}^\nabla(\mathcal A_{\vec{g},\vec{\Delta}}) \le \sum_{\vec{\Delta}:\Delta_i \ge r_i} \sum_{(\vec{\mathcal W}^{(\tau)})_{\tau\in S}}  \exp \Big( - \frac{1}{4} V_1(1) \sum_{\tau \in S} \sum_{i}^{k_\tau} \sum_{f\in \mathcal W^{\tau}_i} |W_f|\Big)\,.
    \end{align*}
    Since the weight of each of the witnesses $\mathcal W^{\tau}$ exceeds $\sum_{i=1}^{k_\tau} \Delta_{i_i}^{(\tau)}$, the sum of the $W_f$'s exceeds $\sum_{i=1}^k r_i$, and we can split off a factor of $e^{ - \frac{1}{8}V_1(1) \sum_{i=1}^k r_i}$ from the above. By Lemma~\ref{lem:many-edges-and-gradients-bound}, we have that if $|\vec{\Delta}^{(\tau)}| = m^{\tau}$, then the number of summands from the second sum with $\sum_{\tau \in S} \sum_{i}^{k_\tau} \sum_{f\in \mathcal W_i^{\tau}} |W_f| = M$ is at most $C^M$ for a universal constant $C = C(d)$.  For a fixed $M$, the number of $\vec{\Delta}$ such that $\sum \Delta_i \le M$ is bounded by $2^{2M}$. Therefore, the above is bounded as 
        \begin{align*}
        \mu_{L,\mathfrak{s}}^\nabla(\mathcal A_{\vec{g},\vec{\Delta}}) \le e^{ - \frac{1}{8}V_1(1) \sum_{i=1}^k r_i} \sum_{M \ge \sum r_i} 2^{2M} C^M e^{ - \frac{1}{8} V_1(1) M}\,.
    \end{align*}
    The claimed bound follows using that $V_1(1) \ge \beta_0$ for a sufficiently large constant $\beta_0$.   
    \end{proof}

\section{Constructing limiting height function Gibbs measures}\label{sec:limits}

In this section, we utilize the properties of the finite-volume gradient measure with sloped-periodic boundary, and in particular the fact that non-zero gradients in the $\{\mathfrak{e}_1,\mathfrak{e}_2\}$-directions are exponentially suppressed, to deduce properties of sloped limiting infinite-volume Gibbs measures. Namely, we prove existence of localized height Gibbs measures, and establish that they have dominant heights in each layer. In particular, we will prove the existence half of  Theorem~\ref{thm:main-localization}.

 \subsection{Passing to infinite-volume for the gradient Gibbs measure}
We begin by passing the exponential bounds of Theorem~\ref{thm:1-2-direction-rigidity} to subsequential limits of the gradient measure as $L\to\infty$. Since the events in Theorem~\ref{thm:1-2-direction-rigidity} are cylinder events, we immediately get the following. 

\begin{cor}\label{cor:subsequential-gradient-limit-has-suppressed-gradients}
    Let $\mu^\nabla_{\infty,\mathfrak{s}}$ be a subsequential limit of $\mu^\nabla_{L,\mathfrak{s}}$ as $L\to\infty$. Then 
    \begin{align}\label{eq:exp-suppressed-gradients-condition}
    & \mbox{for any finite set of $\{\mathfrak{e}_1,\mathfrak{e}_2\}$-oriented edges $e_1,...,e_m$ and integers $r_1,...,r_m\ge 1$,} \nonumber  \\ 
        & \qquad\qquad\mu^\nabla_{\infty,\mathfrak{s}} \Big( \bigcap_{i=1}^m \{|\nabla_{e_i} \varphi| \ge  r_i\} \Big)\le e^{ - c \min_{j=1,2}V_j(1) \sum_{i=1}^m r_i}\,.
    \end{align}
\end{cor}

In what follows, we assume $\mu_\infty^\nabla$ is any infinite-volume gradient Gibbs measure satisfying~\eqref{eq:exp-suppressed-gradients-condition} and $\min_{j=1,2}V_j(1)=\beta \ge \beta_0$ is sufficiently large, and deduce that any such measure has dominant heights, and in each layer, the deviations from the dominant height have the desired properties of a very subcritical percolation process. 

\begin{defn}\label{def:density-dominant-heights}
    We say an infinite-volume gradient Gibbs measure $\mu^\nabla_\infty$ has \emph{densely dominant heights} if the following all hold:    
    \begin{enumerate}
        \item $\mu_\infty^\nabla$ has dominant heights in all layers (recalling Definition~\ref{defn:dominant-height}). 
        \item Almost surely, in every layer $\mathcal L$, 
        \begin{align*}
            \limsup_{L\to\infty} \frac{1}{4L^2} \{x\in [-L,L]^d \cap \mathcal L: \varphi_x \ne h(\mathcal L)\} \le 1/20\,.
        \end{align*}
    \end{enumerate}
    Note that the above properties are all measurable with respect to the gradients, and therefore the above makes sense. 
\end{defn}

Our first aim in this section is to use Peierls arguments to prove the following. 

\begin{prop}\label{prop:exp-suppressed-implies-dominant-heights}
    There exists $\beta_0(d)$ such that if $\min_{j=1,2} V_j(1)= \beta \ge \beta_0$, the following holds. 
    If $\mu^\nabla_\infty$ is any translation-invariant infinite-volume gradient Gibbs measure satisfying~\eqref{eq:exp-suppressed-gradients-condition}, then it has densely dominant heights. 
\end{prop}

\subsection{Sub-criticality of excitations away from the dominant height}
In this section, we use the output of the previous section together with standard Peierls arguments to deduce the fact that each layer has a dominant height, and excitations away from that dominant height are small.  
First, we establish that each layer indeed has one height that can be called its \emph{dominant height}. For that purpose, it will help to introduce some notation. 

Recall the definition of the layers of $\mathbb Z^d$, and identify them with vectors in $\mathbb Z^{d-2}$ naturally by saying the layer $\mathcal L_{\vec{k}}$ indexed by $\vec{k}\in \mathbb Z^{d-2}$ is the plane $\mathbb Z^2 \times \vec{k}$. 
Let  $E(\mathcal L_{\vec{k}})$ be the set of horizontal edges on that layer. 

\begin{defn}\label{def:horizontal-vertical-disagreement}
    An \emph{horizontal disagreement} in a layer $\mathcal L$ is a $(d-1)$-dimensional plaquette $e^*$ dual to an edge $e= vw$ for $vw\in E(\mathcal L)$ for which $\phi_v\ne \phi_w$. 
\end{defn}

The set of horizontal disagreements at layer $\mathcal L_{\vec{k}}$ for $\vec{k} \in \mathbb Z^{d-2}$ can be written as 
        \begin{align*}
    \Gamma_{\vec{k}} = \{ e^*: e\in E(\mathcal L_{\vec{k}}),  \nabla_e\phi \ne 0\}\,.
\end{align*}
Decorate the set $\Gamma_{\vec{k}}$ with integers associated to each of its plaquettes to arrive at $\mathbf{\Gamma}_{\vec{k}}:= (\Gamma_{\vec{k}}, (\nabla_e \phi)_{e^*\in \Gamma_{\vec{k}}})$. Then given $\mathbf{\Gamma}_{\vec{k}}$ and the value of $\phi$ at a single vertex in $\mathcal L_{\vec{k}}$, one can uniquely reconstruct the values of $\phi$ on all of $\mathcal L_{\vec{k}}$. Note that the horizontal disagreement plaquettes can equivalently be seen as dual edges in the layer by only taking their intersection with $\mathbb R^2 \times \{\vec{k}\}$, so we will often refer to its constituent elements as (dual) edges. 

Define a \emph{horizontal contour} at layer $\vec{k}$ to be a maximal connected component of $\Gamma_{\vec{k}}$. The set $\Gamma_{\vec{k}}$ then naturally decomposes into its constituent contours $\Gamma_{\vec{k}}= \bigcup_i \gamma_{\vec{k},i}$, and each can be decorated by the values of $\mathbf{\Gamma}_{\vec{k}}$ to get decorated horizontal contours $\boldsymbol{\gamma}_{\vec{k},i}$.

 We can assign weights to horizontal contours as follows: 
\begin{align*}
    W(\boldsymbol{\gamma}_{\vec{k}}) = \sum_{e^* \in \gamma_{\vec{k}}} |\nabla_{e}\phi | \,,
\end{align*}
From Theorem~\ref{thm:1-2-direction-rigidity} and a Peierls argument, we get the following. 

\begin{lem}\label{lem:largest-contour-in-layer-bound}
    There exists $c>0$ such that if $\mu_\infty^\nabla$ satisfies~\eqref{eq:exp-suppressed-gradients-condition} and $\beta_0$ is sufficiently large, for every layer $\vec{k}$, 
    $$\mu_{\infty}^\nabla( \exists \gamma\in \Gamma_{\vec{k}}:  \gamma \cap [-L,L]^2\times \{\vec{k}\} \ne \emptyset, W(\boldsymbol{\gamma})\ge r ) \le 4L^{2} e^{ - c\beta  r}\,.$$  
\end{lem}

\begin{proof}
    Consider the probability that there exists such a $\gamma \in \Gamma_{\vec{k}}$ containing a vertex $v\in [-L,L]^2 \times \vec{k}$. By a union bound, it suffices to show that the probability of that is at most $e^{ - c\beta   r}$. Fix any decorated contour $\bar{\boldsymbol{\gamma}}$ going through a fixed vertex $v$, and consider 
    \begin{align*}
        \mu_{\infty}^\nabla ( \bar{\boldsymbol{\gamma}}\in \mathbf{\Gamma}_{\vec{k}})\,.
    \end{align*}
    The event $\bar{\boldsymbol{\gamma}} \in \mathbf{\Gamma}_{\vec{k}}$ is exactly of the form: fix a dual edge-set $\Gamma$ in $E(\mathcal L_{\vec{k}})$ and prescribe non-zero gradients along those edges. The assumption~\eqref{eq:exp-suppressed-gradients-condition} gives that this probability is at most $e^{ - c \beta W(\bar{\boldsymbol{\gamma}})}$.  

    We can then union bound 
    \begin{align}\label{eq:horizontal-contours-suppressed-bound}
        \mu_{\infty}^\nabla(\exists \gamma \ni v: \gamma \in \Gamma_{\vec{k}}, W(\boldsymbol{\gamma}) \ge r ) \le \sum_{\ell \ge r} \sum_{\boldsymbol{\gamma}: \gamma \ni v, W(\boldsymbol{\gamma}) = \ell} e^{ - c\beta \ell}\,.
    \end{align}
    The following lemma bounds the number of summands in the second sum above. 
    
       \begin{lem}\label{lem:counting-horizontal-contours}
        The number of decorated horizontal contours $\boldsymbol{\gamma} \subset E(\mathcal L_{\vec{k}})$ containing a fixed vertex $v\in \mathcal L_{\vec{k}}$ in its $\mathcal L_{\vec{k}}$-interior, and of total weight $\ell$ is at most $C^\ell$ for a universal $C$. 
    \end{lem}
     \begin{proof}
        We first enumerate over the dual-edges in $E(\mathcal L_{\vec{k}})$ that form $\gamma$. These are connected sets of at most $\ell$ edges in a copy of $\mathbb Z^2$, of which there are at most $C^\ell$ for a universal constant $C$. 
        In order to augment this into a decorated horizontal contour, we must partition the total weight $\ell$ amongst those $|\gamma| \le \ell$ many edges, with each one having weight at least $1$. There are at most $2^{2\ell}$ many choices for such assignments of absolute values of gradients to the edges in $\gamma$, and a further $2^\ell$ for the choice of signs on each of these gradients. 
    \end{proof}

    Combining Lemma~\ref{lem:counting-horizontal-contours} with~\eqref{eq:horizontal-contours-suppressed-bound} and summing concludes the proof up to a change in the choice of constant $c$, so long as $\beta\ge \beta_0$ is big enough.    
\end{proof}

We are already able to deduce the proof of item (1) of Definition~\ref{def:density-dominant-heights}. 

\begin{cor}\label{cor:has-dominant-heights}
    If $\mu_\infty^\nabla$ satisfies~\eqref{eq:exp-suppressed-gradients-condition} and $\beta_0$ is sufficiently large, then $\mu_\infty^\nabla$ almost surely has dominant heights in the sense of Definition~\ref{defn:dominant-height}. 
\end{cor}
\begin{proof}
    By an essentially identical argument to the proof of Lemma~\ref{lem:largest-contour-in-layer-bound}, the probability of a connected path of horizontal disagreement edges of length at least $r$, starting at the origin, is at most $e^{ - c\beta r}$ under $\mu^\nabla_{\infty}$. In particular, by continuity of probability, almost surely under $\mu^\nabla_{\infty}$, for every layer $\vec{k}$, all contours of $\Gamma_{\vec{k}}$ are finite. As such, 
\begin{itemize}
    \item Every $\boldsymbol{\gamma}$ (viewed as a subset of $\mathbb R^2$) has complement in $\mathbb R^2$ having exactly one infinite component, which we call the exterior $\text{Ext}(\boldsymbol{\gamma})$. The intersection of all the exteriors of all components in $\boldsymbol{\Gamma}_{\vec{k}}$ is a connected subset of $\mathbb R^2\times \{\vec{k}\}$ which we call the exterior $\text{Ext}(\mathbf{\Gamma}_{\vec{k}})$ in layer $\mathcal L_{\vec{k}}$. This also means that distinct contours in the same layer are such that either one is interior to the other, or they are mutually exterior. 
\end{itemize} 
We now claim that $\text{Ext}(\boldsymbol{\Gamma}_{\vec{k}})$ percolates almost surely, and because all vertices in it get the same height (they are connected by paths of gradient zero), their height is the dominant height of the layer. The almost sure percolation will follow from Lemma~\ref{lem:largest-contour-in-layer-bound} by continuity of probability. Indeed, consider the event that $[-M,M]^2\times \{\vec{k}\}$ is connected to infinity through $\text{Ext}(\boldsymbol{\Gamma}_{\vec{k}})$ in $\mathcal L_{\vec{k}}$. The complement of this requires some contour  in $\boldsymbol{\Gamma}_{\vec{k}}$ enclosing $[-M,M]^2\times \{\vec{k}\}$ which is bounded by Lemma~\ref{lem:largest-contour-in-layer-bound} applied with $r=L$, and a union bound over $L$, as at most $\sum_{L\ge M} e^{ - c\beta L/2}$ which goes to zero as $M\to\infty$. 
\end{proof}

We also immediately obtain the following one-point tightness result.

\begin{cor}\label{cor:one-point-exponential-tails}
    If $\mu_\infty^\nabla$ satisfies~\eqref{eq:exp-suppressed-gradients-condition} and $\beta_0$ is sufficiently large, then the height difference at a layer's origin $o = (0,0,\vec{k})$ to its layer's dominant height satisfies $\mu^\nabla_\infty(|\varphi_o - h_\varphi(\vec{k})|\ge r) \le e^{ - c \beta r}$.  
\end{cor}
\begin{proof}
    By Corollary~\ref{cor:has-dominant-heights} almost surely, every layer has a well-defined $\text{Ext}(\boldsymbol{\Gamma}_{\vec{k}})$, relative to which we can define height gradients. The event that $|\varphi_o - h_{\varphi}(\vec{k})|$ exceeds $r$ necessitates that there are a family of nested horizontal contours in $E(\mathcal L_{\vec{k}})$ all confining the origin in their interiors, such that their total gradient changes exceeds $r$.  This is bounded by a Peierls argument, similar to the proof of Lemma~\ref{lem:largest-contour-in-layer-bound}. Fix $M\ge r$ to be the total weight of all contours in layer $\vec{k}$  to which $o$ is internal, enumerate over $K\le M$ which counts the number of such horizontal contours, and $C^M$ for enumerating over the specific choices of those $K$ decorated contours by Lemma~\ref{lem:counting-horizontal-contours}. Then by applying~\eqref{eq:exp-suppressed-gradients-condition} and summing over the choices made in the enumeration, we get the desired. 
\end{proof}

\subsection{Empirical counts of excitations from the dominant height}
It remains to establish that item (2) of Definition~\ref{def:density-dominant-heights} holds for $\mu_\infty^\nabla$ satisfying~\eqref{eq:exp-suppressed-gradients-condition}. 

We give bounds on the statistics of numbers of vertices internal to contours within each layer at distance at most $L$ from the origin. For that purpose, we fix a layer index $\vec{k}$, drop it from the notation for now, and partition the set of contours in $\mathbf{\Gamma}_{\vec{k}}$ intersecting $[-L,L]^2 \times \{\vec{k}\}$ by a dyadic scaling as follows: 
\begin{align*}
    A_\ell = \{\boldsymbol{\zeta} \in \mathbf{\Gamma}_{\vec{k}}: \gamma \cap [-L,L]^2 \times \{\vec{k}\}\ne\emptyset\,,\, W(\boldsymbol{\zeta}) \in [2^\ell, 2^{\ell+1}]\} \qquad \ell \ge 0\,.
\end{align*}

\begin{lem}\label{lem:excitation-statistics-bound}
    If $\mu_\infty^\nabla$ satisfies~\eqref{eq:exp-suppressed-gradients-condition} and $\beta_0$ is sufficiently large, for a universal constant $c'>0$ and for every $0 \le \ell \le \log_2 (\frac{1}{100 c \beta_0} \log L)$ (where $c$ is the constant from~\eqref{eq:exp-suppressed-gradients-condition}), 
    \begin{align*}
        \mu_{\infty}^\nabla\Big(\sum_{\boldsymbol{\zeta} \in A_\ell} W(\boldsymbol{\zeta}) \ge e^{ - c' \beta_0 2^{\ell}} L^2 \Big) \le e^{ - L^{0.9}}\,.
    \end{align*}
\end{lem}

\begin{proof}
    In order to have $\sum_{\boldsymbol{\zeta} \in A_\ell} W(\boldsymbol{\zeta}) \ge \rho L^2$, 
     one must have that $$|\{\boldsymbol{\zeta}: \boldsymbol{\zeta} \in A_\ell\}| \ge 2^{- (\ell+1)} \rho L^2\,.$$
     We now consider any fixed collection $\bar A_\ell$ with total  weight $\sum_{\boldsymbol{\zeta}\in \bar A_\ell} W(\boldsymbol{\zeta})$. By~\eqref{eq:exp-suppressed-gradients-condition}, the probability of having that collection of disagreement edges with their decorations as their gradients, is at most 
    \begin{align*}
        \mu_{\infty}^\nabla(\bar A_\ell) \le \exp\Big( - c \beta\sum_{\boldsymbol{\zeta}\in \bar A_\ell} W(\boldsymbol{\zeta})\Big)\,.
    \end{align*}

    We now enumerate the choices for the collection of contours $\bar A_\ell$ with specified total weight $W(\bar A_\ell):=\sum_{\boldsymbol{\zeta}\in \bar A_\ell}W(\boldsymbol{\zeta}) = T$. Since each has weight at least $2^{\ell}$, the total number of distinct contours in $\bar A_\ell$, denoted $N$, is at most $T 2^{-\ell }$. We are going to describe how to enumerate over $\bar A_{\ell}$ for fixed $T\ge \rho L^2$ and $N\le 2^{-\ell} T$. 
    \begin{enumerate}
        \item First locate $N$ vertices $v_1,...,v_N$ in $\mathbb T_L^2$ at which to root the contours. The number of choices for these $N$ vertices is~$\binom{L^2}{N}$. 
        \item Partition the integer $T$ into $N$ parts, $T_1,...,T_N$ with $\sum_{i} T_i = T$. These will be the total weights of the $N$ contours. Since $N\le T$, there are at most $2^{2T}$ many such partitions.  
        \item Next for each such root vertex $v_i$ in $\mathbb T_L^2$, enumerate over a decorated contour of total weight $T_i$ going through the vertex $v_i$. By Lemma~\ref{lem:counting-horizontal-contours}, the number of decorated horizontal contours going through vertex $v_i$ with weight $T_i$ is at most $C^{T_i}$ for a universal $C$. 
    \end{enumerate}

    Considering all of the above, if we sum over $\bar A_\ell$ which are possible collections of mutually compatible decorated contours belonging to $A_\ell$, we get that for each $\ell$, 
    \begin{align*}
        \mu_{\infty}^\nabla\Big(\sum_{\boldsymbol{\zeta}\in A_\ell} W(\boldsymbol{\zeta}) \ge \rho L^2\Big) \le   \sum_{T \ge \rho L^2 } \sum_{\bar A_\ell: W(\bar A_\ell) = T} \exp\Big( - c\beta \sum_{\boldsymbol{\zeta}\in \bar A_\ell} W(\boldsymbol{\zeta})\Big)\,.
    \end{align*} 
    By the above enumeration over $ \bar A_{\ell}$, the right-hand side is given by 
    \begin{align*}
     \sum_{T \ge \rho L^2} e^{ - c \beta_0 T}   \sum _{N \in [2^{- \ell-1} \rho L^2, 2^{-\ell} T]} \binom{L^2}N 2^{2T} \prod_{i=1}^N C^{T_i}  
    \end{align*}
    As long as $\beta_0$ is large, for $c'' = c/2$ say, 
    \begin{align*}
        \sum_{T\ge \rho L^2} \sum_{N\le 2^{ - \ell} T} \binom{L^2}{N} 2^{2T} C^T e^{ - c \beta_0 T} \le  \sum_{T\ge \rho L^2} \exp\Big( T 2^{ - \ell} \log \frac{L^2}{T 2^{-\ell}} \Big) e^{ - c''\beta_0 T}\,.
    \end{align*}
    So long as $\rho  = \rho_\ell$ is such that $T \ge \rho L^2$ implies 
    \begin{align*}
        2^{ - \ell} \log \frac{L^2}{T 2^{-\ell}} \le \frac{1}{2} c'' \beta_0 \qquad \text{i.e.} \qquad \log \frac{L^2}{T}  + \ell \log 2 \le 2^{ \ell-1} c'' \beta_0\,,
    \end{align*}
    which is to say that as long as  $2^{- \ell}\rho \ge C e^{ - c''\beta_0 2^{\ell-1}}$ for $C$ large, or equivalently, $\rho \ge C 2^{ \ell} e^{ - c'' \beta_0 2^{\ell-1}}$, then the above inequality holds true and one has the bound  
    \begin{align*}
        \sum_{T \ge \rho L^2} e^{ - c'' \beta_0 T/2}  \le C e^{ - c'' \beta_0 \rho L^2 /2}\,.
    \end{align*}
    for a universal constant $C$.  
    That constraint on $\rho$ is satisfied for $e^{ - c'' \beta_0 2^\ell/8}$, say, and the upper bound on $\ell$ implies the probability is still at most $e^{ - L^{0.9}}$.  
    \end{proof}

\begin{cor}\label{cor:dominant-height-exists}
    For any $C$ large, if $\mu_\infty^\nabla$ satisfies~\eqref{eq:exp-suppressed-gradients-condition} and $\min_{j=1,2} V_j(1)= \beta \ge \beta_0$ is sufficiently large, we have that for all large $L$ and all $\vec{k}\in \mathbb Z^{d-2}$, that 
    \begin{align*}
        \mu_\infty^\nabla \Big(|\{x\in [-L,L]^2 \times \{\vec{k}\}: \varphi_x \ne h(\mathcal L_{\vec{k}}) \}| \ge \frac{1}{C} 4L^2\Big) \le L^{- C}\,.
    \end{align*}
\end{cor}

\begin{proof}
    To start, let $\beta_0'$ be the $\beta_0$ from Lemma~\ref{lem:excitation-statistics-bound}. 
    By Lemma~\ref{lem:largest-contour-in-layer-bound}, except with probability $L^{-\Omega(\beta/\beta_0')}$, there is no contour of weight at least $\frac{1}{100 c\beta_0'} \log L$ intersecting $[-L,L]^2 \times \{\vec{k}\}$. The total number of vertices in $[-L,L]^2 \times \{\vec{k}\}$ interior to contours of weight at most $\frac{1}{100 c\beta_0'} \log L$ can be bounded as follows: by Lemma~\ref{lem:excitation-statistics-bound}
    \begin{align*}
        \mu_{\infty}^\nabla\Big( \bigcup_{\ell \le \log_2(\frac{1}{100 c\beta_0} \log L)} \Big\{\sum_{\boldsymbol{\zeta}\in A_\ell}W(\boldsymbol{\zeta}) \ge e^{ - c' \beta_0 2^{\ell} }L^2\Big\}\Big) \le  (\log \log L)e^{- L^{0.9}}
    \end{align*}
    On the complement of the union event above, since the number of vertices interior to a $\boldsymbol{\gamma}\in A_\ell$ is at most $2^{\ell+1} W(\boldsymbol{\gamma})$, summing the thresholds above gives that the total number of vertices interior to such contours is bounded by 
    \begin{align*}
        \sum_{\ell \le \log_2(\frac{1}{100 c\beta_0'} \log L)} 2^{\ell + 1} e^{ - c' \beta_0' 2^\ell} (4L^2) \le 4L^2/C
    \end{align*}
    for any large $C$ as long as $\beta_0'$ is sufficiently large. Subsequently take $\beta_0$ sufficiently large so that $\beta\ge \beta_0$ implies the probability of having a large contour was $o(L^{-C})$ to conclude. 
\end{proof}

The above Corollary gives item (2) of Definition~\ref{def:density-dominant-heights}, concluding the proof of Proposition~\ref{prop:exp-suppressed-implies-dominant-heights}.

\subsection{Constructing the infinite-volume height measure} 

We now use the dominant heights (which are tail measurable since they percolate) to lift subsequential limiting gradient measures $\mu^\nabla_{\infty,\mathfrak{s}}$ to a limiting height function Gibbs measure.

\begin{lem}\label{lem:existence-of-height-measures}
    There exists an infinite-volume Gibbs measure on height functions $\mu_{\infty,\mathfrak{s}}$ satisfying the DLR conditions, whose gradient measure is ergodic, minimal, and of slope $\mathfrak{s}$.  
\end{lem}

\begin{proof}
    By Lemma 4.2.6 in~\cite{Sheffield-Random-Surfaces}, the gradient measure $\mu_{L,\mathfrak{s}}^\nabla$ admits translation-invariant subsequential limits. Call $\tilde \mu_{\infty,\mathfrak{s}}^\nabla$ one such limit; it has a decomposition as a mixture of ergodic components, almost every one of which is minimal and has slope $\mathfrak{s}$, and by Corollary~\ref{cor:subsequential-gradient-limit-has-suppressed-gradients} together with Proposition~\ref{prop:exp-suppressed-implies-dominant-heights}, specifically item (1) of Definition~\ref{def:density-dominant-heights}, has dominant heights at every layer. Let $\mu_{\infty,\mathfrak{s}}^\nabla$ be any such ergodic component.

    We demonstrate how to lift $\mu_{\infty,\mathfrak{s}}^\nabla$ to a limiting height Gibbs measure. 
    From a gradient configuration $\phi^\nabla \sim \mu_{\infty,\mathfrak{s}}^\nabla$, generate a height function $\phi$ by setting 
    \begin{align*}
        \phi_v = \phi^\nabla_v - \phi_{\text{Ext}(\Gamma_{\vec{0}})}^\nabla
    \end{align*}
    which is to say by setting the height of the dominant height at layer $\vec{0}$ to $0$. 
    Because the right-hand side makes sense everywhere, it will suffice to show the following two items: 
    \begin{enumerate}
        \item \emph{Tightness}: The random variable $\phi_0$ is finite almost surely.
        \item \emph{Tail-measurability}:  In such a construction, for any finite $\Lambda$ the values of $\phi$ assigned to $\Lambda^c$ are measurable with respect to $(\phi_v^\nabla)_{v\in \Lambda^c}$. 
    \end{enumerate}  
    Indeed, item (1) implies that we arrive from the above construction at a probability measure on height functions $\phi:\mathbb Z^d \to \mathbb Z$. Item (2) above implies that when resampling the configuration on a finite $\Lambda$ conditional on its exterior, this is equivalent to resampling the gradients in $\Lambda$ conditional on the gradients in $\Lambda^c$ (which has the right distribution on gradients by the fact that $\mu_{\infty,\mathfrak{s}}^\nabla$ is a gradient Gibbs measure) then filling in the heights inside $\Lambda$ using that sample, and that would match with the finite-volume sample with the height boundary conditions $\phi(\Lambda^c)$. 

    Item (1) above follows from Corollary~\ref{cor:subsequential-gradient-limit-has-suppressed-gradients} together with Corollary~\ref{cor:one-point-exponential-tails}. Item (2) follows from the fact that almost surely in a sample from $\mu_{\infty,\mathfrak{s}}^\nabla$, the only percolating component of level sets of $\phi^\nabla$ in layer $\vec{0}$ is $\text{Ext}(\Gamma_{\vec{0}})$, so differences of heights from $v\in\Lambda^c$ to $\text{Ext}(\Gamma_{\vec{0}})$ are always $\Lambda^c$-measurable. 
\end{proof}

\subsection{All minimal ergodic gradient Gibbs measures satisfy~\eqref{eq:exp-suppressed-gradients-condition}}
When we go to show uniqueness of the limiting measures via disagreement percolation, we will use the following lemma which will show that all ergodic gradient measures of slope $\mathfrak{s}$ have exponentially small probabilities of non-zero gradients, and therefore have dominant heights on every layer. The proof goes by using the ergodic theorem to make absence of such rigidity an exponentially unlikely event in $L^d$, whereas differences of boundary conditions (Dirichlet vs.\ sloped periodic) cause only exponential in $L^{d-1}$ Radon--Nikodym derivatives. 

\begin{lem}\label{lem:ergodic-minimal-has-dominant-height}
    Suppose $\mathsf{m}^\nabla$ is any minimal ergodic gradient Gibbs measure of slope $\mathfrak{s}$. Then $\mathsf{m}^\nabla$ satisfies~\eqref{eq:exp-suppressed-gradients-condition} (and as a result by Proposition~\ref{prop:exp-suppressed-implies-dominant-heights}, has densely dominant heights per Definition~\ref{def:density-dominant-heights}). 
\end{lem}

\begin{proof}
    Fix any finite set of horizontal edges $e_1,...,e_m$ in $\mathbb Z^d$ and gradients $r_1,...,r_m$; let $p$ be the bound $e^{ - c \min_{j=1,2} V_j(1) \sum_{i} r_i}$ for $c$ which is the constant of Theorem~\ref{thm:1-2-direction-rigidity}; our aim is to show that any minimal ergodic gradient measure of slope $\mathfrak{s}$ satisfies   
    \begin{align*}
        \mathsf{m}^\nabla\Big( \bigcap_{i=1}^m\{|\nabla_{e_i} \varphi| \ge r_i \}\Big) \le p\,.
    \end{align*}
    Let $A$ be the event in the above probability. The event $A$ only depends on the gradients of $\varphi$ in a finite box $\Lambda_{M}= [-M,M]^d$; if we denote by $\theta_{v}$ for $v \in \Lambda_N$ the shift operator by vector $v$, by translation invariance of $\mathsf{m}^\nabla$ and the ergodic theorem, 
    \begin{align*}
        \mathsf{m}^\nabla(A)  = \lim_{N\to\infty } \mathsf{m}^\nabla \Big(  \frac{1}{|\Lambda_N|} \sum_{t\in \Lambda_N} \mathbf 1_{\theta_{t} A}\Big)\,,
    \end{align*}
    It therefore suffices to prove that $\mathsf{m}^\nabla$-almost surely, 
    \begin{align*}
        \limsup_{N\to\infty} \frac{1}{|\Lambda_N|} \sum_{t\in \Lambda_N}\mathbf 1_{\theta_{t} A} \le p
    \end{align*}
    We now apply Theorem 6.2.1 of~\cite{Sheffield-Random-Surfaces} to argue that for all $k$, restricting to this event as well as boundary conditions that are exactly a random rounding of the hyperplane $\langle v,\mathfrak{s}\rangle$, denoted here $\phi_{\mathfrak{s}}$, is  equivalent on the $e^{\Theta(L^{d})}$-scale to the full measure. To be more precise, following the notation of Theorems 6.1.1 and 6.2.1 of~\cite{Sheffield-Random-Surfaces}, let 
    \begin{align*}
        B = \{\nu^\nabla: \nu^\nabla(A) \ge p +\delta\} \quad C_N^{\mathfrak{s}} = \{\varphi (\partial \Lambda_N) = \phi_\mathfrak{s}(\partial \Lambda_N)\}
    \end{align*}
    and the empirical version of $B$, given by 
    \begin{align*}B_N = \{\varphi: \frac{1}{|\Lambda_N|} \sum_{t\in \Lambda_N} \delta_{\theta_t \varphi}\in B\} = \{\varphi: \frac{1}{|\Lambda_N|} \sum_{t\in \Lambda_N}\mathbf 1_{\theta_t A} \ge p+\delta\}
    \end{align*} 
    Then, define the pinned boundary free energy as 
    \begin{align*}
        PBL_{B}^{\mathfrak{s}}  =\limsup_{N\to\infty } - \frac{1}{|\Lambda_N|}\log Z_{\Lambda_N}^{\mathfrak{s}}(C_N^{\mathfrak{s}} \cap B_N) 
    \end{align*}
    meaning the free energy (log-probability) associated to configurations on $\Lambda_N$ (extended arbitrarily outside $\Lambda_N$) with the hyperplane boundary conditions pinned by $C_N^{\mathfrak{s}}$ and on the event $B_N$. Suppose by way of contradiction that $\mathsf{m}^\nabla \in B$ meaning that $\mathsf{m}^\nabla(A) \ge p+\delta$. 
    Then  Theorem 6.2.1 of~\cite{Sheffield-Random-Surfaces} says that for $SFE(\mathsf{m}^\nabla)$ being its specific free energy as discussed in Section~\ref{sec:background gradient Gibbs measures},~\eqref{eq:SFE def}, 
    \begin{align*}
        PBL_{B}^{\mathfrak{s}}  \le SFE(\mathsf{m}^\nabla)
    \end{align*}
    Next, we observe that every configuration on $\Lambda_N$ (with free boundary) in $C_N^{\mathfrak{s}} \cap B_N$ is mapped to a configuration on $\mathbb T_N^d$, with a change of energy of the sloped-periodic interactions of $\varphi(\partial \Lambda_N)$. But by definition of the sloped periodic interaction, and the fact that $\mathfrak{s}\in \mathfrak{S}^\circ$ from~\eqref{eq:admissible-slopes}, this has $O(N^{d-1})$ energy cost along the sloped-periodic boundary. Thus   
    \begin{align*}
         e^{o(N^d)} e^{-|\Lambda_N| SFE(\mu)}\le  Z_{\Lambda_N}^{\mathfrak{s}}(C_N^{\mathfrak{s}} \cap B_N)   \le  e^{O(N^{d-1})} Z_{\mathbb T_N^d}^{\mathfrak{s}}(B_N)
    \end{align*}
    with $Z_{\mathbb T_N^d}^{\mathfrak{s}}(B_N)$ being the free energy on the torus with sloped-periodic boundary conditions, restricted to the event $B_N$. (Technically, there might be a loss in the events $B_N$ for the translates that intersect the boundary of $\Lambda_N$, but these are  an $O(N^{-1})$ fraction, and can be absorbed into the $\delta$.) 
    That restricted torus free energy can be rewritten as $\mu_{N,\mathfrak{s}}^\nabla (B_N) Z_{\mathbb T_N^d}^{\mathfrak{s}}$. Lemmas 8.2.7--8.2.8 of~\cite{Sheffield-Random-Surfaces} say that the free energy of the sloped-periodic measure is minimal, i.e., converges to $SFE(\mathsf{m}^\nabla)$ as $\mathsf{m}^\nabla$ is also minimal, so altogether, we deduce  
    \begin{align}\label{eq:nts-is-a-contradiction}
        \mathsf{m}^\nabla\in B \implies \mu^\nabla_{N,\mathfrak{s}}(B_N) \ge e^{-o(N^d)}\,.
    \end{align}
    It remains to establish that for each $\delta>0$, the right-hand side of~\eqref{eq:nts-is-a-contradiction} does not hold, giving the desired contradiction.

    For that purpose, recall that $M$ was such that all the edges $e_{1},...,e_m$ belong to $\Lambda_M$ and split  
    \begin{align*}
        \frac{1}{|\Lambda_{N}|} \sum_{t\in \Lambda_{N}} \mathbf 1_{\theta_t A} \le  \max_{r\in \Lambda_M}\frac{1}{|\Lambda_{N/M}|} \sum_{t\in \Lambda_{N/M}} \mathbf 1_{\theta_{r+ Mt} A}
    \end{align*}
    and therefore it suffices to establish that all the densities on the right-hand side are at most $p+\delta$ simultaneously except with probability $e^{- \Omega(N^d)}$. The benefit now is that for each fixed $r$ the translates of $\Lambda_M$ appearing above are disjoint. 
    In particular, 
    \begin{align*}
        \mu_{N,\mathfrak{s}}^\nabla (B_N) \le M^d \max_{r\in \Lambda_M} \mu_{N,\mathfrak{s}}^\nabla(B_{N,r}) \qquad \text{where} \qquad B_{N,r} = \Big\{\frac{1}{|\Lambda_{N/M}|} \sum_{t\in \Lambda_{N/M}} \mathbf 1_{\theta_{r+ Mt} A} \ge p+\delta\Big\}\,.
    \end{align*}
    By translation invariance, it is sufficient to consider $B_{N,0}$. The probability of $B_{N,0}$ is the probability that (letting $n=|\Lambda_{N/M}|$) $S_{n} = \sum_{t=1}^n X_t$ for $X_t = \mathbf 1_{\theta_{Mt}A}$ exceeds $(p+\delta)n$. 
    
    For any fixed subset $T \subset \Lambda_{N,M}$, the event $\bigcap_{t\in T} \{\theta_{Mt} A\}$ is equivalent to $\bigcap_{\tau \in \mathcal T}\tau \mathcal E_{\tau}$ where $\mathcal T$ is a subset of the horizontal reflections, and the $\mathcal E_{\tau}$ are events of the form that in $\{0,...,M\}^2 \times \mathbb T_N^{d-2}$ all the vertical translates of $A$ that are in the subset $T$ occur (possibly reflected horizontally so that $\tau$ acts as a shift of $\{0,...,M\}^2 \times \mathbb T_N^{d-2}$). Therefore, by the chessboard estimate Corollary~\ref{cor:gradient-chessboard-estimate}  with Theorem~\ref{thm:1-2-direction-rigidity}, 
    \begin{align*}
        \mu^\nabla_{N,\mathfrak{s}} \Big[\prod_{t\in T} X_t\Big] \le p^{|T|}\,.
    \end{align*}
    It follows that the exponential moments of $S_n$ are bounded by those of a $\text{Bin}(n,p)$ random variable (by writing $\mathbb E[e^{ \lambda S} ] = \mathbb E[\prod_t(1+(e^{\lambda} -1)X_t)]$ and expanding). Therefore, by the Chernoff bound, 
    \begin{align*}
        \mu_{N,\mathfrak{s}}^\nabla (B_{N,0}) \le e^{ - n D(p+\delta \,\|\, p)}
    \end{align*}
    where $D(\cdot \,\| \,\cdot)$ is the KL-divergence. The above is $\exp( - \Omega(N^{d}))$ for any positive $\delta>0$. Therefore, multiplying by $M^d$, we get $\mu_{N,\mathfrak{s}}^\nabla(B_N) \le e^{ - \Omega(N^{d})}$ yielding the desired contradiction to~\eqref{eq:nts-is-a-contradiction}. 
\end{proof}

\section{Determining the dominant height profile}\label{sec:extremal-components}

In what follows, $\mu_{\mathfrak{s}}$ is any height function Gibbs measure of slope $\mathfrak{s}$ with ergodic minimal gradients.  Therefore by Lemma~\ref{lem:ergodic-minimal-has-dominant-height}, it satisfies~\eqref{eq:exp-suppressed-gradients-condition} and has densely dominant heights. 
We will connect the results of \cite{Sheffield-Random-Surfaces} summarized in Section~\ref{sec:background gradient Gibbs measures} to dominant heights in order to deduce the distribution over dominant height sequences in $\mu_{\mathfrak{s}}$. In order to talk about height function measures while not worrying about (deterministic or tail measurable) global shifts of the height, which retain the Gibbs property and give the same gradient measure, in what follows we always apply a global shift to have the dominant height at layer zero, $h(\mathcal L_0) =0$.

\subsection{The slope ${\mathfrak{s}} = (0,0, \mathfrak{s}_{3},\ldots, \mathfrak{s}_{d})$ has only rational components}

The analysis that we perform in this section will rely heavily on  Theorem \ref{thm:Sheffield}, case 1.

If we write $\mathfrak{s}_i = \frac{p_i}{q_i}$, where $p_i$ and $q_i$ are relatively prime, let $n$ be the least common multiple of the $q_i$'s.
Then, Theorem \ref{thm:Sheffield}  will state that $\mu_{\mathfrak{s}}$ has exactly $n$ extremal components. Namely,
\begin{equation}\label{eq:extremaldecomp}
\mu_\mathfrak{s} = {\frac{1}{n}}\sum_{i=0}^{n-1} \mu_{\mathfrak{s},\mathfrak{c} + \frac{i}{n}} .
\end{equation}
The set $\{\mathfrak{c},\mathfrak{c} + \frac{1}{n}, \ldots, \mathfrak{c}+\frac{n-1}{n} \}$ corresponds to what is called the \textit{height-offset spectrum} and the subindex $\mathfrak{c} + \frac{i}{n}$ denotes the average height in a box around the origin of the Gibbs measure. 
 Note that since $\mu_\mathfrak{s}$ has dominant heights with probability $1$, each of these extremal components also has a dominant height. We will use the notation $h_{\mu_{\mathfrak{s},a}}(x_{3},\ldots,x_{d})$ to denote the dominant height on the layer whose last $d-2$ components are given by $x_{3},\ldots,x_{d}$ of the extremal component $\mu_{\mathfrak{s},a}$.

Now, we claim that by applying Theorem \ref{thm:Sheffield}, we can establish the following lemma concerning the dominant height sequence in any one of these extremal components.

\begin{lem} \label{lem:controldiff}
Consider some measure $\mu_{\mathfrak{s},a}$ that lies in the extremal decomposition of $\mu$. Let $h_{\mu_{\mathfrak{s},a}}(x_{3},x_{4},\ldots,x_{d})$ be the dominant height on the layer whose last $d-2$ coordinates are given by $x_{3},\ldots,x_{d}$. Then, we have
$$
h_{\mu_{\mathfrak{s},a}}(x_{3}',x_{4}',\ldots,x_{d}') - h_{\mu_{\mathfrak{s},a}}(x_{3},\ldots,x_{d}) \in \bigg\{\bigg\lfloor \sum_{k=3}^{d} \mathfrak{s}_k(x'_k - x_k)\bigg\rfloor, \bigg\lceil \sum_{k=3}^{d} \mathfrak{s}_k(x'_k - x_k) \bigg\rceil \bigg\}.
$$

\end{lem}

\begin{proof}
We remark that before we start this proof, that the only elements of Theorem~\ref{thm:Sheffield}  that we will use in what follows are Parts 2,3, and 6 of Case 1.

If the sum $\sum_{k=3}^{d} \mathfrak{s}_k(x'_k - x_k)$ is an integer, then spatial translation symmetry, Case 1, Part 6 of Theorem~\ref{thm:Sheffield},  will give us that $\theta_{(0,0,x_{3}'-x_{3},x_{4}'-x_{4},\ldots,x_{d}'-x_{d})} \mu_{\mathfrak{s},a}$ 
will give the measure $\mu_{\mathfrak{s},a + \sum_{k=3}^{d} \mathfrak{s}_k(x_k - x_k')}$. Then, height translation symmetry, Case 1, Part 3 of Theorem \ref{thm:Sheffield}, will show that this measure is the same as $\sum_{k=3}^{d} \mathfrak{s}_k(x_k - x_k') + \mu_{\mathfrak{s},a}$. Thus, they would have the same dominant heights. Translating all these statements would show that $h_{\mu_{\mathfrak{s},a}}(x_{3},\ldots,x_{d}) = h_{\mu_{\mathfrak{s},a}}(x_{3}',\ldots,x_{d}') + \sum_{k=3}^{d} \mathfrak{s}_k(x_k-x_k')$

Now, let us consider the case when $\sum_{k=3}^{d} \mathfrak{s}_k(x'_k - x_k)$ is not an integer.
Consider the measure $-\lfloor \sum_{k=3}^{d} \mathfrak{s}_k(x_k'-x_k) \rfloor + \theta_{(0,0,x_{3}'- x_{3},\ldots,x'_{d}-x_{d}) } \mu_{\mathfrak{s},a} = \mu_{\mathfrak{s},a + \sum_{3}^{d} \mathfrak{s}_k(x'_k - x_k) - \lfloor \sum_{k=3}^{d} \mathfrak{s}_k(x_k'-x_k) \rfloor} \succ \mu_{\mathfrak{s},a}$. 
Again, here we applied spatial translation symmetry and height translation symmetry. The domination inequality is Case 1, Part 2 of Theorem \ref{thm:Sheffield}.

We remark that if $\nu \succ \mu$ (and they both have dominant heights), then the dominant heights of $\nu$ are greater than those of $\mu$. By comparing the dominant heights at the point $(x_{3},\ldots,x_{d})$, we have 
$$
h_{\mu_{\mathfrak{s},a}}(x_{3}',\ldots,x_{d}') -  \bigg\lfloor \sum_{k=3}^{d} \mathfrak{s}_k(x_k'-x_k) \bigg\rfloor \ge h_{\mu_{\mathfrak{s},a}}(x_{3},\ldots,x_{d}).
$$

If we instead considered, $$-\bigg\lceil \sum_{k=3}^{d} \mathfrak{s}_k(x_k'-x_k) \bigg\rceil + \theta_{(0,0,x_{3}'- x_{3},\ldots,x_{d}-x_{d}) } \mu_{\mathfrak{s},a} = \mu_{\mathfrak{s},a + \sum_{3}^{d} \mathfrak{s}_k(x'_k - x_k) - \lceil \sum_{k=3}^{d} \mathfrak{s}_k(x_k'-x_k) \rceil} \prec \mu_{\mathfrak{s},a},$$
we have
$$
h_{\mu_{\mathfrak{s},a}}(x_{3}',\ldots,x_{d}') -  \bigg\lceil \sum_{k=3}^{d} \mathfrak{s}_k(x_k'-x_k) \bigg\rceil \le h_{\mu_{\mathfrak{s},a}}(x_{3},\ldots,x_{d}).
$$
This is what we desired.
\end{proof}

With this lemma in hand, we can now establish the following.

\begin{lem} \label{lem:heightoffset0}
Consider any measure $\mu_{\mathfrak{s},a}$ in the extremal decomposition of $\mu_\mathfrak{s}$. Then, there exists some point $(x_{3},\ldots,x_{d})$ such that
\begin{equation} \label{eq:heightdiffs}
h_{\mu_{\mathfrak{s},a}}(x'_{3},\ldots,x'_{d}) - h_{\mu_{\mathfrak{s},a}}(x_{3},\ldots,x_{d}) = \bigg\lfloor \sum_{k=3}^{d} \mathfrak{s}_k(x'_k - x_k) \bigg\rfloor ,
\end{equation}
for all other $(x'_{3},\ldots,x'_{d})$.
\end{lem}
\begin{proof}
We will show that there is a point $(x_{3},\ldots,x_{d})$ such that the equation  \eqref{eq:heightdiffs} 
holds whenever $x_k' \ge x_k$ for all $k$. Note that by the periodicity relationship that $h_{\mu_{\mathfrak{s},a}}(x_{3}',x_{4}',\ldots,x_{d}) - h_{\mu_{\mathfrak{s},a}}(x_{3},\ldots,x_{d}) = \sum_{k=3}^{d} \mathfrak{s}_k (x_k' -x_k)$, whenever $\sum_{k=3}^{d} \mathfrak{s}_k(x_k' -x_k)$ is an integer, this means that we can always determine the value of $h_{\mu_{\mathfrak{s},a}}(x_{3}',\ldots, x_k' - n, x_{k+1}',\ldots,x_{d}')  $ from $h_{\mu_{\mathfrak{s},a}}(x_{3}',\ldots, x_k' , x_{k+1}',\ldots,x_{d}')$.

Assume for contradiction that for every point $(x_{3},\ldots,x_{d})$, there is some point $(x'_{3},\ldots,x'_{d})$ with $x_j' \ge x_j$ for all $j$ such that $h_{\mu_{\mathfrak{s},a}}(x'_{3},\ldots,x'_{d}) - h_{\mu_{\mathfrak{s},a}}(x_{3},\ldots,x_{d}) = \lceil \sum_{k=3}^{d} \mathfrak{s}_k(x'_k - x_k) \rceil$ and that this value is strictly greater than $\lfloor \sum_{k=3}^{d} \mathfrak{s}_k(x'_k - x_k) \rfloor$ (so $\sum_{k=3}^{d} \mathfrak{s}_k(x'_k - x_k) $ is not an integer.) By the periodicity relationship we mentioned earlier, we may assume that $|x'_{k} -x_k| \le n$.

Then for every point $v \in \mathbb{Z}^{d-2} $, there is another point $t(v) $ with $t(v) - v \in [0,n]^{d-2}$ and $h_{\mu_{\mathfrak{s},a}}(t(v)) - h_{\mu_{\mathfrak{s},a}}(v) = \lceil \langle \mathfrak{s}, t(v) -v \rangle \rceil > \langle \mathfrak{s}, t(v) - v \rangle + \epsilon,$
where this $\epsilon>0$ will be uniform for all points $v$. (This $\epsilon$ is uniform since we have a priori enforced that $t(v) - v $ is in a compact domain.) Note that when writing $\langle \mathfrak{s}, v \rangle$, we are implicitly dropping the first two entries of $\mathfrak{s}$ (as they are both $0$) and, thus, we can treat $\mathfrak{s}$ as a $d-2$ dimensional vector.

Now, pick some arbitrary vertex $v_0$ and consider the sequence $v_i = t(v_{i-1})$ for $i \ge 1$. Then, we must have that,
\begin{equation}
h_{\mu_{\mathfrak{s},a}}(v_i) - h_{\mu_{\mathfrak{s},a}}(v_0) \ge \sum_{k=0}^{i-1} h_{\mu_{\mathfrak{s},a}}(v_{k+1}) - h_{\mu_{\mathfrak{s},a}}(v_k) \ge \sum_{k=0}^{i-1} [\langle \mathfrak{s},v_{k+1} - v_k \rangle +\epsilon ]  = \langle \mathfrak{s}, v_{i} -v_0 \rangle + i\epsilon. 
\end{equation}
For $i$ large enough, we can set $i\epsilon \ge 2$. That contradicts $h_{\mu_{\mathfrak{s},a}}(v_i) - h_{\mu_{\mathfrak{s},a}}(v_0) \le \lceil \sum_{k=3}^{d} \mathfrak{s}_k((v_i)_k - (v_0)_k) \rceil.$ 
\end{proof}

\begin{cor} \label{cor:dominheight}
All the measures in the extremal decomposition of $\mu_{\mathfrak{s}}$ have the following dominant height profiles. For each $a\in \mathfrak{c}+ \frac{i}{n}$ there is some rational number of the form  $\frac{k}{n}$ such that 
\begin{equation}\label{eq:heightdecomp}
   h_{\mu_{\mathfrak{s},a}}(x_{3},\ldots,x_{d}) = 
   \left\lfloor \frac{k}{n} + \sum_{j=3}^{d}\mathfrak{s}_j x_j \right\rfloor. 
\end{equation}

\end{cor}

\begin{proof}
By the previous Lemma \ref{lem:heightoffset0}, for the measure $\mu_{\mathfrak{s},a}$ in the extremal decomposition, there is a point $(\hat{x}_{3},\ldots,\hat{x}_{d})$ such that we have $$h_{\mathfrak{s},\alpha}(x_{3}',\ldots,x'_{d} ) = h_{\mathfrak{s},a}(\hat{x}_{3},\ldots,\hat{x}_{d}) +\Big\lfloor\sum_{k=3}^{d} \mathfrak{s}_k(x_k'-\hat{x}_k) \Big\rfloor\,.$$ In other words, the dominant heights of $\mu_{\mathfrak{s},\alpha}$ can be written as
\begin{equation}
h_{\mu_{\mathfrak{s},a}}(x_{3},\ldots,x_{d}) = \left \lfloor  h_{\mu_{\mathfrak{s},a}}(\hat{x}_{3},\ldots,\hat{x}_{d})  - \sum_{j=3}^{d} \mathfrak{s}_j \hat{x}_j + \sum_{j=3}^{d} \mathfrak{s}_j x_j\right\rfloor
\end{equation}
{The first two terms in the floor evidently are a rational number with denominator $n$.} 
\end{proof}

\subsection{The slope $\mathfrak{s}$ has at least 1 irrational component}

This section will rely instead on Case 2 of Theorem  \ref{thm:Sheffield}. Here, extremal measures can be decomposed as,
$$
\mu_{\mathfrak{s}} = \int_0^1 \mu_{\mathfrak{s},\mathfrak{c}+a} \text{d}a,
$$
for some $\mathfrak{c}\in \mathbb{R}$.

{

 To describe the properties of all extremal Gibbs measures, we will introduce a weaker version of the dense dominant height notion from Definition~\ref{def:density-dominant-heights}, which will relax 95\% density of the dominant height level set to 80\%. 

\begin{defn}[Weak Dominant Height]
We will say that an extremal measure $\mu$ has a weak dominant height on the layer with coordinates given by  $\vec{x}:= (x_3,\ldots,x_d)$ if there is a value $h_{\mu}(x_3,\ldots,x_d)$ such that almost surely, for all $N$ sufficiently large, at least a $0.8$ fraction of all values $\phi$ takes on the box $[-N,N]^2 \times \{ \vec{x}\}$ are $h_{\mu}(x_3,\ldots,x_d)$.
\end{defn}

By Lemma~\ref{lem:ergodic-minimal-has-dominant-height}, almost every extremal component of $\mu_{\mathfrak{s}}$ satisfies items (1)--(2) of Definition~\ref{def:density-dominant-heights}. Moreover, for any of these extremal components, if they have a weak dominant height, that must be the same as its dense dominant height. Therefore, there is no ambiguity in using the same notation $h_\mu$ for the weak dominant height of an extremal measure as for its dominant height in the sense of (1)--(2) of Definition~\ref{def:density-dominant-heights} or in the sense of Definition~\ref{defn:dominant-height}. 

The following lemma is essentially the only new ingredient necessary for the irrational case.
\begin{lem} \label{lem:contheights}
Let $a_1,a_2,\ldots,a_n,\ldots$ be a sequence of numbers such that $a_i \downarrow a \in \mathbb{R}$ and  all of the measures $\mu_{\mathfrak{s},a_i}$ have densely dominant heights. Furthermore, assume that there exists some $\tilde a< a$ such that $\mu_{\mathfrak{s},\tilde{a}}$ has densely dominant heights. Then, $\mu_{\mathfrak{s},a}$ has weak dominant heights and 
\begin{equation}
h_{\mu_{\mathfrak{s},a}}(x_{3},\ldots,x_{d}) = \lim_{n \to \infty} h_{\mu_{\mathfrak{s},a_n}}(x_{3},\ldots,x_{d})
\end{equation}

\end{lem}
\begin{proof}
Define $L_{(x_{3},\ldots,x_{d})} = \lim_{n \to \infty} h_{\mu_{\mathfrak{s},a_n}}(x_{3},\ldots,x_{d}).$ Notice, that by the stochastic domination of Case 2 of Theorem  \ref{thm:Sheffield} , we have $\mu_{\mathfrak{s},a_i} \succ \mu_{\mathfrak{s},a_j}$ if $a_i>a_j$ and, as a consequence, we have that $h_{\mu_{\mathfrak{s},a_i}}(x_{3},\ldots,x_{d}) \ge h_{\mu_{\mathfrak{s},a_j}}(x_{3},\ldots,x_{d})$ if $i < j$.
Thus, the limit must exist or be $-\infty$.

Stochastic domination also dictates that $\mu_{\mathfrak{s},\tilde{a}} \prec \mu_{\mathfrak{s},a_i}$ for any $i$. Thus, $h_{\mu_{\mathfrak{s},\tilde{a}}}(x_{3},\ldots,x_{d}) \le h_{\mu_{\mathfrak{s},a_i}}(x_{3},\ldots,x_{d}).$ Since $h_{\mu_{\mathfrak{s},\tilde{a}}}(x_{3},\ldots,x_{d})$ is a finite number, $L_{(x_{3},\ldots,x_{d})}$ cannot be $-\infty$. Furthermore, since $h_{\mu_{\mathfrak{s},a_i}}(x_{3},\ldots,x_{d})$ is an integer sequence, 
 there must be some $K$ such that $h_{\mu_{\mathfrak{s},a_K}}(x_{3},\ldots,x_{d}) = L_{(x_{3},\ldots,x_{d})}$.
{ We remark that the dominant heights are tail-measurable, so on our extreme components, they must almost surely be constant. }

Let $A  = A_{(x_{3},\ldots,x_{d})}$ be the increasing event that  for all sufficiently large $N$, at least $90 \%$ of all vertices inside the box $[-N,N]^{2}\times \{\vec{x}\}$ take value greater than or equal to $L_{(x_{3},\ldots,x_{d})}$.  

We have that $\mu_{\mathfrak{s},a_i}(A) = 1$ for all $i$ since $\mu_{\mathfrak{s},a_i}$ has densely dominant heights greater than $$L_{(x_{3},\ldots,x_{d})}$$  By right continuity on increasing events of Case 2 of Theorem \ref{thm:Sheffield}, we have that $\mu_{\mathfrak{s},a}(A) = \lim \mu_{\mathfrak{s},a_i}(A) =1$. Note, that we also have the stochastic domination $\mu_{\mathfrak{s},a} \prec \mu_{\mathfrak{s},a_K}$, where $\mu_{\mathfrak{s},a_K}$ has densely dominant height  $L_{(x_{3},\ldots,x_{d})}$. The only way that $\mu_{\mathfrak{s},a}(A) = 1$ and the domination relation $\mu_{\mathfrak{s},a} \prec \mu_{\mathfrak{s},a_K}$ both hold is if $\mu_{\mathfrak{s},a}$ has  weak  dominant height $L_{(x_{3},\ldots,x_{d})}$, as the intersection of the two level sets of asymptotically at least 90\% density must have density at least 80\%. 
\end{proof}

With this lemma in hand, we can now have the following corollary.
\begin{cor}
For every $\tilde{a}$, $\mu_{\mathfrak{s},\tilde{a}}$ has weak dominant heights. 

\end{cor}
\begin{proof}
Note that if $\mu_{\mathfrak{s},\tilde a}$ has densely dominant height, it must automatically have weak dominant heights. Furthermore, since the measure $\mu_\mathfrak{s}$ satisfies that, via the decomposition $\mu_\mathfrak{s} = \int_0^1 \mu_{\mathfrak{s},\mathfrak{c} +a} \text{d}a$, if we let $A \subset(0,1)$ be the collection of $a$ such that $a \in A$ implies $\mu_{\mathfrak{s}, \mathfrak{c}+a}$ has item (2) of Definition~\ref{def:density-dominant-heights}, then, $A$ has measure 1. Since $A$ has measure $1$, for every $b \in A^c$, there must be a sequence $(\alpha_i)_{i\ge 0}\in A$ decreasing to $b$,  and further, there exists some $\tilde{\alpha}<b$ such that $\tilde \alpha \in A$. Thus, we can apply Lemma \ref{lem:contheights} to deduce that $\mu_{\mathfrak{s},b}$ has weak dominant heights. By the height translation property, we can transfer this statement to all $\mu_{\mathfrak{s},a} $ for $a\in (-\infty,\infty).$  
\end{proof}

The following lemma has the same proof as that of Lemma \ref{lem:controldiff}, with the Case 2 analogs replacing the Case 1 versions. 

\begin{lem} \label{lem:heightoffsetirr}
Consider some measure $\mu_{\mathfrak{s},a}$ that lies in the extremal decomposition of $\mu_{\mathfrak{s}}$. Let $h_{\mu_{\mathfrak{s},a}}(x_{3},\ldots,x_{d})$ be the weak dominant height on the layer whose last $d-2$ coordinates are given by $(x_{3},\ldots,x_{d})$. Then, we have
$$
h_{\mu_{\mathfrak{s},a}}(x_{3}',\ldots,x_{d}') - h_{\mu_{\mathfrak{s},a}}(x_{3},\ldots,x_{d}) \in \bigg\{\bigg\lfloor \sum_{k=3}^{d} \mathfrak{s}_k(x'_k - x_k)\bigg\rfloor, \bigg\lceil \sum_{k=3}^{d} \mathfrak{s}_k(x'_k - x_k) \bigg\rceil \bigg\}\,.
$$

\end{lem}

We cannot get the exact same statement as Lemma \ref{lem:heightoffset0}, but the following lemma has the same proof, and will serve almost the same purpose in the end.
\begin{lem} \label{lem:somediff}
Consider any measure $\mu_{\mathfrak{s},a}$ in the extremal decomposition of $\mu_\mathfrak{s}$ and fix some $B \in \mathbb{N}$. Then, there exists some point $(x_{3},\ldots,x_{d})$ such that
\begin{equation} \label{eq:heighdiffs}
h_{\mu_{\mathfrak{s},a}}(x'_{3},\ldots,x'_{d}) - h_{\mu_{\mathfrak{s},a}}(x_{3},\ldots,x_{d}) = \bigg\lfloor \sum_{k=3}^{d} \mathfrak{s}_k(x'_k - x_k) \bigg\rfloor ,
\end{equation}
for all other $(x'_{3},\ldots,x'_{d}) \in [-B,B]^{d-2}$.
\end{lem}

The following will be the main consequence of the preceding lemma and Lemma \ref{lem:contheights}.

\begin{lem}
There exists {an $a\in (-\infty,\infty)$ and} a measure of the form $\mu_{\mathfrak{s},a}$ whose   weak  dominant heights are of the form,
\begin{equation} \label{eq:heightoffset}
h_{\mu_{\mathfrak{s},a}}(x_{3},\ldots,x_{d}) = \bigg\lfloor \sum_{k=3}^{d} \mathfrak{s}_k  x_k \bigg\rfloor.
\end{equation}
\end{lem}

\begin{proof}
Shifting to have dominant height $0$ at layer $0$, we can apply Lemma \ref{lem:somediff} %
combined with height and spatial translation invariance to  construct a sequence of measures $\mu_{\mathfrak{s}, a_i}$ having densely dominant heights, and whose  dominant heights satisfy, 
$$
h_{\mu_{\mathfrak{s},a_i}}(x_{3},\ldots,x_{d})=  \bigg\lfloor \sum_{k=3}^{d} \mathfrak{s}_k  x_k \bigg\rfloor\,, \qquad\forall (x_{3},\ldots,x_{d}) \in [-i,i]^{d-2}\,.
$$

Now, if one of these $\mu_{\mathfrak{s},a_i}$ satisfies equation \eqref{eq:heightoffset} everywhere, then we are done. If not, we claim that we can find a decreasing sequence of $a_i$ with finite limit $a$.

If $\mu_{\mathfrak{s},a_i}$ does not satisfy equation \eqref{eq:heightoffset}, then there exists some $(\tilde{x}_{3},\ldots,\tilde{x}_{d})$ with $$h_{\mu_{\mathfrak{s},a_i}}(\tilde{x}_{3},\ldots,\tilde{x}_{d})  = \bigg\lceil \sum_{k=3}^{d} \mathfrak{s}_k  \tilde{x}_k \bigg\rceil\, > \bigg\lfloor \sum_{k=3}^{d} \mathfrak{s}_k  \tilde{x}_k \bigg\rfloor.$$  We must thenceforth have $a_i \ge a_{k}$ for any $k > \max_{i=3}^{d}|\tilde{x}_i|$. Indeed, if, for contradiction, we had $a_i <a_k$, we must have the stochastic domination $\mu_{\mathfrak{s},a_i} \prec \mu_{\mathfrak{s},a_k}$. This in turn would imply that the weak dominant heights of $\mu_{\mathfrak{s},a_i}$ are all less than or equal to the weak dominant heights of $\mu_{\mathfrak{s},a_k}$. However, the dominant height of $\mu_{\mathfrak{s},a_i}$ at $(\tilde{x}_3,\ldots,\tilde{x}_d)$ is larger than that of $\mu_{\mathfrak{s},a_{k}}$  by construction. 

Thus, we can find some decreasing subsequence $a_{s_1}, a_{s_2},\ldots,a_{s_n},\ldots$. To see that this is bounded below, note that the  weak dominant height of this sequence at $(0,\ldots,0)$ is $0$, but the weak dominant height of  $\mu_{\mathfrak{s},a_{s_1}} -1 = \mu_{\mathfrak{s},a_{s_1}-1}$  at $(0,\ldots,0)$ is $-1$. Thus, by stochastic domination, we must have $a_{s_j} \ge a_{s_1}-1$ for all $j$.

Thus, there exists some $\hat{a}= \lim_{j\to \infty} a_{s_j}$ and by Lemma \ref{lem:contheights} the measure $\mu_{\mathfrak{s},\hat{a}}$ satisfies our desired properties. 
\end{proof}

Our final corollary is the following. 
\begin{cor} \label{cor:fincor}
For almost every measure $\mu_{\mathfrak{s},a}$ that belongs to the extremal decomposition of $\mu_\mathfrak{s}$, there exists some $r \in \mathbb R$ such that,
\begin{equation}
h_{\mu_{\mathfrak{s},a}}(x_{3},\ldots,x_{d}) = \bigg\lfloor r+ \sum_{k=3}^{d} \mathfrak{s}_k  x_k \bigg\rfloor.
\end{equation}

\end{cor}

\begin{proof}
Consider the measure $\mu_{\mathfrak{s},c}$ that satisfies equation \eqref{eq:heightoffset}. By spatial translation symmetry, we see that for any $\mu_{\mathfrak{s}, c+ \sum_{k=3}^{d} \mathfrak{s}_k \tilde{x}_k}$, we have
$$
h_{\mu_{\mathfrak{s},c+ \sum_{k=3}^{d} \mathfrak{s}_k \tilde{x}_k }}(x_{3},\ldots,x_{d}) =  \bigg\lfloor \sum_{k=3}^{d} \mathfrak{s}_k \tilde{x}_k + \sum_{k=3}^{d} \mathfrak{s}_k  x_k \bigg\rfloor.
$$
We set $r= \sum_{k=3}^{d} \mathfrak{s}_k \tilde{x}_k $. 
Notice that since $\mathfrak{s}$ has an irrational coordinate, terms of the form $\sum_{k=3}^{d} \mathfrak{s}_k \tilde{x}_k$ modulo $1$ form a dense set of $[0,1]$, and any $v \in \mathbb{R}$ can be obtained up to a global integer shift as a limit $\lim_{i \to \infty } \sum_{k=3}^{d} \mathfrak{s}_k \tilde{x}^i_k$ for $i \in \mathbb{N}$ and an appropriate sequence $(x^i_{3},\ldots, x^i_{d}) \in \mathbb{Z}^{d-2}$. We can now apply Lemma \ref{lem:contheights} (technically now with each of the measures having a weak dominant height, which will be enough to guarantee asymptotically 60\% of vertices in each layer $\mathbb Z^2\times \{\vec{x}\}$  take height $h_\mu(x_3,...,x_d)$, but which still implies that is the same as the dense dominant height for the full measure set of extremal measures having dense dominant heights) and conclude. 
\end{proof}

\section{Uniqueness and exponential decay in extremal components}\label{sec:uniqueness-exp-decay}

In this section, our goal is to show that any two measures with ergodic gradient measures have the same distribution over dominant height sequences, have exponential decay of excitations away from that dominant height, and indeed must be the same measure modulo global integer shift. That is, we will obtain the uniqueness part of Theorem~\ref{thm:main-localization}. The same argument will also give exponential decay of correlations within the extremal measures establishing Theorem~\ref{thm:structural results-exponential-decay}.

\begin{thm} \label{thm:disagreementperc}
Let $\mu_1,\mu_2$ be height Gibbs measures of slope $\mathfrak{s}$ with ergodic, minimal gradients. Then $\mu_1^\nabla$ and $\mu_2^\nabla$ are the same.
\end{thm}

By Lemma~\ref{lem:ergodic-minimal-has-dominant-height}, both $\mu_1$ and $\mu_2$ must have dominant heights. Since we are interested in their properties modulo global integer shifts (i.e., their gradient measures), we can without loss of generality set both their dominant heights at layer zero to be zero in what follows. 

Before we can proceed with the direct proof, we need to first define some preliminary notions.
In what follows, we use $B_r$ to denote $[-r,r]^d$, the centered box of side length $2r$.

\begin{defn}[Horizontal Cluster]
In any measure $\mu$ satisfying densely dominant heights, almost every extremal component of $\mu$ has a dominant height. Given an extremal component that has dominant heights, we say that an \emph{excitation} is any vertex $v$ such that $\phi_v $ is not equal to the dominant height at that layer. 

In a configuration $\phi$, a horizontal cluster is a connected set of excited vertices within a layer. 
 Abusing notation, a  horizontal cluster 
is a connected set of vertices that lie within a single layer, not necessarily associated with a  configuration or measure. The interior of a horizontal cluster $HC$ in layer $\mathcal L_{\vec{k}} = \mathbb Z^2\times \vec{k}$ is the union of vertices in the horizontal cluster or in a finite connected component of $\mathcal L \setminus HC$.
\end{defn}

\begin{defn}[Connecting Horizontal clusters]

Given two sets $B_1$ and $B_2$, we say that $\mathcal{HC}= (HC_1,a_1),(HC_2,a_2),\ldots,(HC_m,a_m)$ is a connecting set of horizontal clusters if each of the sets $HC_i$ is a connected set of vertices that lie within a layer (though possibly different layers for each $i$) and each $a_i\in \{1,2\}.$ Furthermore, we desire that interior of the set $HC_i$ is $\mathbb Z^d$-connected to the interior of the set $HC_{i-1}$ for all $i \in \{2,\ldots m\}$  and that the interior of $HC_1$ intersects the set $B_1$ while the interior of $HC_m$ intersects the set $B_2$.
\end{defn}

\begin{defn}[Good Pair for $\mathcal{HC}$]
Given a pair of extremal measures $\pi_1$ and $\pi_2$, we say that $(\pi_1,\pi_2)$ is a good pair for a given connecting set of horizontal  clusters $\mathcal{HC}$ if we have that,
\begin{align*}
&\pi_1 \otimes \pi_2 \Big(\bigcap_{i:a_i=1} \{ HC_i  \text{ is a horizontal cluster for } \phi_1\} \cap \bigcap_{i: a_i =2} \{ HC_i \text{ is a horizontal cluster for } \phi_2 \}\Big) \\ \nonumber &
\qquad \le \exp\Big(- \frac{c \beta_0}{4} \sum_{i=1}^m |\partial_o(HC_i)| \Big),
\end{align*}
where $\partial_o(HC_i)$ is the outer boundary of the horizontal cluster, i.e., the number of dual edges in the horizontal contour that forms its outer boundary, and 
where $c \beta_0$ is as in~\eqref{eq:exp-suppressed-gradients-condition}.
\end{defn}

\begin{proof}[\textbf{\emph{Proof of Theorem~\ref{thm:disagreementperc}}}]
First, by the decomposition item of Theorem~\ref{thm:Sheffield}, together with Corollaries~\ref{cor:dominheight} and~\ref{cor:fincor}, we can decompose $\mu_i$ into its extremal components as $\int_{0}^1 \pi_{i,\alpha} \nu(\text{d}\alpha)$, with the $\alpha$ being in one-to-one correspondence with the dominant height sequence. If $\mathfrak{s}$ has only rational components with least common denominator $n$, then $\nu$ is the uniform measure on $\{0, \frac{1}{n},\ldots, \frac{n-1}{n}\}$. If $\mathfrak{s}$ has irrational coordinates, then $\nu$ is uniform on $[0,1]$. 

 In particular, for almost every $\alpha$, the pair $\pi_{1,\alpha},\pi_{2,\alpha}$ are extremal measures with the same dominant height sequences. We first sample a pair of configurations $(\phi_1,\phi_2)$ from the product measure $\pi_{1,\alpha} \otimes \pi_{2,\alpha}$. A cluster, $\mathcal{C}$, of a configuration $\phi$ is a maximal $\mathbb Z^d$-connected set of excited vertices.

\begin{defn}[Supercluster]
Let $\bigcup_{i=1}^{\infty} \mathcal{C}_{1,i}$ (respectively  $\bigcup_{i=1}^\infty \mathcal{C}_{2,i}$) be the union of the clusters of the configurations $\phi_1$ (respectively $\phi_2$). A \emph{supercluster} $\mathcal{S}$ will be a maximal union of connected clusters of the form $\mathcal{C}_{1,i}$ or $\mathcal{C}_{2,j}$, where two of these clusters are connected if one can find vertices $v \in \mathcal{C}_{1,i}$ and $w \in \mathcal{C}_{2,j}$ such that either $v$ and $w$ are connected by an edge or $v=w$. 
\end{defn}

Given any finite set $B\subset \Z^d$, if one can show that there is almost surely no infinite supercluster that intersects $B$, then as we will argue, the distribution of $\pi_{1,\alpha}$ on $B$ will be the same as the distribution of $\pi_{2,\alpha}$ on $B$.  
Let $\mathcal{S}_B$ be the union of superclusters that intersect  $B$. Between the configurations $\phi_1$ and $\phi_2$, one can swap the values of $\phi_1(\mathcal{X})$ and $\phi_2(\mathcal{X})$ on the points $\mathcal{X}$ of the clusters that belong to $\mathcal{S}_B$. Let $(\phi_1',\phi_2')$ be the new configurations derived after this swap. Notice, that the map $(\phi_1,\phi_2) \to (\phi_1',\phi_2')$ is an involution,  preserves probabilities (see a formal discussion in the next paragraph), and exchanges the values of the configurations $\phi_1,\phi_2$ on the inside of  $B$. The existence of this transformation means that $\pi_{1,\alpha}$ and $\pi_{2,\alpha}$ have the same distribution on~$B$.

 Formally, the notion of preserving probabilities can be expressed as follows. Fix integer $N>0$, let $E_1$ and $E_2$ be two cylinder events supported on  $B$, let $\Omega_N$ be the event that $\mathcal S_B$ does not intersect the boundary of the box $B_N$, and let $\xi_1$ and $\xi_2$ be boundary conditions on $\partial B_N$ for configurations $\phi_1$ and $\phi_2$. Then, we can write,
\begin{align}\label{eq:swapping-argument}
&\mathbb{P}(\phi_1 \in E_1, \phi_2 \in E_2) \nonumber \\
&\le \sum_{\xi_1,\xi_2} \mathbb{P}\left(\phi_1 \in E_1, \phi_2 \in E_2, (\phi_1,\phi_2) \in \Omega_N \mid \phi_1|_{\partial B_N} = \xi_1, \phi_2|_{\partial B_N} = \xi_2 \right)\mathbb{P}(\phi_1|_{\partial B_N}= \xi_1, \phi_2|_{\partial B_N} = \xi_2) \nonumber \\
&\qquad +\mathbb{P}((\phi_1,\phi_2)\in \Omega_N^c) \nonumber\\
& \le \sum_{\xi_1,\xi_2} \mathbb{P}\left(\phi_1' \in E_2, \phi_2' \in E_1, (\phi_1',\phi_2') \in \Omega_N \mid \phi_1'|_{\partial B_N} = \xi_1, \phi_2'|_{\partial B_N} = \xi_2 \right) \mathbb{P}(\phi_1'|_{\partial B_N}= \xi_1, \phi_2'|_{\partial B_N} = \xi_2) \nonumber \\& \qquad + \mathbb{P}((\phi_1, \phi_2) \in \Omega_N^c) \nonumber \\
& \le \mathbb{P}(\phi_1' \in E_2, \phi_2' \in E_1) + \mathbb{P}((\phi_1,\phi_2) \in \Omega_N^c). 
\end{align}
The second inequality comes from the fact that on the event $(\phi_1,\phi_2) \in \Omega_N$ the swapping operation does not affect the boundary condition, and preserves probabilities (as $\phi_1,\phi_2$ both take the same values---namely the common dominant height value---on the outer boundary of $\mathcal S_B$). 
We can similarly run the bound in the other direction and derive that \begin{align}\label{eq:tv-bound-by-supercluster-connection}|\mathbb{P}(\phi_1 \in E_1, \phi_2 \in E_2) - \mathbb{P}(\phi_1' \in E_1, \phi_2' \in E_2)| \le 2\mathbb{P}((\phi_1,\phi_2) \in \Omega_N^c)\,.
\end{align}

We are now left with showing that for almost every $\alpha$, the probability that $\mathcal S_B$ connects  $B$ to the boundary $\partial B_N$ of the box $[-N,N]^d$ will go to $0$ as $N\to\infty$. 
 This will be divided into two parts. The first is to show it for any pair $\pi_{1}$ and $\pi_{2}$ that satisfies the good pair property for all connecting horizontal clusters between $B$ and $\partial B_N$. The second is to show that almost all pairs $(\pi_{1,\alpha},\pi_{2,\alpha})$ will be good pairs for all connecting horizontal clusters for $N$ large enough.

\begin{lem} \label{lem:supercluster}
Suppose $\beta_0$ is large and a pair of extremal measures $\pi_1$ and $\pi_2$ are a good pair for all connecting horizontal clusters between $B$ and $\partial B_N$. Then, the probability that there exists a supercluster that connects $B\subset B_r$ to  $\partial B_N$ can be bounded as follows: 
\begin{equation} \label{eq:suppercluster}
\pi_1 \otimes \pi_2(\mathcal S_B \cap \partial B_N \ne \emptyset ) \le |B|\exp(-c' \beta_0 (N-r))\,,
\end{equation}
for some universal constant $c'>0$. 
\end{lem}
\begin{proof}

 Notice that if $\mathcal S_B$ intersects $\partial B_N$, then there exists  a connecting horizontal cluster between $B$ and $\partial B_N$. 
If we apply the assumption of being a good pair for all connecting horizontal  clusters to $\pi_{1}$ and $\pi_{2}$, we see that the probability that a specific collection of connecting horizontal clusters $\mathcal{HC} = (HC_1,a_1), (HC_2,a_2),\ldots, (HC_m,a_m)$  is excited is bounded by  $\exp(-(c\beta_0/4) \sum_{i=1}^m |\partial_o(HC_i)|)$. 
In addition, in order to connect from $B$ to $\partial B_N$, the total sum of the outer boundary lengths of these horizontal clusters  must be greater than $N-r$  since $B \subset B_r$. 

It remains to enumerate over the possible connecting horizontal  clusters from $B$ to $\partial B_N$. This can be upper bounded by counting the number of paths of horizontal clusters that have total outer boundary length equal to  $M \ge N-r$ and intersect $B$ times $\exp(-c \beta_0 M/4)$. 
In total, we get 
\begin{align}\label{eq:disagreementpercbound}
\pi_1 \otimes \pi_2(\mathcal S_B \cap \partial B_N \ne \emptyset) &\le |B|\sum_{K=1}^{\infty} \sum_{t_1 + \ldots + t_K \ge N-r, t_i \ge 1 \forall i}  \prod_{i=1}^K\exp(-c\beta_0 t_i/4) C^{t_i} 4d|t_i|^2  
\end{align}
for some constant $C$. 
We derived~\eqref{eq:disagreementpercbound} bound as follows. The first sum is over the number of horizontal clusters used, and the second over the assignment of boundary costs $|\partial_o(HC_i)|=t_i$ to each of them. Then, for each $i$, the horizontal  clusters $\tilde{\mathcal{HC}}_i$ must contain a point $p_i$ in its interior such that $p_i$ is adjacent to some point in the horizontal  cluster $\tilde{\mathcal{HC}}_{i-1}$. There are $|B|$ many choices for the point $p_1$. Given the horizontal  cluster $\tilde{\mathcal{HC}_i}$, there are at most $2d|t_i|^2$ choices for the next center $p_{i+1}$ of $\tilde{\mathcal{HC}}_{i+1}$, and $2$ choices for $a_{i+1}$. 

Given a point $p_i$, there is some absolute constant $C$ such that there are at most $C^{t_i}$ choices of horizontal clusters of outer boundary size $t_i$ that contain the point $p_i$ in their interior. Finally, as we have mentioned before, once we have specified all of the horizontal clusters, the probability of seeing them is at most $\exp(-c  \beta_0 \sum_{i=1}^K t_i/4)$. 

So long as $\beta_0$ is large enough, one has $\exp( - c\beta_0 t_i) C^{t_i} 4 d|t_i|^2\le \exp( - c'\beta_0 t_i)$ for some other absolute constant $c'$ and plugging this into \eqref{eq:disagreementpercbound}, we see that,
\begin{equation} \label{eq:clusterexpanbnd}
\begin{aligned}
\frac{1}{|B|} \pi_1 \otimes \pi_2(\mathcal S_B \cap \partial B_N \ne \emptyset) 
& \le \sum_{A= N-r}^{\infty} 2^A \exp(-c'  \beta_0 A) +\frac{\left(\sum_{t_i=1}^{\infty} \exp(- c'  \beta_0 t_i)\right)^{N-r}  }{1-\sum_{t_i=1}^{\infty} \exp(c'  \beta_0 t_i) }\,.
\end{aligned}
\end{equation}
Again, provided $ \beta_0$ is large enough, we can find some $c''$ such that the last line will be less than $\exp(-c''  \beta_0 (N-r))$. This completes our proof. 
\end{proof}

We now turn to showing that almost every pair of extremal components will satisfy the property of  being a good pair with respect to all connecting horizontal  clusters between $B$ and $\partial B_N$.

\begin{lem} \label{lem:notgoodpair}
Let $\mu_1$ and $\mu_2$ be two measures satisfying densely dominant heights, and with their respective extremal decompositions. 
Then, for $\beta_0$ large enough, the total measure of all $\alpha$ such that $(\pi_{1,\alpha},\pi_{2,\alpha})$ is not a good pair for some connecting horizontal clusters between $B$ and $\partial B_N$ is less than $|B| \exp[-c'' \beta_0 (N-r)]$.
\end{lem}

\begin{proof}
Let $\mathcal{HC}$ be a collection of horizontal connecting  clusters  between $B$ and $\partial B_N$. Let $\mathcal{HC}_1$ denote its components with $a_i =1$ and let $\mathcal{HC}_2$ denote its components with $a_i =2$. Notice that $\mu_1 \otimes \mu_2(\mathcal{HC}) = \mu_1(\mathcal{HC}_1) \mu_2(\mathcal{HC}_2) \le \exp[- c \beta_0 \sum_{i=1}^m |\partial_o(HC_i)|]$, where we applied Lemma~\ref{lem:ergodic-minimal-has-dominant-height} and~\eqref{eq:exp-suppressed-gradients-condition}. Without loss of generality, assume that $\mu_1(\mathcal{HC}_1) \le \exp[-(c/2) \beta_0 \sum_{i=1}^{m} |\partial_o(HC_i)| ]$. Notice, then that we can say,
\begin{equation}
 \int_0^1 \pi_{1,\alpha} \otimes \pi_{2,\alpha}(\mathcal{HC}) \nu(\text{d}\alpha) \le \int_0^1 \pi_{1,\alpha}(\mathcal{HC}_1) \nu(\text{d}\alpha) = \mu_1(\mathcal{HC}_1) \le \exp\Big(-\frac{c\beta_0}{2}  \sum_{i=1}^m |\partial_o(HC_i)|\Big)\,.
\end{equation}
(understanding the integral as against the counting measure in the case where $\mathfrak{s}$ is rational, for brevity.)

Now by applying Markov's inequality, we have that the measure of $\alpha\in [0,1]$ such that $\pi_{1,\alpha} \otimes \pi_{2,\alpha}(\mathcal{HC}) \ge \exp[- (c/4) \beta_0 \sum_{i=1}^m |\partial_o(HC_i)|]$ must be less than $\exp[-(c/4) \beta_0 \sum_{i=1}^m |\partial_o(HC_i)|]$.

We can now apply the same logic of the union bound leading up to equation \eqref{eq:clusterexpanbnd} to say that the total measure of $\alpha\in [0,1]$ such that the pair $\pi_{1,\alpha} \otimes \pi_{2,\alpha}$ is not a good pair for some horizontal clusters connecting between $B$ and $\partial B_N$ is less than $|B| \exp[-c'' \beta_0 (N-r)]$ provided $\beta_0$ is large enough, for another absolute constant $c''$. 
\end{proof}

To conclude the proof, by Lemma~\ref{lem:notgoodpair} and Lemma~\ref{lem:supercluster}, the measure of $\alpha$ for which $\pi_{1,\alpha}\otimes \pi_{2,\alpha}(\mathcal S_B \cap \partial B_N)\ge |B|e^{ - c' \beta_0 (N-r)}$ is $o_N(1)$. Fixing $B \subset B_r$, using this bound on the right-hand side of~\eqref{eq:tv-bound-by-supercluster-connection} and sending $N\to\infty$ implies by continuity of measure that for a full-measure set of $\alpha$, $\pi_{1,\alpha}$ and $\pi_{2,\alpha}$ have the same marginals on $B$. Since this holds for all $r$, they are the same measure.  
\end{proof}

As a corollary of the same proof strategy, we also can obtain exponential decay of correlations on the extremal components. 

\begin{cor}\label{cor:doc}
 Let $\mu_{\mathfrak{s}}$ be a Gibbs measure with ergodic minimal gradient, and write its extremal decomposition in the form $\int_0^1 \pi_{\alpha} \text{d}\alpha$ (again understanding this integral as against the counting measure when $\mathfrak{s}$ is rational). Fix any two finite domains $\Lambda_1$ and $\Lambda_2$ and let $f_1$, $f_2$ be any bounded functions from $\mathbb Z^{\Lambda_i} \to \mathbb{R}$. There is a universal constant $c$ such that the measure of $\alpha\in [0,1]$ such that the equation
$$
|\pi_\alpha[f_1(\phi) f_2(\phi)] - \pi_\alpha[f_1(\phi)] \pi_\alpha[f_2(\phi)]| \le \|f_1\|_{\infty} \|f_2\|_{\infty} |\Lambda_1| \exp(-c  \beta_0  d(\Lambda_1,\Lambda_2))\,,
$$ 
is violated has measure at most $|\Lambda_1|\exp(-c  \beta_0 d(\Lambda_1,\Lambda_2))$. 
\end{cor}

\begin{proof}
One can first rewrite
$$
|\pi_\alpha[f_1(\phi) f_2(\phi)] - \pi_\alpha[f_1(\phi)] \pi_\alpha[f_2(\phi)]| =|\pi_1[f_1(\phi) f_2(\phi)] \pi_2 [1]- \pi_1[f_1(\phi)] \pi_2[f_2(\phi)]|,
$$
where $\pi_1$ and $\pi_2$ are two independent copies of $\pi_\alpha$. Then one can apply the swapping argument from the previous proof of Theorem \ref{thm:disagreementperc} to show that there is cancellation as long as there is no large supercluster connecting $\Lambda_1$ to $\Lambda_2$. The probability of having such a supercluster is bounded by the probability of having a supercluster that goes from $\Lambda_1$ out a distance $d(\Lambda_1,\Lambda_2)$. After translating to the origin, this is bounded by $|\Lambda_1| \exp( - c'\beta_0 d(\Lambda_1,\Lambda_2))$ per Lemma~\ref{lem:supercluster} so long as $\alpha$ is such that $\pi_1,\pi_2$  are good pairs for all connecting horizontal clusters between $\Lambda_1$ and $\Lambda_2$. Lemma~\ref{lem:notgoodpair} in turn bounds the measure of those $\pi_1,\pi_2$ that are not good pairs by $|\Lambda_1| \exp( - c''\beta_0 d(\Lambda_1,\Lambda_2))$.  
\end{proof}

\subsection{Proofs of main theorems} 
In this section, we combine all of the above to conclude our three main theorems. 

\begin{proof}[\textbf{\emph{Proof of Theorem~\ref{thm:main-localization}}}]
        The fact that the slope $\mathfrak{s}$ is localized (i.e., there exists a height function Gibbs measure of slope $\mathfrak{s}$) was shown in Lemma~\ref{lem:existence-of-height-measures}. The uniqueness of ergodic minimal gradient measures of slope $\mathfrak{s}$ was shown in Theorem~\ref{thm:disagreementperc} (recalling Theorem 8.6.3 of~\cite{Sheffield-Random-Surfaces} says if $\mathfrak{s}$ is localized, any such gradient measure must be the gradient of a height function Gibbs measure). 
\end{proof}

\begin{proof}[\textbf{\emph{Proof of Theorem \ref{thm:structural results}}}]
Item (1) of Theorem~\ref{thm:structural results} was shown in Lemma~\ref{lem:ergodic-minimal-has-dominant-height}. Corollary~\ref{cor:dominheight} together with the decomposition property of Theorem~\ref{thm:Sheffield} proves part (2) of Theorem \ref{thm:structural results} in the rational case. Corollary~\ref{cor:fincor} together with the decomposition property of Theorem~\ref{thm:Sheffield} proves it in the irrational case. Finally, the uniqueness of Theorem \ref{thm:disagreementperc}, together with the decomposition of Theorem~\ref{thm:Sheffield} proves part (3).
\end{proof}

\begin{proof}[\textbf{\emph{Proof of Theorem \ref{thm:structural results-exponential-decay}}}]

The first item follows from Corollary~\ref{cor:one-point-exponential-tails}, Corollary~\ref{cor:dominant-height-exists}. The second item follows from integrating Corollary~\ref{cor:doc} in $\alpha$ and using boundedness of $f_1,f_2$. 
\end{proof}

\bibliographystyle{plain}
 \bibliography{slopes}

\end{document}